\documentclass[reqno]{amsart}
\usepackage{amsfonts,amsmath,amssymb}

\numberwithin{equation}{section}

\def\be{{\beta}}

\def\al{\alpha}

\newtheorem{theorem}{Theorem}[section]
\newtheorem{lemma}[theorem]{Lemma}

\newtheorem{proposition}[theorem]{Proposition}

\newtheorem{remark}[theorem]{Remark}

\begin{document}

\title{The Euler--Poisson system in $2D$: global stability of the constant equilibrium solution}

\author{Alexandru D. Ionescu}
\address{Princeton University}
\email{aionescu@math.princeton.edu}

\author{Benoit Pausader}
\address{Courant Institute}
\email{pausader@cims.nyu.edu}

\thanks{The first author was partially supported by a Packard Fellowship and NSF grant DMS-1065710. The second author was partially supported by NSF grant DMS-1142293.}

\begin{abstract}

We consider the (repulsive) Euler-Poisson system for the electrons in two dimensions and prove that small smooth perturbations of a constant background exist for all time and remain smooth (never develop shocks). This extends to $2D$ the work of Guo \cite{Guo}.
\end{abstract}
\maketitle
\tableofcontents

\section{Introduction}\label{intro}

In this paper we investigate the question of global existence for small perturbations of a constant background for the 
following {\it Euler--Poisson} system for the electrons in $2$ dimensions:
\begin{equation}\label{EP}
\begin{split}
\partial _{t}n_{-}+\nabla \cdot \left( n_{-}v\,_{-}\right) & =0 \\
n_{-}m_{-}(\partial _{t}v_{-}+v_{-}\cdot \nabla v_{-})+\nabla p(n_{-})&
=en_{-}\nabla \phi  \\
\Delta \phi & =4\pi e(n_{-}-n_{0}).
\end{split}
\end{equation}
Here the unknowns are $n_-\ge 0$, the density of electrons, and $v_-\in\mathbb{R}^2$, the velocity field of the electrons. These are functions defined for $(t,x)\in\mathbb{R}\times\mathbb{R}^2$. 
The positive constants $m_-$, $e$ and $n_0$ correspond respectively to the mass of an electron, its charge and the average charge of an ion background. Finally, $p=p(n_-)$ is a pressure function, given by a constitutive relation which for simplicity we assume to be quadratic.

These equations model the behavior of a fluid of electrons in a warm adiabatic fully ionized plasma when the magnetic field and the motion of the ions is neglected. Neglecting the magnetic field is customary and reduces the number of unknowns. Neglecting the ion motion is relevant since the ratio of the masses of the electrons and the ions is typically very small\footnote{It is no bigger than the ratio of the electron mass to the proton mass which equals $1/1836$.}.
We refer to \cite{Bit} for more on the physical background.

Equations \eqref{EP} represent a coupling of a compressible (inviscid) fluid with an electrostatic field. 
For the pure compressible Euler equation, even small and smooth initial perturbations of a constant equilibrium can lead to formation of shocks in finite time \cite{Sid}.

However, here we show that the coupling with the self-consistent electric field stabilizes the system in the sense that small smooth perturbations of a constant background remain global and return to equilibrium. 
This phenomenon was first observed in $3D$ by Guo \cite{Guo}. We also refer to Guo-Pausader \cite{GuoPau} for a similar result for the ion equation and to the recent work of Germain-Masmoudi \cite{GerMas} for the complete Euler-Maxwell equation for the electrons in $3D$. 
On the other hand, large perturbations can lead to blow-up in finite time for the Euler-Poisson equation \cite{GuoTad}.

Previous work on the Euler-Poisson system in $2D$ also includes the work of Jang \cite{Juhi} and Wei-Tadmore-Bae \cite{WTB} for radial data, and Jang-Li-Zhang \cite{JangLiZhang} for the existence of wave operators.

\medskip

For simplicity, we assume that the pressure law is quadratic $p(n_-)=T_-(n_-)^2/2$. The purpose of this is only to minimize the number of terms in the nonlinearity, but other powers could be treated similarly.
After rescaling, we can then reduce to the following system\footnote{We can even reduce further to $a=b=1$; however, we have prefered to keep the constants $a$ and $b$ for their own physical interest.}
\begin{equation}\label{EP1}
\begin{split}
\partial _{t}n+\nabla \cdot \left( nv\right)&=0,\\
\partial _{t}v+v\cdot \nabla v+a\nabla n&=b\nabla \phi,\\
\Delta \phi&=n-1,
\end{split}
\end{equation}
where $a,b\in(0,\infty)$, $a=T_-n_0/m_-$ is the square of the speed of sound and $b=4\pi e^2n_0/m_-$ is the electron plasma frequency. Our main Theorem asserts that small, neutral irrotational perturbations of a constant equilibrium $(n,v)=(1,0)$ are global.

\begin{theorem}\label{MainThm}
There exists $\varepsilon_0>0$ and a norm $\Vert \cdot\Vert_Y$ such that any initial data $(n(x,0),v(x,0))$ satisfying
\begin{equation*}
\Vert (n(\cdot,0)-1,v(\cdot,0))\Vert_Y\le\varepsilon_0,\quad \hbox{curl}_x(v(\cdot,0))=0
\end{equation*}
leads to a global solution $(n,v)$ of \eqref{EP1} which returns to equilibrium in the sense that
\begin{equation*}
\Vert n(\cdot,t)-1\Vert_{L^\infty}+\Vert\nabla v(\cdot,t)\Vert_{L^\infty}\lesssim(1+t)^{-1}.
\end{equation*}
\end{theorem}

\begin{remark}
The precise nature of the norm is given in \eqref{sec4} below. It controls a finite number of derivatives and  requires localization of the initial perturbations.
Its most notable feature is that finiteness of the norm implies that the perturbation is electrically neutral:
\begin{equation*}
\int_{\mathbb{R}^2}(n(x,0)-1)dx=0,
\end{equation*}
which is then conserved by the flow.
This condition is also necessary in order to ensure finiteness of the physical (conserved) energy
\begin{equation}\label{Energy}
E=\frac{1}{2}\int_{\mathbb{R}^2}\left[n\vert v\vert^2+a (n-1)^2+b\vert\nabla\phi\vert^2\right]dx,
\end{equation}
hence we find it an acceptable assumption.

The irrotationality assumption is also propagated by the flow. It removes a component in the system that is only transported and does not obey nice decay estimates.
\end{remark}

Theorem \ref{MainThm} follows from the more precise Theorem \ref{MainThm2} below. In order to attack this problem, we restate it as a quasilinear dispersive equation. The main difficulties then come from the slow decay of the norms, the quadratic power of the nonlinearity and some nonlocal terms in the nonlinearity (Riesz transforms) that prevent good localization of the small frequencies and convenient use of the only almost integrable norm ($L^\infty$).

The first two problems are classical in the study of quasilinear dispersive systems and several methods have been developed to address these difficulties, including normal form transformations \cite{Sha} and commuting vector fields \cite{Kla}. More recently refinements and new developments from Gustafson-Nakanishi-Tsai \cite{GNT2,GNT} and Germain-Masmoudi-Shatah \cite{GerMas,GerMasSha,GerMasSha2} have led to progress in dealing with various physical problems. Our analysis is in the framework of this general scheme, and is especially close to the analysis of the water-wave problem in \cite{GerMasSha2}.

The proof relies on two distinct components. The first component, {\it{the energy method}}, exploits the existence of a conserved physical energy given in \eqref{Energy}, which comes from the subtle structure of the nonlinearity. This implies almost conservation of higher order energies, up to commutators that are {\it lower order} in derivative and are at least of {\it cubic order} and therefore can be controlled if the solution decays sufficiently fast. In dimension $3$, these error terms are integrable and the higher order energies remain bounded (see e.g. \cite{Guo,Sha}). In dimension $2$, all the norms are only almost integrable and the highest order energy is allowed to grow slowly.

The second component of the proof is concerned with proving the decay estimate needed to control the commutators above. It exploits the fact that we consider a perturbation of a {\it constant} solution, which allows us to use the {\it{Fourier transform method}}. The analysis here depends on the structure of the linearized problem (in particular the {\it dispersion relation} $\omega(k)$) and on delicate cancellation properties of the nonlinearity (the {\it{null structure}}), which are particularly important to bound the constribution of very low frequencies. We first integrate the quadratic terms using a normal form transformation\footnote{This transformation, introduced in \cite{Sha}, is always possible if the phase velocity $\omega(k)/\vert k\vert$ is decreasing and $\omega(k)\ge 0$ is increasing.}. We are then left with several cubic terms which oscillate along different phases. We account for the elliptic phases (which never vanish) using another iteration of the normal form transformation. This produces quartic terms which can be easily estimated. We account for the remaining non-elliptic phase by exploiting its more ``hyperbolic'' nature: it is somewhat local in space and ``commutes\footnote{in the sense that, as for ``commuting vector fields'', it obeys a variant of Leibniz rule: $xT[f,g,h]=T_1[xf,g,h]+T_2[f,xg,h]+T_3[f,g,xh]$ where $T_1$, $T_2$ and $T_3$ have similar properties as $T$}'' at first order with the position operator $f\mapsto xf$.

As mentioned before, this is similar to the approach used in \cite{GerMasSha2} for the water-wave problem, which has a similar structure. However, we introduce some different ideas, such as more involved space--frequency analysis, that we hope can find further applications.

First of all, we only work on the linear profile and derive all our estimates from information about it. This was already done in \cite{GNT} in $3D$ but encounters difficulties in $2D$ due to the fact that we need an endpoint dispersion estimate ($L^\infty$). This forces us to localize in space using the resolution of unity associated with the position operator in order to create a norm having the nice properties of $L^2$ but which is stronger than the $L^1$-norm. Controlling this norm leads us to make most of our analysis in the physical space rather than in the Fourier space. We believe this gives a clearer picture and accounts better for the finite speed of propagation inherent in dispersive equations\footnote{Note that this finite speed of propagation is only for the linear flow. The nonlinear flow contains nonlocal operators (Riesz transforms) and does not satisfy exact analogues of the finite speed of propagation for solutions of Klein-Gordon equation, except in the radial case as pointed out in \cite{Juhi}.}.

In particular, once we introduce localization in physical space, the main problem becomes to bound uniformly the linear profile on a spatial dyadic ring in $xL^2$. Assume to simplify that the initial data is only concentrated in the ball of radius $1$ and that we are looking at a distance $R>>1$. The bound follows from three distinct ingredients:
\begin{enumerate}
\item First, by variation on the finite speed of propagation principle, one can see that, for times $t\lesssim R$, the solution at $R$ is only strongly influenced by the solution at nearby locations (at distance smaller than $R/2$). Since it starts small and it interacts nonlinearly, it remains so.
\item Second, by decay property of the solutions, all the interactions after a time $t>R^{1+\delta}$ add up in $L^2$ to a size smaller than $R^{1-\delta/100}$ and thus are acceptable.
\item Finally, one only needs to consider a small portion of space-time when $R\sim x$, $R\le t\le R^{1+\delta}$. This is of course the main interaction region. In this region, various decay estimates (in time or in space) become equivalent which gives more flexibility. Since for cubic nonlinearities, one only need a small improvement on the decay estimates, this region can be controlled by exploiting the ``hyperbolic'' structure of the phase: it cannot be stationary in all directions. 
\end{enumerate}

Our strategy is as follows: In Section \ref{section2}, we reformulate \eqref{EP} into a quasilinear scalar dispersive equation for a complex-valued unknown, \eqref{norm1} by diagonalizing the linearized system. We then prove our main theorem about the new equation, Theorem \ref{MainThm2} assuming two propositions. This implies Theorem \ref{MainThm} In Section \ref{LocalExistence}, we prove our first proposition, Proposition \ref{LocExLemma} which gives a good local existence theory with energy estimates suitable for our analysis. In Section \ref{Bootstrap} we prove the second proposition, Proposition \ref{Norm}, which implies the decay of solutions needed for the energy estimates. This is done by first using a normal form transformation and then following the strategy explained in points $(1)$, $(2)$ and $(3)$ above. Finally in Section \ref{technical}, we collect various technical estimates needed in the analysis.

\section{Main definitions and propositions}\label{section2}

We first remark that since we start with irrotational initial velocity, the velocity remains irrotational, since $\omega=\hbox{curl}(v)$ satisfies a well-known transport equation. Hence we may assume that $v=\nabla h$ for some velocity potential $h$.
We also set $n=1+\rho$ and  \eqref{EP1} becomes
\begin{equation}\label{EP2}
\begin{split}
\partial _{t}\rho+\Delta h+\partial_\al(\rho\partial_\al h)&=0,\\
\partial _{t}h+(1/2)(\partial_\al h\partial_\al h)+a\rho&=b\phi.\\
\Delta \phi&=\rho.
\end{split}
\end{equation}

Let $|\nabla|$ denote the operator on $\mathbb{R}^2$ defined by the Fourier multiplier $\xi\mapsto |\xi|$, and let $g:=|\nabla|^{-1}\rho$. In terms of $g,h$ the system \eqref{EP2} becomes
\begin{equation}\label{EP3}
\begin{split}
&\partial_tg-|\nabla|h+|\nabla|^{-1}\partial_j(|\nabla|g\cdot\partial_j h)=0,\\
&\partial_th+(a|\nabla|+b|\nabla|^{-1})g+(1/2)(\partial_j h\partial_j h)=0.
\end{split}
\end{equation}
Letting
\begin{equation*}
U:=\sqrt{a\vert\nabla\vert^2+b} g+i|\nabla|h,
\end{equation*}
we derive the equation
\begin{equation}\label{norm1}
(\partial_t+i\Lambda)U=\frac{i}{4}\sum_{j=1}^2R_j\Lambda\big[|\nabla|\Lambda^{-1}(U+\overline{U})\cdot R_j(U-\overline{U})\big]+\frac{i}{8}\sum_{j=1}^2|\nabla|\big[R_j(U-\overline{U})\cdot R_j(U-\overline{U})\big],
\end{equation}
where
\begin{equation*}
\Lambda:=\sqrt{a|\nabla|^2+b},\qquad R_j:=\partial_j/|\nabla|.
\end{equation*}

We fix $\varphi:\mathbb{R}\to[0,1]$ an even smooth function supported in $[-8/5,8/5]$ and equal to $1$ in $[-5/4,5/4]$. Let
\begin{equation*}
\varphi_k(x):=\varphi(|x|/2^k)-\varphi(|x|/2^{k-1})\text{ for any }k\in\mathbb{Z},\,x\in\mathbb{R}^2,\qquad \varphi_I:=\sum_{m\in I\cap\mathbb{Z}}\varphi_m\text{ for any }I\subseteq\mathbb{R}.
\end{equation*}
Let $P_k$, $k\in\mathbb{Z}$, denote the operator on $\mathbb{R}^2$ defined by the Fourier multiplier $\xi\to \varphi_k(\xi)$. Similarly, for any $I\subseteq \mathbb{R}$ let $P_I$ denote the operator on $\mathbb{R}^2$ defined by the Fourier multiplier $\xi\to \varphi_I(\xi)$. For $T\geq 1$ and integers $1\leq N_0<N$ we define
\begin{equation}\label{sec4}
\begin{split}
X^N_T(\mathbb{R}^2\times[0,T]):=\{&f\in C([0,T]:H^N(\mathbb{R}^2)):e^{it\Lambda}f\in C([0,T]:Y^{N_0}(\mathbb{R}^2))\\
&\text{ and }\|f\|_{X^N_T}:=\sup_{t\in[0,T]}(1+t)^{-\delta}\|f(t)\|_{H^N}+\sup_{t\in[0,T]}\|e^{it\Lambda}f(t)\|_{Y^{N_0}}<\infty\},
\end{split}
\end{equation}
where
\begin{equation}\label{sec5}
\begin{split}
&Y^{N_0}(\mathbb{R}^2):=\{\phi\in L^2(\mathbb{R}^2):\|\phi\|_{Y^{N_0}}:=\|\phi\|_{H^{N_0}}+\|\phi\|_Z<\infty\},\\
&\|\phi\|_Z:=\sup_{k\in\mathbb{Z}}(2^{k/10}+2^{10k})\Big[\|P_k\phi\|_{L^2}+\sum_{j\in\mathbb{Z}_+}2^j\|\varphi_j(x)\cdot P_k\phi(x)\|_{L^2}\Big].
\end{split}
\end{equation}

The spaces $X^N_T$ are our main spaces, and we use them to control our "smooth" solutions. 
First we control the high energy norm $H^N$ which is allowed to grow slowly in time; this growth appears to be necessary in $2$ dimensions, due to the 
 non-integrable factor $(1+t)^{-1}$ in the dispersive bound \eqref{keydisperse} below, but is not necessary in $3$ dimensions (see \cite{Guo}). 

We also control an intermediate energy norm $H^{N_0}$, for some $N_0$ chosen smaller than $N$, uniformly in time. 
This intermediate norm is mostly for convenience and can be removed.

Finally, we control the key $Z$ norm described in \eqref{sec5}, which captures the dispersive nature of our flow. 
Variants of this norm are of course possible, see for example the similar norms used in \cite{GerMas,GerMasSha,GerMasSha2}. 
We make the specific choice described in \eqref{sec5} in order to achieve two basic inequalities,
\begin{equation}\label{keychoice}
\begin{split}
&\|R_1f\|_{Z}+\|R_2f\|_{Z}\lesssim \|f\|_{Z},\\
&\sup_{k\in\mathbb{Z}}(2^{k/10}+2^{10k})\|P_k f\|_{L^1}\lesssim \|f\|_{Z}. 
\end{split}
\end{equation}
The first inequality in \eqref{keychoice} is a consequence of the proof of Lemma \ref{tech1} and is necessary because our nonlinearity contains several Riesz transforms, see \eqref{norm1}. 
The second inequality in \eqref{keychoice} is an easy consequence of the definition and explains, in particular, the choice of the $l^1$ sum in $j$. It can be combined with the basic dispersive estimate  
\begin{equation}\label{keydisperse}
 \|P_le^{-it\Lambda}f\|_{L^\infty}\lesssim (1+t)^{-1}(1+2^{2l})\|f\|_{L^1},\qquad t\in\mathbb{R},
\end{equation}
to show that
\begin{equation}\label{keydisperse2}
\|P_ke^{-it\Lambda}f\|_{L^\infty}\lesssim (1+t)^{-1}(2^{k/10}+2^{8k})^{-1}\|f\|_{Z},\qquad t\in\mathbb{R},k\in\mathbb{Z}.
\end{equation}

We can now state our main theorem in terms of the function $U$. 

\begin{theorem}\label{MainThm2}
 Assume $N=30$, $N_0=20$, and $\delta=1/100$. There is $\overline{\varepsilon}$ sufficiently small such that if 
\begin{equation*}
 \|U_0\|_{H^N}+\|U_0\|_{Y^{N_0}}\leq\overline{\varepsilon}
\end{equation*}
then there is a unique global solution $U\in C([0,\infty):H^N)$ of the initial-value problem
\begin{equation}\label{invp}
\begin{cases}
&(\partial_t+i\Lambda)U=\frac{i}{4}\sum_{j=1}^2\Lambda R_j\big[|\nabla|\Lambda^{-1}(U+\overline{U})\cdot R_j(U-\overline{U})\big]+\frac{i}{8}\sum_{j=1}^2|\nabla|\big[R_j(U-\overline{U})\cdot R_j(U-\overline{U})\big],\\
&U(0)=U_0.
\end{cases}
\end{equation}
In addition, $U\in X^N_T$ for any $T\geq 1$ and
\begin{equation}\label{invp2}
 \sup_{t\in[0,\infty)}\big[(1+t)^{-\delta}\|e^{it\Lambda}U(t)\|_{H^N}+\|e^{it\Lambda}U(t)\|_{Y^{N_0}}\big]\lesssim \|U_0\|_{H^N}+\|U_0\|_{Y^{N_0}}.
\end{equation}
\end{theorem}

It is easy to see that Theorem \ref{MainThm2} implies Theorem \ref{MainThm}, since
\begin{equation*}
 n-1=\Lambda^{-1}|\nabla|(\Re U),\qquad v_1=R_1(\Im U),\qquad v_2=R_2(\Im U).
\end{equation*}
On the other hand, Theorem \ref{MainThm2} is a consequence of Proposition \ref{LocExLemma} and Proposition \ref{Norm} below. We start with the local existence theory:

\begin{proposition}\label{LocExLemma}
(i) There exists $\delta_0>0$ such that for any $U_0\in H^3$ satisfying
\begin{equation}\label{SmallNess}
\Vert U_0\Vert_{H^3}\le \delta_0,
\end{equation}
there exists a unique solution $U\in C([0,1]:H^3)$ of \eqref{invp} such that $U(0)=U_0$ and
\begin{equation*}
 \sup_{t\in[0,1]}\|U(t)\|_{H^3}\lesssim \Vert U_0\Vert_{H^3}.
\end{equation*}

(ii) Assume in addition that $U_0\in H^M$, $M\in[3,40]\cap\mathbb{Z}$. Then $U\in C([0,1]:H^M)$ and for any $0\le s\le t\le 1$,
\begin{equation}\label{EGrowth}
E_M(t)\le E_M(s)+C\int_s^t\|U(t')\|_{Z'}\cdot E_M(t^\prime)dt^\prime
\end{equation}
where the $M$-th order energy, defined in \eqref{DefE} below, satisfies $E_M(t)\simeq \Vert U(t)\Vert_{H^M}^2$ uniformly in time and
\begin{equation}\label{Z'norm}
 \|f\|_{Z'}:=\sup_{k\in\mathbb{Z}}(2^{k/2}+2^{2k})\|P_k f\|_{L^\infty}.
\end{equation}

(iii) With $N_0$ as in Theorem \ref{MainThm2}, assume $T\geq 1$ and $U\in C([0,T]:H^{N_0})$ is a solution of \eqref{invp} with the property that $U_0\in Y^{N_0}$. Then 
\begin{equation*}
\begin{split}
&\sup_{t\in[0,T]}\|e^{it\Lambda}U(t)\|_{Y^{N_0}}\leq C\big(\|U_0\|_{Y^{N_0}},\sup_{t\in[0,T]}\|U(t)\|_{H^{N_0}},T\big),\\
&\lim_{t'\to t}\|e^{it'\Lambda}U(t')-e^{it\Lambda}U(t)\|_{Y^{N_0}}=0,\qquad\text{ for any }t\in[0,T].
\end{split}
\end{equation*}
\end{proposition}

The main conclusion in the proposition is the energy inequality \eqref{EGrowth}, which depends on the $Z'$ norm defined in \eqref{Z'norm}. This norm has to be chosen strong enough to allow for the energy inequality \eqref{EGrowth} to hold; in particular
\begin{equation*}
\|\partial_iR_j f\|_{L^\infty}+\|R_j f\|_{Z'}\lesssim \|f\|_{Z'},\qquad j,k\in\{1,2\}.
\end{equation*}
On the other hand, the $Z'$ norm has to be chosen small enough, in such a way that the function $t\to\|U(t)\|_{Z'}$ is almost integrable in time; the main inequality we need is
\begin{equation}\label{basic}
 \|e^{-it\Lambda}f\|_{Z'}\lesssim (1+t)^{-1}\|f\|_{Z},\qquad \text{ for any }t\in[0,\infty),
\end{equation}
which is an easy consequence of \eqref{keydisperse2}.

To prove the global result in Theorem \ref{MainThm2} we also need the following bootstrap estimate:

\begin{proposition}\label{Norm}
Assume that $U\in X^N_T$ is a solution of the equation
\begin{equation*}
(\partial_t+i\Lambda)U=\frac{i}{4}\sum_{j=1}^2\Lambda R_j\big[|\nabla|\Lambda^{-1}(U+\overline{U})\cdot R_j(U-\overline{U})\big]+\frac{i}{8}\sum_{j=1}^2|\nabla|\big[R_j(U-\overline{U})\cdot R_j(U-\overline{U})\big]
\end{equation*}
with the property that
\begin{equation}\label{sec2}
\|U\|_{X^N_T}\leq\varepsilon_0\leq 1.
\end{equation}
Then
\begin{equation}\label{sec3}
\sup_{t\in[0,T]}\|e^{it\Lambda}U(t)-U(0)\|_{Y^{N_0}}\lesssim \varepsilon^2_0.
\end{equation}
\end{proposition}

\subsection{Proof of Theorem \ref{MainThm2}}\label{Uproof}
Theorem \ref{MainThm2} is a consequence of Proposition \ref{LocExLemma} and Proposition \ref{Norm}.
Indeed, assume that we start with data $U_0$ satisfying $\|U_0\|_{H^N}+\|U_0\|_{Y^{N_0}}=\epsilon\leq\overline{\varepsilon}$, for some $\overline{\varepsilon}$
sufficiently small relative to the value of $\delta_0$ in Proposition \ref{LocExLemma} (i). Assume that we have constructed a solution $U\in C([0,T]:H^N)$
 (using Proposition \ref{LocExLemma} (ii)), for some $T\geq 1$, such that
\begin{equation}\label{cons1}
 \sup_{t\in[0,T]}\|e^{it\Lambda}U(t)\|_{Y^{N_0}}\leq \epsilon^{3/4}.
\end{equation}
Then, using \eqref{basic},
\begin{equation*}
 \|U(t)\|_{Z'}\lesssim (1+t)^{-1}\epsilon^{3/4},\qquad\text{ for any }t\in[0,T].
\end{equation*}
Then, using \eqref{EGrowth},
\begin{equation*}
E_N(t)-E_N(0)\leq C\epsilon^{3/4}\int_0^t(1+s)^{-1}E_N(s)\,ds,\qquad t\in[0,T],
\end{equation*}
which shows that $E_N(t)\leq E_N(0)(1+t)^\delta$, for any $t\in[0,T]$. Therefore $U\in X^N_T$ and
\begin{equation*}
 \|U\|_{X^N_T}\leq 2\epsilon^{3/4}.
\end{equation*}
We can apply now Proposition \ref{Norm} to conclude that
\begin{equation*}
 \sup_{t\in[0,T]}\|e^{it\Lambda}U(t)-U(0)\|_{Y^{N_0}}\lesssim \epsilon^{3/2}.
\end{equation*}
In other words, if $U$ satisfies $\sup_{t\in[0,T]}\|e^{it\Lambda}U(t)\|_{Y^{N_0}}\leq \epsilon^{3/4}$ as in \eqref{cons1} then $U$ satisfies the stronger inequality  
$\sup_{t\in[0,T]}\|e^{it\Lambda}U(t)\|_{Y^{N_0}}\leq \|U(0)\|_{Y^{N_0}}+C\epsilon^{3/2}$. In view of the continuity of the norm $Y^{N_0}$ 
(see Proposition \ref{LocExLemma} (iii)), it follows that
\begin{equation*}
 \sup_{t\in[0,T]}\|e^{it\Lambda}U(t)\|_{Y^{N_0}}\leq 2\|U(0)\|_{Y^{N_0}}
\end{equation*}
for any solution $U\in C([0,T]:H^N)$ of the initial-value problem \eqref{invp} with $\|U_0\|_{H^N}+\|U_0\|_{Y^{N_0}}\leq\overline{\varepsilon}$. 
The global regularity part of the theorem and the bound \eqref{invp2} follow, using again the energy estimate \eqref{EGrowth}.

\section{Proof of Proposition \ref{LocExLemma}}\label{LocalExistence}

In this section we prove Proposition \ref{LocExLemma}, using the energy method. 
We regularize the equation by parabolic regularization, get uniform estimates and then pass to the limit. The method is, of course, well known. We present all the details here for two reasons: to justify the key energy estimate \eqref{EGrowth}, with the somewhat unusual $Z'$ norm in the right-hand side of the inequality, and to justify the claim in Proposition \ref{LocExLemma} that the solution defines a continuous flow in the $Z$ space. These facts are important in passing to the global result.

We start with the main lemma.

\begin{lemma}\label{energylem1} 
With $K=50$, there is $\delta_1>0$ such that if
\begin{equation*}
 U_0\in H^\infty\text{ and }\|U_0\|_{H^3}\leq \delta_1
\end{equation*}
then, for any $\varepsilon>0$, there is a solution $U=U^\varepsilon\in C([0,1]:H^{K+2})$ of the approximate equation
\begin{equation}\label{AppEq}
\left(\partial_t+i\Lambda-\varepsilon\Delta\right)U=\frac{i}{4}\sum_{j=1}^2\Lambda R_j\left[\vert\nabla\vert\Lambda^{-1}(U+\overline{U})\cdot R_j(U-\overline{U})\right]+\frac{i}{8}\sum_{j=1}^2\vert\nabla\vert\left[R_j(U-\overline{U})\cdot R_j(U-\overline{U})\right]
\end{equation}
with initial data $U(0)=U_0$. Moreover,
\begin{equation}\label{ale1.1}
 \sup_{t\in[0,1]}\|U(t)\|_{H^3}\lesssim \delta_1
\end{equation}
and, for any $\sigma\in[3,K]\cap\mathbb{Z}$ and $0\leq s\leq t\leq 1$,
\begin{equation}\label{ale1}
E_\sigma(t)\le E_\sigma(s)+C\int_s^t\|U(t')\|_{Z'}\cdot E_\sigma(t^\prime)dt^\prime,
\end{equation}
where
\begin{equation}\label{DefE}
\begin{split}
&E_P:=\int_{\mathbb{R}^2}PU\cdot\overline{PU}dx-\frac{1}{8}\sum_{j=1}^2\int_{\mathbb{R}^2}\left[\vert\nabla\vert\Lambda^{-1}(U+\overline{U})\right]\cdot\left[R_jP(U-\overline{U})\cdot R_jP(U-\overline{U})\right]dx,\\
&E_\sigma:=\sum_{P=D^\alpha,\,\vert\alpha\vert\le \sigma}E_P.
\end{split}
\end{equation}
\end{lemma}

\begin{proof}[Proof of Lemma \ref{energylem1}]
By standard parabolic theory, for any fixed $\varepsilon>0$, the equation \eqref{AppEq} admits a local solution $U=U^\varepsilon\in C([0,T^\varepsilon],H^{K+2})$, for some $T^\varepsilon>0$, with 
\begin{equation}\label{work}
\sup_{t\in[0,T^\varepsilon]}\|U(t)\|_{H^3}\lesssim \delta_1.
\end{equation}
In view of the definition \eqref{DefE}
\footnote{It is important to include the cubic terms in the definition of $E_P$, in order to obtain a key cancellation between the terms $I_P$ and $IV_P$ in the formulas below.}, it follows that
\begin{equation}\label{work2}
\|U(t)\|_{H^\sigma}^2\lesssim E_{\sigma}(t)\lesssim \|U(t)\|_{H^\sigma}^2,\qquad t\in[0,T^\varepsilon],\,\sigma\in [0,K].
\end{equation}

We rewrite $U=X+iY$, thus
\begin{equation}\label{realsystem}
\begin{split}
&\partial_tX=\Lambda Y+\varepsilon\Delta X-\sum_{j=1}^2\Lambda R_j(|\nabla|\Lambda^{-1}X\cdot R_jY),\\
&\partial_tY=-\Lambda X+\varepsilon\Delta Y-(1/2)\sum_{j=1}^2|\nabla|(R_jY\cdot R_jY),\\
&E_P=\int_{\mathbb{R}^2}(PX)^2+(PY)^2\,dx+\sum_{j=1}^2\int_{\mathbb{R}^2}|\nabla|\Lambda^{-1}X\cdot (R_jPY)^2\,dx.
\end{split}
\end{equation}
Therefore, for $P=D^\alpha$, $|\alpha|\leq\sigma$,
\begin{equation*}
\frac{d}{dt}E_P=I_P+II_P+III_P+IV_P+V_P+VI_P-\varepsilon VII_P,
\end{equation*}
where
\begin{equation*}
I_P:=\int_{\mathbb{R}^2}-2PX\cdot \sum_{j=1}^2P\Lambda R_j(|\nabla|\Lambda^{-1}X\cdot R_jY)\,dx,
\end{equation*}
\begin{equation*}
II_P:=\int_{\mathbb{R}^2}-PY\cdot \sum_{j=1}^2P|\nabla|(R_jY\cdot R_jY)\,dx,
\end{equation*}
\begin{equation*}
III_P:=\sum_{j=1}^2\int_{\mathbb{R}^2}|\nabla|Y\cdot (R_jPY)^2\,dx,
\end{equation*}
\begin{equation*}
IV_P:=\sum_{j=1}^2\int_{\mathbb{R}^2}-|\nabla|\Lambda^{-1}X\cdot 2R_jPY\cdot \Lambda R_jPX\,dx,
\end{equation*}
\begin{equation*}
V_P:=\sum_{j=1}^2\int_{\mathbb{R}^2}-\sum_{m=1}^2|\nabla|R_m(|\nabla|\Lambda^{-1}X\cdot R_mY)\cdot (R_jPY)^2\,dx,
\end{equation*}
\begin{equation*}
VI_P:=\sum_{j=1}^2\int_{\mathbb{R}^2}-|\nabla|\Lambda^{-1}X\cdot R_jPY\cdot R_jP|\nabla|\big(\sum_{m=1}^2R_mY\cdot R_mY\big)\,dx
\end{equation*}
\begin{equation*}
\begin{split}
VII_P&:=2\int_{\mathbb{R}^2}\sum_{m=1}^2(\partial_mPX)^2+(\partial_mPY)^2\,dx\\
&+\sum_{j,m=1}^2\int_{\mathbb{R}^2}4\partial_m|\nabla|\Lambda^{-1}X\cdot R_jPY\cdot \partial_mR_jPY+2|\nabla|\Lambda^{-1}X\cdot (\partial_mR_jPY)^2\,dx.
\end{split}
\end{equation*}

We would like to bound the terms in the expression above in terms of $\|U\|_{H^\sigma}^2\|U\|_{Z'}$. For this we need the bounds
\begin{equation}\label{ka}
\begin{split}
&\sum_{m,j=1}^2\|\partial_jR_m U\|_{L^\infty}+\sum_{m,j=1}^2\|\partial_j(|\nabla|\Lambda^{-1}X\cdot R_mY)\|_{L^\infty}\lesssim \|U\|_{Z'},\\
&\|\Lambda(D^\alpha fD^\beta g)\|_{L^2}\lesssim\| \nabla f\|_{L^\infty}\Vert g\Vert_{H^{\sigma}}+\|\nabla g\|_{L^\infty}\| f\|_{H^{\sigma}},\qquad|\al|,|\be|\geq 1,\,|\alpha|+|\beta|\leq\sigma.
\end{split}
\end{equation}
which are proved in Lemma \ref{tech1.5} (recall that $\|U(t)\|_{H^3}\lesssim\delta_1$, $t\in[0,T^\varepsilon]$).

Using \eqref{ka}
\begin{equation*}
|III_P|+|V_P|\lesssim \|U\|_{H^\sigma}^2\|U\|_{Z'}.
\end{equation*}
Moreover, using also integration by parts and the formulas $-|\nabla|=\partial_1R_1+\partial_2R_2$ and $\partial_jR_m=\partial_jR_m$,
\begin{equation*}
\begin{split}
&|II_P|\lesssim\sum_{j,m=1}^2\Big|\int_{\mathbb{R}^2}PR_mY\cdot P\partial_m(R_jY\cdot R_jY)\,dx\Big|\\
&\lesssim \sum_{j,m=1}^2\Big[\Big|\int_{\mathbb{R}^2}PR_mY\cdot \partial_m[P(R_jY\cdot R_jY)-bR_jY\cdot PR_jY]\,dx\Big|
+\Big|\int_{\mathbb{R}^2}PR_mY\cdot \partial_m(R_jY\cdot PR_jY)\,dx\Big|\Big]\\
&\lesssim \|U\|_{H^\sigma}^2\|U\|_{Z'}+\sum_{j,m=1}^2\Big|\int_{\mathbb{R}^2}PR_mY\cdot \partial_jPR_mY\cdot R_jY\,dx\Big|\\
&\lesssim \|U\|_{H^\sigma}^2\|U\|_{Z'},
\end{split}
\end{equation*}
and similarly,
\begin{equation*}
\begin{split}
|VI_P|&\lesssim\sum_{j,m=1}^2\Big|\int_{\mathbb{R}^2}|\nabla|\Lambda^{-1}X\cdot R_jPY\cdot \partial_jP\big(R_mY\cdot R_mY\big)\,dx\Big|\\
&\lesssim \sum_{j,m=1}^2\Big|\int_{\mathbb{R}^2}|\nabla|\Lambda^{-1}X\cdot R_jPY\cdot \partial_j(PR_mY\cdot R_mY\big)\,dx\Big|\\
&+\sum_{j,m=1}^2\Big|\int_{\mathbb{R}^2}|\nabla|\Lambda^{-1}X\cdot R_jPY\cdot \partial_j\big[P\big(R_mY\cdot R_mY\big)-bPR_mY\cdot R_mY\big]\,dx\Big|\\
&\lesssim \sum_{j,m=1}^2\Big|\int_{\mathbb{R}^2}\big(|\nabla|\Lambda^{-1}X\cdot R_mY\big)\cdot R_jPY\cdot \partial_mPR_jY\,dx\Big|+\|U\|_{H^\sigma}^2\|U\|_{Z'}\\
&\lesssim \|U\|_{H^\sigma}^2\|U\|_{Z'},
\end{split}
\end{equation*}
where $b=1$ if $P=\mathrm{Id}$ and $b=2$ if $P=D^\alpha$, $|\alpha|\geq 1$. Therefore
\begin{equation}\label{ale2}
\sum_{P=D^\alpha,\,|\alpha|\leq\sigma}\big[|II_P|+|III_P|+|V_P|+|VI_P|\big]\lesssim \|U\|_{H^\sigma}^2\|U\|_{Z'}.
\end{equation}

In addition, using \eqref{ka} and the bound \eqref{easy2},
\begin{equation*}
\begin{split}
|I_P+IV_P|&=2\Big|\sum_{j=1}^2\int_{\mathbb{R}^2}R_jPX\cdot \big[P\Lambda (|\nabla|\Lambda^{-1}X\cdot R_jY)-\Lambda(|\nabla|\Lambda^{-1}X\cdot R_jPY)\big]\,dx\Big|\\
&\lesssim \sum_{j=1}^2\Big|\int_{\mathbb{R}^2}R_jPX\cdot \Lambda \big[P|\nabla|\Lambda^{-1}X\cdot R_jY\big]\,dx\Big|+\|U\|_{H^\sigma}^2\|U\|_{Z'}\\
&\lesssim \sum_{j=1}^2\Big|\int_{\mathbb{R}^2}R_jPX\cdot P|\nabla| X\cdot R_jY\,dx\Big|+\|U\|_{H^\sigma}^2\|U\|_{Z'}.
\end{split}
\end{equation*}
Since $|\nabla|=-\partial_1R_1-\partial_2R_2$ and $\partial_jR_m=\partial_mR_j$ we can further estimate
\begin{equation}\label{ale3}
\begin{split}
|I_P+IV_P|&\lesssim \sum_{j,m=1}^2\Big|\int_{\mathbb{R}^2}R_jPX\cdot P\partial_mR_mX\cdot R_jY\,dx\Big|+\|U\|_{H^\sigma}^2\|U\|_{Z'}\\
&\lesssim \sum_{j,m=1}^2\Big|\int_{\mathbb{R}^2}\partial_jR_mPX\cdot PR_mX\cdot R_jY\,dx\Big|+\|U\|_{H^\sigma}^2\|U\|_{Z'}\\
&\lesssim \|U\|_{H^\sigma}^2\|U\|_{Z'}.
\end{split}
\end{equation}

Notice also that, for any $m,j\in\{1,2\}$
\begin{equation*}
\begin{split}
&\Big|\int_{\mathbb{R}^2}\partial_m|\nabla|\Lambda^{-1}X\cdot R_jPY\cdot \partial_mR_jPY\,dx\Big|\\
&\lesssim \Big|\int_{\mathbb{R}^2}\partial_m|\nabla|\Lambda^{-1}X\cdot P_{(-\infty,0]}R_jPY\cdot \partial_mR_jPY\,dx\Big|+\Big|\int_{\mathbb{R}^2}\partial_m|\nabla|\Lambda^{-1}X\cdot P_{[1,\infty)}R_jPY\cdot \partial_mR_jPY\,dx\Big|\\
&\lesssim \|\nabla U\|^2_{H^\sigma}\big(\|P_{(-\infty,0]}R_jPY\|_{L^\infty}+\|\partial_m|\nabla|\Lambda^{-1}X\|_{L^\infty}\big)\\
&\lesssim \|\nabla U\|^2_{H^\sigma}\|U\|_{H^3}.
\end{split}
\end{equation*}
Since $\|U\|_{H^3}$ is small, it follows that
\begin{equation*}
\sum_{P=D^\alpha,\,|\alpha|\leq\sigma}VII_P\geq 0.
\end{equation*}
The desired inequality \eqref{ale1} follows for $0\leq s\leq t\leq T^\varepsilon$, using also \eqref{ale2} and \eqref{ale3}. A simple continuity argument, using this inequality with $\sigma=3$ and the bound $\|U(t)\|_{Z'}\lesssim \|U(t)\|_{H^3}$, shows that the solution $U^\varepsilon$ can be extended up to time $1$, which completes the proof of the lemma.
\end{proof}

We estimate now differences of smooth solutions.

\begin{lemma}\label{enerdif}
(i) Assume $K\in[10,50]$, $I\subseteq\mathbb{R}$ is an interval, and $U,U'\in C(I:H^{K+2})$ satisfy
\begin{equation}\label{fgh7}
\begin{split}
&(\partial_t+i\Lambda)U=F+\frac{i}{4}\sum_{j=1}^2\Lambda R_j\left[\vert\nabla\vert\Lambda^{-1}(U+\overline{U})\cdot R_j(U-\overline{U})\right]+\frac{i}{8}\sum_{j=1}^2\vert\nabla\vert\left[R_j(U-\overline{U})\cdot R_j(U-\overline{U})\right],\\
&(\partial_t+i\Lambda)U'=F'+\frac{i}{4}\sum_{j=1}^2\Lambda R_j\left[\vert\nabla\vert\Lambda^{-1}(U'+\overline{U'})\cdot R_j(U'-\overline{U'})\right]+\frac{i}{8}\sum_{j=1}^2\vert\nabla\vert\left[R_j(U'-\overline{U'})\cdot R_j(U'-\overline{U'})\right]
\end{split}
\end{equation}
on $\mathbb{R}^2\times I$. Assume in addition the $F,F'\in C(I:H^K)$ and let
\begin{equation}\label{fgh2}
M_\sigma(t):=1+\|U(t)\|_{H^\sigma}+\|U'(t)\|_{H^\sigma},\qquad\sigma\in [3,K]\cap\mathbb{Z}.
\end{equation}
Let
\begin{equation*}
\begin{split}
&U=X+iY,\,\,U'=X'+iY',\,\,U^\delta=U'-U=X^\delta+iY^\delta,\\
&F=G+iH,\,\,F'=G'+iH',\,\,F^\delta=F'-F=G^\delta+iH^\delta.
\end{split}
\end{equation*}
For any $P=D^\alpha$, $|\alpha|\leq\sigma\in[3,K]\cap\mathbb{Z}$, let
\begin{equation}\label{fgh0}
E_P^\delta:=\int_{\mathbb{R}^2}(PX^\delta)^2+(PY^\delta)^2\,dx+\sum_{j=1}^2\int_{\mathbb{R}^2}|\nabla|\Lambda^{-1}X'\cdot (R_jPY^\delta)^2\,dx.
\end{equation}
Then, for any $t\in I$ for which $M_3(t)\leq 2$, we have
\begin{equation}\label{EEE}
|\partial_tE_P^\delta(t)|\lesssim (M_\sigma+\|F'(t)\|_{H^3})\|U^\delta(t)\|_{H^\sigma}^2+\|U\|_{H^{\sigma+1}}\|U^\delta\|_{H^{3/2}}\|U^\delta\|_{H^\sigma}+\|U^\delta\|_{H^\sigma}\|F^\delta\|_{H^\sigma}.
\end{equation}

(ii) If $U,U'\in C(I:H^3)$ are solutions of \eqref{fgh7} with $F=F'=0$ and $M_3(t)=1+\|U(t)\|_{H^3}+\|U'(t)\|_{H^3}\leq 2$ for some $t\in I$ then
\begin{equation}\label{EEE1}
|\partial_tE^\delta_{\mathrm{Id}}(t)|\lesssim \|U^\delta(t)\|_{L^2_x}^2.
\end{equation}
\end{lemma}

\begin{proof}[Proof of Lemma \ref{enerdif}] The real variables $X^\delta,Y^\delta,X'$ satisfy the equations
\begin{equation}\label{fgh1}
\begin{split}
&\partial_tX^\delta=\Lambda Y^\delta+G^\delta-\sum_{j=1}^2\Lambda R_j(|\nabla|\Lambda^{-1}X'\cdot R_jY'-|\nabla|\Lambda^{-1}X\cdot R_jY),\\
&\partial_tY^\delta=-\Lambda X^\delta+H^\delta-(1/2)\sum_{j=1}^2|\nabla|(R_jY'\cdot R_jY'-R_jY\cdot R_jY),\\
&\partial_tX'=\Lambda Y'+G'-\sum_{j=1}^2\Lambda R_j(|\nabla|\Lambda^{-1}X'\cdot R_jY').
\end{split}
\end{equation}
Using the definition \eqref{fgh0} we calculate
\begin{equation*}
\partial_tE^\delta_P=I^\delta_P+II^\delta_P+III^\delta_P+IV^\delta_P+V^\delta_P+VI^\delta_P,
\end{equation*}
where
\begin{equation*}
I_P^\delta:=\int_{\mathbb{R}^2}-2PX^\delta\cdot \sum_{j=1}^2P\Lambda R_j(|\nabla|\Lambda^{-1}X'\cdot R_jY'-|\nabla|\Lambda^{-1}X\cdot R_jY)\,dx,
\end{equation*}
\begin{equation*}
II_P^\delta:=\int_{\mathbb{R}^2}-PY^\delta\cdot \sum_{j=1}^2P|\nabla|(R_jY'\cdot R_jY'-R_jY\cdot R_jY)\,dx,
\end{equation*}
\begin{equation*}
III_P^\delta:=\sum_{j=1}^2\int_{\mathbb{R}^2}|\nabla|\Lambda^{-1}\partial_tX'\cdot (R_jPY^\delta)^2\,dx,
\end{equation*}
\begin{equation*}
IV_P^\delta:=\sum_{j=1}^2\int_{\mathbb{R}^2}-|\nabla|\Lambda^{-1}X'\cdot 2R_jPY^\delta\cdot \Lambda R_jPX^\delta\,dx,
\end{equation*}
\begin{equation*}
V^\delta_P:=\sum_{j=1}^2\int_{\mathbb{R}^2}-|\nabla|\Lambda^{-1}X'\cdot R_jPY^\delta\cdot R_jP|\nabla|\big(\sum_{m=1}^2R_mY'\cdot R_mY'-R_mY\cdot R_mY\big)\,dx
\end{equation*}
\begin{equation*}
\begin{split}
VI^\delta_P&:=2\int_{\mathbb{R}^2}PX^\delta\cdot PG^\delta+PY^\delta\cdot PH^\delta\,dx+2\sum_{j=1}^2\int_{\mathbb{R}^2}|\nabla|\Lambda^{-1}X'\cdot R_jPY^\delta\cdot R_jPH^\delta\,dx.
\end{split}
\end{equation*}

Clearly,
\begin{equation*}
|VI^\delta_P|\lesssim M_3\|U^\delta\|_{H^\sigma}\|F^\delta\|_{H^\sigma}
\end{equation*}
and, using the last identity in \eqref{fgh1},
\begin{equation*}
|III^\delta_P|\lesssim (M_3^2+\|F'\|_{H^3})\|U^\delta\|^2_{H^\sigma}.
\end{equation*}
Also, using Lemma \ref{tech1.5} (i) and the identities $|\nabla|=-\partial_1R_1-\partial_2R_2$ and $\partial_jR_m=\partial_mR_j$,
\begin{equation*}
\begin{split}
|II^\delta_P|&\lesssim\sum_{j=1}^2\Big|\int_{\mathbb{R}^2}PY^\delta\cdot P|\nabla|(R_jY^\delta\cdot R_jY^\delta)\,dx\Big|+\sum_{j=1}^2\Big|\int_{\mathbb{R}^2}PY^\delta\cdot P|\nabla|(R_jY\cdot R_jY^\delta)\,dx\Big|\\
&\lesssim M_3\|U^\delta\|^2_{H^\sigma}+\|U\|_{H^\sigma}\|U^\delta\|_{H^\sigma}\|U^\delta\|_{H^3}+\sum_{j,m=1}^2\Big|\int_{\mathbb{R}^2}R_mPY^\delta\cdot \partial_mPR_jY^\delta\cdot R_jY^\delta\,dx\Big|\\
&+\sum_{j,m=1}^2\Big[\Big|\int_{\mathbb{R}^2}R_mPY^\delta\cdot \partial_mPR_jY^\delta\cdot R_jY\,dx\Big|+\Big|\int_{\mathbb{R}^2}R_mPY^\delta\cdot \partial_mPR_jY\cdot R_jY^\delta\,dx\Big|\Big]\\
&\lesssim M_\sigma\|U^\delta\|^2_{H^\sigma}+\|U\|_{H^{\sigma+1}}\|U^\delta\|_{H^\sigma}\|U^\delta\|_{H^{3/2}}.
\end{split}
\end{equation*}
Similarly,
\begin{equation*}
\begin{split}
&|V^\delta_P|\lesssim\sum_{j,m=1}^2\Big|\int_{\mathbb{R}^2}|\nabla|\Lambda^{-1}X'\cdot R_jPY^\delta\cdot \partial_jP\big(R_mY^\delta\cdot R_mY^\delta+2R_mY\cdot R_mY^\delta\big)\,dx\Big|\\
&\lesssim M_3^2\|U^\delta\|_{H^\sigma}^2+M_3\|U\|_{H^\sigma}\|U^\delta\|_{H^\sigma}\|U^\delta\|_{H^3}+\sum_{j,m=1}^2\Big[\Big|\int_{\mathbb{R}^2}|\nabla|\Lambda^{-1}X'\cdot R_jPY^\delta\cdot \partial_jPR_mY^\delta\cdot R_mY^\delta\,dx\Big|\\
&+\Big|\int_{\mathbb{R}^2}|\nabla|\Lambda^{-1}X'\cdot R_jPY^\delta\cdot \partial_jPR_mY^\delta\cdot R_mY\,dx\Big|+\Big|\int_{\mathbb{R}^2}|\nabla|\Lambda^{-1}X'\cdot R_jPY^\delta\cdot \partial_jPR_mY\cdot R_mY^\delta\,dx\Big|\Big]\\
&\lesssim M_3M_\sigma\|U^\delta\|^2_{H^\sigma}+M_3\|U\|_{H^{\sigma+1}}\|U^\delta\|_{H^\sigma}\|U^\delta\|_{H^{3/2}},
\end{split}
\end{equation*}
and
\begin{equation*}
\begin{split}
|I^\delta_P+IV_P^\delta|&=2\Big|\sum_{j=1}^2\int_{\mathbb{R}^2}R_jPX^\delta\cdot\Lambda\big[P(|\nabla|\Lambda^{-1}X'\cdot R_jY^\delta+|\nabla|\Lambda^{-1}X^\delta\cdot R_jY)-|\nabla|\Lambda^{-1}X'\cdot R_jPY^\delta\big]\,dx\Big|\\
&\lesssim M_3\|U^\delta\|_{H^\sigma}^2+M_\sigma\|U^\delta\|_{H^\sigma}\|U^\delta\|_{H^3}+\sum_{j=1}^2\Big|\int_{\mathbb{R}^2}R_jPX^\delta\cdot\Lambda\big[P|\nabla|\Lambda^{-1}X'\cdot R_jY^\delta\big]\,dx\Big|\\
&+\sum_{j=1}^2\Big|\int_{\mathbb{R}^2}R_jPX^\delta\cdot\Lambda\big[P|\nabla|\Lambda^{-1}X^\delta\cdot R_jY\big]\,dx\Big|+\sum_{j=1}^2\Big|\int_{\mathbb{R}^2}R_jPX^\delta\cdot\Lambda\big[|\nabla|\Lambda^{-1}X^\delta\cdot PR_jY\big]\,dx\Big|.
\end{split}
\end{equation*}
Using also the estimate
\begin{equation*}
\|\Lambda(f\cdot g)-\Lambda f\cdot g\|_{L^2}\lesssim \|f\|_{L^2}\|g\|_{H^3},
\end{equation*}
see the proof of Lemma \ref{tech1.5} (iii), it follows that
\begin{equation*}
\sum_{j=1}^2\Big|\int_{\mathbb{R}^2}R_jPX^\delta\cdot\Lambda\big[|\nabla|\Lambda^{-1}X^\delta\cdot PR_jY\big]\,dx\Big|\lesssim \|U^\delta\|_{H^\sigma}\|U^\delta\|_{H^{3/2}}\|U\|_{H^{\sigma+1}}+M_\sigma\|U^\delta\|_{H^\sigma}\|U^\delta\|_{H^3}.
\end{equation*}
In addition, for $\widetilde{Y}\in\{Y,Y^\delta\}$,
\begin{equation*}
\begin{split}
\sum_{j=1}^2\Big|\int_{\mathbb{R}^2}&R_jPX^\delta\cdot\Lambda\big[P|\nabla|\Lambda^{-1}X^\delta\cdot R_j\widetilde{Y}\big]\,dx\Big|\lesssim M_3\|U^\delta\|_{H^\sigma}^2+\sum_{j=1}^2\Big|\int_{\mathbb{R}^2}R_jPX^\delta\cdot P|\nabla|X^\delta\cdot R_j\widetilde{Y}\,dx\Big|\\
&\lesssim M_3\|U^\delta\|_{H^\sigma}^2+\sum_{j,m=1}^2\Big|\int_{\mathbb{R}^2}\partial_mR_jPX^\delta\cdot PR_mX^\delta\cdot R_j\widetilde{Y}\,dx\Big|\lesssim M_3\|U^\delta\|_{H^\sigma}^2.
\end{split}
\end{equation*}
Finally
\begin{equation*}
\begin{split}
\sum_{j=1}^2\Big|\int_{\mathbb{R}^2}R_jPX^\delta\cdot\Lambda\big[P|\nabla|\Lambda^{-1}X\cdot R_jY^\delta\big]\,dx\Big|\lesssim \|U^\delta\|_{H^\sigma}\|U^\delta\|_{H^{3/2}}\|U\|_{H^{\sigma+1}},
\end{split}
\end{equation*}
and it follows that
\begin{equation*}
|I^\delta_P+IV_P^\delta|\lesssim M_\sigma\|U^\delta\|_{H^\sigma}^2+\|U^\delta\|_{H^\sigma}\|U^\delta\|_{H^{3/2}}\|U\|_{H^{\sigma+1}}.
\end{equation*}
The inequality \eqref{EEE} follows from these estimates.

To prove \eqref{EEE1} we notice that the previous estimates can be improved if $P=\mathrm{Id}$ and $F=F'=0$. Indeed, similar arguments as before show that
\begin{equation*}
\begin{split}
&|VI^\delta_{\mathrm{Id}}|=0,\qquad |III^\delta_{\mathrm{Id}}|\lesssim M_3^2\|U^\delta\|^2_{L^2},\qquad |II^\delta_{\mathrm{Id}}|\lesssim M_3\|U^\delta\|_{L^2}^2,\\
&|V^\delta_{\mathrm{Id}}|\lesssim M_3^2\|U^\delta\|_{L^2}^2,\qquad |I^\delta_{\mathrm{Id}}+IV^\delta_{\mathrm{Id}}|\lesssim M_3\|U^\delta\|_{L^2}^2,
\end{split}
\end{equation*}
and the desired inequality follows.
\end{proof}

We can now complete the proof of the proposition, using the definitions, Lemma \ref{energylem1}, Lemma \ref{enerdif}, and the Bona--Smith argument \cite{BoSm}.

\begin{proof}[Proof of Proposition \ref{LocExLemma} (i), (ii)] Assume $M\in[3,40]\cap\mathbb{Z}$ is fixed and $U_0\in H^M$ is a given data satisfying
\begin{equation*}
\|U_0\|_{H^3}\leq \delta_0:=\delta_1,
\end{equation*}
where $\delta_1$ is the small constant in Lemma \ref{energylem1}. Using Lemma \ref{energylem1}, for any $k\in\mathbb{Z}_+$ and $\varepsilon>0$ we construct a solution $U_k^\varepsilon\in C([0,1]:H^{K+2})$ of the equation \eqref{AppEq}, with initial condition $U_k^\varepsilon(0)=P_{(-\infty,k]}U_0$. It follows from \eqref{ale1.1} and \eqref{ale1} that
\begin{equation}\label{fgh3}
\sup_{t\in[0,1]}\|U_k^\varepsilon(t)\|_{H^3}\lesssim \|U_0\|_{H^3},\qquad \sup_{t\in[0,1]}\|U_k^\varepsilon(t)\|_{H^{M+d}}\lesssim 2^{dk}\|U_0\|_{H^M},\qquad d\in\{0,1,2\}.
\end{equation}

We fix $k$ and apply now Lemma \ref{enerdif} to the solutions $U:=U_k^\varepsilon$ and $U':=U_k^{\varepsilon'}$, with $F:=\varepsilon\Delta U_k^\varepsilon$, $F':=\varepsilon'\Delta U_k^{\varepsilon'}$. It follows from \eqref{EEE} that
\begin{equation*}
\begin{split}
\|U_k^\varepsilon(t')-U_k^{\varepsilon'}(t')\|^2_{H^\sigma}\lesssim\|U_k^\varepsilon(t)-U_k^{\varepsilon'}(t)\|^2_{H^\sigma}+\int_{t}^{t'}&\big(\|U_k^\varepsilon(s)-U_k^{\varepsilon'}(s)\|^2_{H^\sigma}+(\varepsilon+\varepsilon')\|U_k^\varepsilon(s)-U_k^{\varepsilon'}(s)\|_{H^\sigma}\big)\\
&\times\big(1+\|U_k^\varepsilon(s)\|_{H^{\sigma+2}}+\|U_k^{\varepsilon'}(s)\|_{H^{\sigma+2}}\big)\,ds
\end{split}
\end{equation*}
for any $0\leq t\leq t'\leq 1$ and any $\sigma\in[3,K]\cap\mathbb{Z}$. Therefore, using \eqref{fgh3},
\begin{equation*}
\sup_{t\in [0,1]}\|U_k^\varepsilon(t)-U_k^{\varepsilon'}(t)\|^2_{H^{M+2}}\leq C(k)(\varepsilon+\varepsilon'),
\end{equation*}
which shows that $\lim_{\varepsilon\to 0}U_k^\varepsilon$ exists for any $k$ fixed. To summarize, for any $k\in\mathbb{Z}_+$ we have constructed a solution $U_k\in C([0,1]:H^{M+2})$ of the equation
\begin{equation*}
(\partial_t+i\Lambda)U_k=\frac{i}{4}\sum_{j=1}^2\Lambda R_j\big[|\nabla|\Lambda^{-1}(U_k+\overline{U_k})\cdot R_j(U_k-\overline{U_k})\big]+\frac{i}{8}\sum_{j=1}^2|\nabla|\big[R_j(U_k-\overline{U_k})\cdot R_j(U_k-\overline{U_k})\big],
\end{equation*}
with $U_k(0)=P_{(-\infty,k]}U_0$. Moreover, see \eqref{fgh3},
\begin{equation}\label{fgh4}
\sup_{t\in[0,1]}\|U_k(t)\|_{H^3}\lesssim \|U_0\|_{H^3},\qquad \sup_{t\in[0,1]}\|U_k(t)\|_{H^{M+d}}\lesssim 2^{dk}\|U_0\|_{H^M},\qquad d\in\{0,1,2\}.
\end{equation}

We show now that the sequence $\{U_k\}_{k\in\mathbb{Z}_+}$ is Cauchy in $C([0,1]:H^M)$. For this we apply Lemma \ref{enerdif} to the solutions $U:=U_k$ and $U':=U_{k'}$, $k+1\leq k'$, with $F=F'=0$. It follows from \eqref{EEE1} that
\begin{equation*}
\|U_{k'}(t')-U_k(t')\|^2_{L^2}\lesssim\|U_{k'}(t)-U_k(t)\|_{L^2}^2+\int_{t}^{t'}\|U_{k'}(s)-U_k(s)\|^2_{L^2}\,ds,
\end{equation*}
for any $0\leq t\leq t'\leq 1$. Since $\|U_{k'}(0)-U_k(0)\|_{L^2}\leq 2^{-3k}$ it follows that $\|U_{k'}(t)-U_k(t)\|_{L^2}\lesssim 2^{-3k}$ for any $t\in[0,1]$. By interpolation with \eqref{fgh4},
\begin{equation}\label{fgh5}
\sup_{t\in[0,1]}\|U_{k'}(t)-U_k(t)\|_{H^{3/2}}\lesssim 2^{-3k/2}.
\end{equation}
It follows now from \eqref{EEE} and the bound \eqref{fgh4} that
\begin{equation*}
\begin{split}
\|U_{k'}(t')-U_k(t')\|^2_{H^M}&\lesssim\|U_{k'}(t)-U_k(t)\|_{H^M}^2\\
&+(1+\|U_0\|_{H^M})\int_{t}^{t'}\|U_{k'}(s)-U_k(s)\|^2_{H^M}+2^{-k/2}\|U_{k'}(s)-U_k(s)\|_{H^M}\,ds,
\end{split}
\end{equation*}
for any $0\leq t\leq t'\leq 1$. Therefore
\begin{equation*}
\sup_{t\in[0,1]}\|U_{k'}(t)-U_k(t)\|_{H^M}\leq C(\|U_0\|_{H^M})[2^{-k/2}+\|U_{k'}(0)-U_k(0)\|_{H^M}.
\end{equation*}
This shows that the sequence $\{U_k\}_{k\in\mathbb{Z}_+}$ is Cauchy in $C([0,1]:H^M)$. 

Letting $U:=\lim_{k\to\infty}U_k$, we have constructed a solution $U\in C([0,1]:H^M)$ of \eqref{invp} such that $U(0)=U_0$ and
\begin{equation*}
 \sup_{t\in[0,1]}\|U(t)\|_{H^3}\lesssim \| U_0\|_{H^3}.
\end{equation*}
The inequality \eqref{EGrowth} follows from the corresponding inequality \eqref{ale1} satisfied uniformly by the functions $U_k^\varepsilon$. The uniqueness of the solution $U$ follows from \eqref{EEE1}.
\end{proof}

\begin{proof}[Proof of Proposition \ref{LocExLemma} (iii)] {\bf{Step 1.}} For any integer $J\geq 0$ and $f\in H^{N_0}$ we define
\begin{equation*}
\|f\|_{Z_J}:=\sup_{k\in\mathbb{Z}}(2^{k/10}+2^{10k})\Big[\|P_kf\|_{L^2}+\sum_{j\in\mathbb{Z}_+}2^{\min(j,2J-j)}\|\varphi_j(x)\cdot P_kf(x)\|_{L^2}\Big],
\end{equation*}
 and notice that
\begin{equation*}
 \|f\|_{Z_J}\leq\|f\|_{Z},\qquad \|f\|_{Z_J}\lesssim_J\|f\|_{H^{N_0}}.
\end{equation*}
It suffices to prove that for any $t,t'\in [0,T]$ with $0\leq t'-t\leq 1$, and any $J\in\mathbb{Z}_+$ we have
\begin{equation}\label{cu1}
 \|e^{it'\Lambda}U(t')-e^{it\Lambda}U(t)\|_{Z_J}\leq C\big(T,\sup_{t\in[0,T]}\|U(t)\|_{H^{N_0}},\|U_0\|_Z\big)\cdot |t'-t|(1+\sup_{s\in[t,t']}\|e^{is\Lambda}U(s)\|_{Z_J}).
\end{equation}
Indeed, assuming \eqref{cu1}, it follows easily that 
\begin{equation*}
\begin{split}
&\sup_{t\in[0,T]}\|e^{it\Lambda}U(t)\|_{Z_J}\leq C\big(T,\sup_{t\in[0,T]}\|U(t)\|_{H^{N_0}},\|U_0\|_Z\big),\\
&\|e^{it'\Lambda}U(t')-e^{it\Lambda}U(t)\|_{Z_J}\leq C\big(T,\sup_{t\in[0,T]}\|U(t)\|_{H^{N_0}},\|U_0\|_Z\big)\cdot |t'-t|, \qquad t,t'\in[0,T],
\end{split}
\end{equation*}
uniformly in $J$, and the desired conclusions follow by letting $J\to\infty$.

To prove \eqref{cu1} we define $V_\pm(t)=e^{\pm it\Lambda}U(t)$, as in section \ref{Bootstrap}, and use the formula \eqref{norm3},
\begin{equation}\label{cu1.5}
\begin{split}
&\widehat{V_+}(\xi,t')-\widehat{V_+}(\xi,t)\\
&=\sum_{(\mu,\nu)\in\{(+,+),(+,-),(-,-)\}}\int_{t}^{t'}\int_{\mathbb{R}^2}e^{is[\Lambda(\xi)-\mu\Lambda(\xi-\eta)-\nu\Lambda(\eta)]}m_{\mu\nu}(\xi,\eta)\widehat{V_\mu}(\xi-\eta,s)\widehat{V_\nu}(\eta,s)\,d\eta ds.
\end{split}
\end{equation}
Letting
\begin{equation*}
\widehat{A_{\mu\nu}}(\xi,s):=\int_{\mathbb{R}^2}e^{is[\Lambda(\xi)-\mu\Lambda(\xi-\eta)-\nu\Lambda(\eta)]}m_{\mu\nu}(\xi,\eta)\widehat{V_\mu}(\xi-\eta,s)\widehat{V_\nu}(\eta,s)\,d\eta,
\end{equation*}
it suffices to prove that for any $s\in[0,T]$ and $(\mu,\nu)\in\{(+,+),(+,-),(-,-)\}$
\begin{equation*}
\|A_{\mu\nu}(s)\|_{Z_J}\leq C\big(s,\|V_+(s)\|_{H^{N_0}},\|V_+(0)\|_Z\big)\cdot(1+\|V_+(s)\|_{Z_J}),\qquad \text{ uniformly in }J.
\end{equation*}
In view of the definition, it suffices to prove that for any $k\in\mathbb{Z}$ and $s\in[0,T]$
\begin{equation}\label{cu2}
\begin{split}
\|P_k A_{\mu\nu}(s)\|_{L^2}&+\sum_{j\geq 0}2^{\min(j,2J-j)}\|\varphi_j\cdot P_k A_{\mu\nu}(s)\|_{L^2}\\
&\leq C\big(s,\|V_+(s)\|_{H^{N_0}},\|V_+(0)\|_Z\big)\cdot(1+\|V_+(s)\|_{Z_J})(2^{k/10}+2^{10 k})^{-1}.
\end{split}
\end{equation}
We prove this bound in several steps: first we estimate the sum over small values of $j$, see \eqref{cu2.5}. Then we estimate the contribution of the very low frequencies of the inputs $V_\mu,V_\nu$, see \eqref{cu4}. Finally, we pass to the physical space to estimate the remaining contributions using localization, see \eqref{cu7} and \eqref{cu15}.

For simplicity of notation, in the rest of the proof we let $\widetilde{C}$ denote constants that may depend only on $s,\|V_+(s)\|_{H^{N_0}},\|V_+(0)\|_Z$, similar to the constant in the right-hand side of \eqref{cu2}. 

{\bf{Step 2.}} Notice first that
\begin{equation*}
\|\widehat{P_kA_{\mu\nu}}(s)\|_{L^\infty}\lesssim 2^{-(N_0-1)k_+}\|V_+(s)\|_{H^{N_0}}^2.
\end{equation*}
Therefore $\|P_kA_{\mu\nu}(s)\|_{L^2}\lesssim 2^k2^{-(N_0-1)k_+}\|V_+(s)\|_{H^{N_0}}^2$ and it follows that
\begin{equation}\label{cu2.5}
 \|P_k A_{\mu\nu}(s)\|_{L^2}+\sum_{j\leq\max(-41k/40,2k)+C(s)}2^j\|\varphi_j\cdot P_k A_{\mu\nu}(s)\|_{L^2}\leq \widetilde{C}(2^{k/10}+2^{10 k})^{-1}.
\end{equation}

{\bf{Step 3.}} To estimate the remaining sum
\begin{equation*}
\sum_{j\geq\max(-41k/40,2k)+C(s)}2^{\min(j,2J-j)}\|\varphi_j\cdot P_k A_{\mu\nu}(s)\|_{L^2},
\end{equation*}
we decompose
\begin{equation}\label{cu3.5}
\begin{split}
&P_k A_{\mu\nu}(s)=\sum_{(k_1,k_2)\in \mathcal{X}_k}A_{k,k_1,k_2}^{\mu\nu}(s),\\
&\widehat{A_{k,k_1,k_2}^{\mu\nu}}(\xi,s):=\int_{\mathbb{R}^2}e^{is[\Lambda(\xi)-\mu\Lambda(\xi-\eta)-\nu\Lambda(\eta)]}\varphi_k(\xi)m_{\mu\nu}(\xi,\eta)\widehat{P_{k_1}V_\mu}(\xi-\eta,s)\widehat{P_{k_2}V_\nu}(\eta,s)\,d\eta,
\end{split}
\end{equation}
where 
\begin{equation*}
\begin{split}
&\mathcal{X}^1_k:=\{(k_1,k_2)\in\mathbb{Z}\times\mathbb{Z}:|\max(k_1,k_2)-k|\leq 4\},\\
&\mathcal{X}^2_k:=\{(k_1,k_2)\in\mathbb{Z}\times\mathbb{Z}:\max(k_1,k_2)\geq k+4\text{ and }|k_1-k_2|\leq 4\},\\
&\mathcal{X}_k:=\mathcal{X}_k^1\cup\mathcal{X}_k^2.
\end{split}
\end{equation*}
Since
\begin{equation*}
|m_{\mu\nu}(\xi,\eta)|\lesssim |\xi|+|\eta|+|\xi-\eta|, 
\end{equation*}
see \eqref{ms1}--\eqref{ms3}, it follows from \eqref{cu1.5} that
\begin{equation}\label{cu7.5}
\begin{split}
\|\widehat{P_l V_{\pm}}(s)\|_{L^\infty}&\lesssim \|\widehat{P_l V_{\pm}}(0)\|_{L^\infty}+(1+s)\Big[2^l\|P_{[l-4,l+4]}V_+(s)\|_{L^2}\|P_{(-\infty,l+4]}V_+(s)\|_{L^2}\\
&+\sum_{(k_1,k_2)\in\mathcal{X}_l^2}(2^{k_1}+2^{k_2}) \|P_{k_1}V_+(s)\|_{L^2}\|P_{k_2}V_+(s)\|_{L^2}\Big]\\
&\leq \widetilde{C}(2^{l/10}+2^{10l})^{-1}.
\end{split}
\end{equation}
Therefore
\begin{equation*}
\begin{split}
\|A_{k,k_1,k_2}^{\mu\nu}(s)\|_{L^2}&\lesssim 2^{\max(k,k_1,k_2)}\min\big(\|\widehat{P_{k_1}V_\mu}(s)\|_{L^1}\|\widehat{P_{k_2}V_\nu}(s)\|_{L^2},\|\widehat{P_{k_1}V_\mu}(s)\|_{L^2}\|\widehat{P_{k_2}V_\nu}(s)\|_{L^1}\big)\\
&\leq \widetilde{C}2^{\max(k_1,k_2)}(1+2^{\max(k_1,k_2)})^{-N_0}2^{(19/10)\min(k_1,k_2)}.
\end{split}
\end{equation*}
Therefore
\begin{equation}\label{cu4}
\sum_{j\geq \max(-41k/40,2k)}\sum_{(k_1,k_2)\in \mathcal{X}_k,\,\min(k_1,k_2)\leq-2j/3}2^j\|A_{k,k_1,k_2}^{\mu\nu}(s)\|_{L^2}\leq \widetilde{C}2^{-10 k_+}.
\end{equation}

{\bf{Step 4.}} It remains to bound the contribution
\begin{equation*}
\sum_{j\geq \max(-41k/40,2k)+C(s)}\sum_{(k_1,k_2)\in \mathcal{X}_k,\,\min(k_1,k_2)\geq-2j/3}2^{\min(j,2J-j)}\|\varphi_j\cdot A_{k,k_1,k_2}^{\mu\nu}(s)\|_{L^2}.
\end{equation*}
For this we write, using \eqref{cu3.5},
\begin{equation*}
\begin{split}
&A_{k,k_1,k_2}^{\mu\nu}(x,s)=\int_{\mathbb{R}^2\times\mathbb{R}^2}P_{k_1}V_\mu(y,s)P_{k_2}V_\nu(z,s)K_{k,k_1,k_2}^{\mu\nu}(x,y,z)\,dydz,\\
&K_{k,k_1,k_2}^{\mu\nu}(x,y,z):=c\int_{\mathbb{R}^2\times\mathbb{R}^2}e^{is[\Lambda(\xi)-\mu\Lambda(\xi-\eta)-\nu\Lambda(\eta)]}e^{i(x\cdot\xi-y\cdot(\xi-\eta)-z\cdot\eta)}m_{k,k_1,k_2}^{\mu\nu}(\xi,\eta)\,d\xi d\eta,\\
&m_{k,k_1,k_2}^{\mu\nu}(\xi,\eta):=m_{\mu\nu}(\xi,\eta)\varphi_k(\xi)\varphi_{[k_1-1,k_1+1]}(\xi-\eta)\varphi_{[k_2-1,k_2+1]}(\eta).
\end{split}
\end{equation*}
Since
\begin{equation*}
\sup_{\xi\in\mathbb{R}^2}|D^\alpha \Lambda(\xi)|\lesssim_{|\alpha|}1,\qquad |\alpha|\geq 1,
\end{equation*}
it follows by integration by parts that
\begin{equation}\label{cu5}
\begin{split}
&|K_{k,k_1,k_2}^{\mu\nu}(x,y,z)|\lesssim (1+2^{4k_1}+2^{4k_2})(|x-y|+|x-z|)^{-20}\\
\text{if }&j\geq C(s),\,\,\min(k,k_1,k_2)\geq -40j/41,\,\,|x-y|+|x-z|\geq 2^{j-4},
\end{split}
\end{equation}
provided that the constant $C(s)$ is fixed sufficiently large. Letting
\begin{equation*}
\begin{split}
&B_{j,k,k_1,k_2}^{\mu\nu}(x,s):=\int_{\mathbb{R}^2\times\mathbb{R}^2}V^\mu_{j,k_1}(y,s)V^\nu_{j,k_2}(z,s)K_{k,k_1,k_2}^{\mu\nu}(x,y,z)\,dydz,\\
&V^\mu_{j,l}(x,s):=\varphi_{[j-4,j+4]}(x)\cdot P_{l}V_\mu(x,s),
\end{split}
\end{equation*}
it follows from \eqref{cu5} that
\begin{equation*}
\|\varphi_j\cdot [A_{k,k_1,k_2}^{\mu\nu}(s)-B_{j,k,k_1,k_2}^{\mu\nu}(s)]\|_{L^2}\lesssim 2^{-10j}(1+2^{4k_1}+2^{4k_2})\|P_{k_1}V_+(s)\|_{L^2}\|P_{k_2}V_+(s)\|_{L^2}.
\end{equation*}
Since $\|P_{l}V_+(s)\|_{L^2}\lesssim 2^{-N_0l_+}\|V_+(s)\|_{H^{N_0}}$, $l\in\mathbb{Z}$, it follows that
\begin{equation}\label{cu7}
\sum_{j\geq \max(-41k/40,2k)+C(s)}\sum_{(k_1,k_2)\in \mathcal{X}_k,\,\min(k_1,k_2)\geq-2j/3}2^j\|\varphi_j\cdot [A_{k,k_1,k_2}^{\mu\nu}(s)-B_{j,k,k_1,k_2}^{\mu\nu}(s)]\|_{L^2}\leq \widetilde{C}2^{-10 k_+}.
\end{equation}

{\bf{Step 5.}} Finally we rewrite $B_{j,k,k_1,k_2}^{\mu\nu}(s)$ in the frequency space,
\begin{equation*}
\widehat{B_{j,k,k_1,k_2}^{\mu\nu}}(\xi,s)=\int_{\mathbb{R}^2}e^{is[\Lambda(\xi)-\mu\Lambda(\xi-\eta)-\nu\Lambda(\eta)]}m^{\mu\nu}_{k,k_1,k_2}(\xi,\eta)\widehat{V^\mu_{j,k_1}}(\xi-\eta,s)\widehat{V_{j,k_2}^\nu}(\eta,s)\,d\eta,
\end{equation*}
and estimate
\begin{equation*}
\|\widehat{B_{j,k,k_1,k_2}^{\mu\nu}}(s)\|_{L^2}\lesssim 2^{\max(k_1,k_2)}2^{\min(k_1,k_2)}\|\widehat{V^\mu_{j,k_1}}(s)\|_{L^2}\|\widehat{V_{j,k_2}^\nu}(s)\|_{L^2}.
\end{equation*}
Therefore
\begin{equation}\label{cu12}
\begin{split}
\sum_{j\geq \max(-41k/40,2k)+C(s)}&\sum_{(k_1,k_2)\in \mathcal{X}_k,\,\min(k_1,k_2)\geq-2j/3}2^{\min(j,2J-j)}\|\varphi_j\cdot B_{j,k,k_1,k_2}^{\mu\nu}(s)\|_{L^2}\\
&\lesssim\sum_{(k_1,k_2)\in \mathcal{X}_k,\,k_1\leq k_2}\sum_{j\geq 10}2^{\min(j,2J-j)}2^{k_1+k_2}\|V^\pm_{j,k_1}(s)\|_{L^2}\|V_{j,k_2}^\pm(s)\|_{L^2}.
\end{split}
\end{equation}
We estimate the sum in the right-hand side of \eqref{cu12} in two different ways, depending on the relative sizes of $k,k_1,k_2$. If\footnote{By convention, here and in the rest of the paper, $[a,b]:=\emptyset$ if $a>b$.} $k_1=\min(k_1,k_2)\notin[100-10k,k-10]$ then we estimate $\|V^\pm_{j,k_1}(s)\|_{L^2}\leq \widetilde{C}2^{-N_0\max(k_1,0)}$, and the corresponding sum in the right-hand side of \eqref{cu12} is bounded by
\begin{equation*}
\sum_{(k_1,k_2)\in \mathcal{X}_k,\,k_1\leq k_2,\,k_1\notin[100-10k,k-10]}\widetilde{C}2^{k_1+k_2}2^{-N_0\max(k_1,0)}(2^{k_2/10}+2^{10k_2})^{-1}\|V_+(s)\|_{Z_J}\leq \widetilde{C}\|V_+(s)\|_{Z_J}2^{-10 k_+}.
\end{equation*}
On the other hand, if $k_1=\min(k_1,k_2)\in[100-10k,k-10]$ (in particular $k\geq 10$, $|k_2-k|\leq 4$) then we use $\|V^\pm_{j,k_2}(s)\|_{L^2}\leq \widetilde{C}2^{-N_0k}$, and the corresponding sum in the right-hand side of \eqref{cu12} is bounded by
\begin{equation*}
\sum_{(k_1,k_2)\in \mathcal{X}_k,\,k_1\leq k_2,\,k_1\in[100-10k,k-10]}\widetilde{C}2^{k_1+k_2}2^{-N_0k}(2^{k_1/10}+2^{10k_1})^{-1}\|V_+(s)\|_{Z_J}\leq \widetilde{C}\|V_+(s)\|_{Z_J}2^{-10 k}.
\end{equation*}
Therefore
\begin{equation}\label{cu15}
\begin{split}
\sum_{j\geq \max(-41k/40,2k)+C(s)}\sum_{(k_1,k_2)\in \mathcal{X}_k,\,\min(k_1,k_2)\geq-2j/3}2^{\min(j,2J-j)}\|\varphi_j\cdot B_{j,k,k_1,k_2}^{\mu\nu}(s)\|_{L^2}\\
\leq\widetilde{C}\|V_+(s)\|_{Z_J}2^{-10 k_+},
\end{split}
\end{equation}
which completes the proof.
\end{proof}

\section{Proof of Proposition \ref{Norm}}\label{Bootstrap}

Recall first some of the notation from section \ref{section2}: for $f\in H^1(\mathbb{R}^2)$ we define $|\nabla|f$, $R_jf$, $\Lambda f$, and $P_kf$ by multiplication with the Fourier multipliers 
\begin{equation}\label{not1}
\xi\to|\xi|,\qquad \xi\to i\xi_j/|\xi|,\qquad\xi\to\Lambda(\xi)=\sqrt{a|\xi|^2+b},\qquad \xi\to\varphi_k(\xi)=\varphi(|\xi|/2^k)-\varphi(|\xi|/2^{k-1})
\end{equation}
respectively, where $a,b\in(0,\infty)$, $j\in\{1,2\}$, $k\in\mathbb{Z}$, and $\varphi:\mathbb{R}\to[0,1]$ is an even smooth function supported in $[-8/5,8/5]$ and equal to $1$ in $[-5/4,5/4]$. For any set $I\in\mathbb{R}$ let
\begin{equation*}
\varphi_I=\sum_{m\in I\cap\mathbb{Z}}\varphi_m.
\end{equation*}
For any $k\in\mathbb{Z}$ let 
\begin{equation*}
k_+:=\max(k,0).
\end{equation*}
Recall also the definition of our main spaces of ``smooth'' solutions $X^N_T(\mathbb{R}^2\times[0,T])$, $Y^{N_0}(\mathbb{R}^2)$, and $Z(\mathbb{R}^2)$ (see \eqref{sec4} and \eqref{sec5}), where $T\geq 1$, $N=30$ and $N_0=20$. 

The rest of the section is concerned with the proof of Proposition \ref{Norm}, which we recall below:

\begin{proposition}\label{MainBootstrap}
Assume that $U\in X^N_T$ is a solution of the equation
\begin{equation}\label{bo1}
(\partial_t+i\Lambda)U=\frac{i}{4}\sum_{j=1}^2\Lambda R_j\big[|\nabla|\Lambda^{-1}(U+\overline{U})\cdot R_j(U-\overline{U})\big]+\frac{i}{8}\sum_{j=1}^2|\nabla|\big[R_j(U-\overline{U})\cdot R_j(U-\overline{U})\big]
\end{equation}
with the property that
\begin{equation}\label{bo2}
\|U\|_{X^N_T}\leq\varepsilon_0\leq 1.
\end{equation}
Then
\begin{equation}\label{bo3}
\sup_{t\in[0,T]}\|e^{it\Lambda}U(t)-U(0)\|_{Y^{N_0}}\lesssim \varepsilon^2_0.
\end{equation}
\end{proposition}

We use first the method of normal forms to derive several new formulas describing the solution $U$. Then we use these formulas to prove the desired estimate \eqref{bo3}.

\subsection{Renormalizations}\label{renorm}

Assume that $U\in X^N_T$ satisfies the hypothesis of Proposition \ref{MainBootstrap} and let
\begin{equation*}
U_+:=U,\qquad U_-:=\overline{U}.
\end{equation*}
Let $\widehat{U_+},\widehat{U_-}$ denote the spatial Fourier transforms,
\begin{equation*}
\widehat{U_\mu}(\xi,t):=\int_{\mathbb{R}^2}U_\mu(x,t)e^{-ix\cdot\xi}\,dx,\qquad \mu\in\{+,-\},\,(\xi,t)\in\mathbb{R}^2\times [0,T].
\end{equation*}
The equation \eqref{bo1} gives
\begin{equation}\label{norm2}
[\partial_t+i\Lambda(\xi)]\widehat{U_+}(\xi,t)=\sum_{(\mu,\nu)\in\{(+,+),(+,-),(-,-)\}}\int_{\mathbb{R}^2}m_{\mu\nu}(\xi,\eta)\widehat{U_\mu}(\xi-\eta,t)\widehat{U_\nu}(\eta,t)\,d\eta,
\end{equation}
where
\begin{equation}\label{ms1}
m_{++}(\xi,\eta):=c_0i\Big[\frac{-\Lambda(\xi)|\xi-\eta|(\xi\cdot\eta)}{4\Lambda(\xi-\eta)|\xi||\eta|}+\frac{-|\xi|(\eta\cdot (\xi-\eta))}{8|\eta||\xi-\eta|}\Big],
\end{equation}
\begin{equation}\label{ms2}
m_{+-}(\xi,\eta):=c_0i\Big[\frac{-\Lambda(\xi)|\eta|(\xi\cdot(\xi-\eta))}{4\Lambda(\eta)|\xi||\xi-\eta|}+\frac{\Lambda(\xi)|\xi-\eta|(\xi\cdot\eta)}{4\Lambda(\xi-\eta)|\xi||\eta|}+\frac{|\xi|(\eta\cdot (\xi-\eta))}{4|\eta||\xi-\eta|}\Big],
\end{equation}
\begin{equation}\label{ms3}
m_{--}(\xi,\eta):=c_0i\Big[\frac{\Lambda(\xi)|\xi-\eta|(\xi\cdot\eta)}{4\Lambda(\xi-\eta)|\xi||\eta|}+\frac{-|\xi|(\eta\cdot (\xi-\eta))}{8|\eta||\xi-\eta|}\Big].
\end{equation}
Letting
\begin{equation*}
\begin{split}
&V_\mu(t):=e^{\mu it\Lambda}U_\mu(t),\qquad \widehat{U_\mu}(\xi,t)=\widehat{V_\mu}(\xi,t)e^{-\mu it\Lambda(\xi)},
\end{split}
\end{equation*}
for $\mu\in\{+.-\}$, the equation \eqref{norm2} is equivalent to
\begin{equation}\label{norm3}
\frac{d}{dt}[\widehat{V_+}(\xi,t)]=\sum_{(\mu,\nu)\in\{(+,+),(+,-),(-,-)\}}\int_{\mathbb{R}^2}e^{it[\Lambda(\xi)-\mu\Lambda(\xi-\eta)-\nu\Lambda(\eta)]}m_{\mu\nu}(\xi,\eta)\widehat{V_\mu}(\xi-\eta,t)\widehat{V_\nu}(\eta,t)\,d\eta.
\end{equation}
Therefore, letting
\begin{equation}\label{keyphi} 
\Phi_{\mu\nu}(\xi,\eta):=\Lambda(\xi)-\mu\Lambda(\xi-\eta)-\nu\Lambda(\eta),\qquad (\mu,\nu)\in\{(+,+),(+,-),(-,-)\},
\end{equation}
it follows that
\begin{equation*}
\widehat{V_+}(\xi,t)-\widehat{V_+}(\xi,0)=\sum_{(\mu,\nu)\in\{(+,+),(+,-),(-,-)\}}\int_0^t\int_{\mathbb{R}^2}e^{is\Phi_{\mu\nu}(\xi,\eta)}m_{\mu\nu}(\xi,\eta)\widehat{V_\mu}(\xi-\eta,s)\widehat{V_\nu}(\eta,s)\,d\eta ds.
\end{equation*}
It follows from the definition that, for any $(\mu,\nu)\in\{(+,+),(+,-),(-,-)\}$,
\begin{equation*}
|\Phi_{\mu\nu}(\xi,\eta)|^{-1}\lesssim 1+\min(|\xi|,|\eta|,|\xi-\eta|).
\end{equation*}
Therefore, we can integrate by parts in $s$ to conclude that, for any $t\in[0,T]$,
\begin{equation*}
\begin{split}
\widehat{V_+}(\xi,t)-\widehat{V_+}(\xi,0)&=\sum_{(\mu,\nu)\in\{(+,+),(+,-),(-,-)\}}-\int_0^t\int_{\mathbb{R}^2}e^{is\Phi_{\mu\nu}(\xi,\eta)}\frac{m_{\mu\nu}(\xi,\eta)}{i\Phi_{\mu\nu}(\xi,\eta)}\frac{d}{ds}\big[\widehat{V_\mu}(\xi-\eta,s)\widehat{V_\nu}(\eta,s)\big]\,d\eta ds\\
&+\sum_{(\mu,\nu)\in\{(+,+),(+,-),(-,-)\}}\int_{\mathbb{R}^2}e^{it\Phi_{\mu\nu}(\xi,\eta)}\frac{m_{\mu\nu}(\xi,\eta)}{i\Phi_{\mu\nu}(\xi,\eta)}\widehat{V_\mu}(\xi-\eta,t)\widehat{V_\nu}(\eta,t)\,d\eta\\
&-\sum_{(\mu,\nu)\in\{(+,+),(+,-),(-,-)\}}\int_{\mathbb{R}^2}\frac{m_{\mu\nu}(\xi,\eta)}{i\Phi_{\mu\nu}(\xi,\eta)}\widehat{V_\mu}(\xi-\eta,0)\widehat{V_\nu}(\eta,0)\,d\eta.
\end{split}
\end{equation*}
For $t\in[0,T]$ let
\begin{equation*}
\widehat{W}(\xi,t):=\widehat{V_+}(\xi,t)-\sum_{(\mu,\nu)\in\{(+,+),(+,-),(-,-)\}}\int_{\mathbb{R}^2}e^{it\Phi_{\mu\nu}(\xi,\eta)}\frac{m_{\mu\nu}(\xi,\eta)}{i\Phi_{\mu\nu}(\xi,\eta)}\widehat{V_\mu}(\xi-\eta,t)\widehat{V_\nu}(\eta,t)\,d\eta.
\end{equation*}
It follows that
\begin{equation}\label{norm4}
\begin{split}
&\widehat{W}(\xi,t)-\widehat{W}(\xi,0)=\sum_{(\mu,\nu)\in\{(+,+),(+,-),(-,-)\}}i\int_0^tH_{\mu\nu}(\xi,s)\,ds,\\
&H_{\mu\nu}(\xi,s):=\int_{\mathbb{R}^2}e^{is\Phi_{\mu\nu}(\xi,\eta)}\frac{m_{\mu\nu}(\xi,\eta)}{\Phi_{\mu\nu}(\xi,\eta)}\frac{d}{ds}\big[\widehat{V_\mu}(\xi-\eta,s)\widehat{V_\nu}(\eta,s)\big]\,d\eta.
\end{split}
\end{equation}

To calculate the functions $H_{\mu\nu}$ we use the formulas, see \eqref{norm3},
\begin{equation}\label{norm14}
\begin{split}
&\frac{d}{ds}[\widehat{V_+}(\xi,s)]=\sum_{(\mu,\nu)\in\{(+,+),(+,-),(-,-)\}}e^{is\Lambda(\xi)}\int_{\mathbb{R}^2}m_{\mu\nu}(\xi,\chi)\widehat{U_\mu}(\xi-\chi,s)\widehat{U_\nu}(\chi,s)\,d\chi,\\
&\frac{d}{ds}[\widehat{V_-}(\xi,s)]=\sum_{(\mu,\nu)\in\{(-,-),(-,+),(+,+)\}}e^{-is\Lambda(\xi)}\int_{\mathbb{R}^2}m'_{\mu\nu}(\xi,\chi)\widehat{U_\mu}(\xi-\chi,s)\widehat{U_\nu}(\chi,s)\,d\chi.
\end{split}
\end{equation}
where
\begin{equation*}
m'_{--}(v,w):=\overline{m_{++}(-v,-w)},\qquad m'_{-+}(v,w):=\overline{m_{+-}(-v,-w)},\qquad m'_{++}(v,w):=\overline{m_{--}(-v,-w)}.
\end{equation*}
Therefore
\begin{equation*}
\begin{split}
&e^{-is\Lambda(\xi)}H_{++}(\xi,s)\\
&=\sum_{(\mu,\nu)\in\{(+,+),(+,-),(-,-)\}}\int_{\mathbb{R}^2\times\mathbb{R}^2}\frac{m_{++}(\xi,\eta)}{\Phi_{++}(\xi,\eta)}\big[\widehat{U_+}(\eta,s)m_{\mu\nu}(\xi-\eta,\chi)\widehat{U_\mu}(\xi-\eta-\chi,s)\widehat{U_\nu}(\chi,s)\big]\,d\eta d\chi\\
&+\sum_{(\mu,\nu)\in\{(+,+),(+,-),(-,-)\}}\int_{\mathbb{R}^2\times\mathbb{R}^2}\frac{m_{++}(\xi,\eta)}{\Phi_{++}(\xi,\eta)}\big[\widehat{U_+}(\xi-\eta,s)m_{\mu\nu}(\eta,\chi)\widehat{U_\mu}(\eta-\chi,s)\widehat{U_\nu}(\chi,s)\big]\,d\eta d\chi,
\end{split}
\end{equation*}
\begin{equation*}
\begin{split}
&e^{-is\Lambda(\xi)}H_{+-}(\xi,s)\\
&=\sum_{(\mu,\nu)\in\{(+,+),(+,-),(-,-)\}}\int_{\mathbb{R}^2\times\mathbb{R}^2}\frac{m_{+-}(\xi,\eta)}{\Phi_{+-}(\xi,\eta)}\big[\widehat{U_-}(\eta,s)m_{\mu\nu}(\xi-\eta,\chi)\widehat{U_\mu}(\xi-\eta-\chi,s)\widehat{U_\nu}(\chi,s)\big]\,d\eta d\chi\\
&+\sum_{(\mu,\nu)\in\{(-,-),(-,+),(+,+)\}}\int_{\mathbb{R}^2\times\mathbb{R}^2}\frac{m_{+-}(\xi,\eta)}{\Phi_{+-}(\xi,\eta)}\big[\widehat{U_+}(\xi-\eta,s)m'_{\mu\nu}(\eta,\chi)\widehat{U_\mu}(\eta-\chi,s)\widehat{U_\nu}(\chi,s)\big]\,d\eta d\chi,
\end{split}
\end{equation*}
and
\begin{equation*}
\begin{split}
&e^{-is\Lambda(\xi)}H_{--}(\xi,s)\\
&=\sum_{(\mu,\nu)\in\{(-,-),(-,+),(+,+)\}}\int_{\mathbb{R}^2\times\mathbb{R}^2}\frac{m_{--}(\xi,\eta)}{\Phi_{--}(\xi,\eta)}\big[\widehat{U_-}(\eta,s)m'_{\mu\nu}(\xi-\eta,\chi)\widehat{U_\mu}(\xi-\eta-\chi,s)\widehat{U_\nu}(\chi,s)\big]\,d\eta d\chi\\
&+\sum_{(\mu,\nu)\in\{(-,-),(-,+),(+,+)\}}\int_{\mathbb{R}^2\times\mathbb{R}^2}\frac{m_{--}(\xi,\eta)}{\Phi_{--}(\xi,\eta)}\big[\widehat{U_-}(\xi-\eta,s)m'_{\mu\nu}(\eta,\chi)\widehat{U_\mu}(\eta-\chi,s)\widehat{U_\nu}(\chi,s)\big]\,d\eta d\chi.
\end{split}
\end{equation*}

After reorganizing the terms, it follows that
\begin{equation*}
\begin{split}
H_{++}(\xi,s)&+H_{+-}(\xi,s)+H_{--}(\xi,s)\\
&=e^{is\Lambda(\xi)}\int_{\mathbb{R}^2\times\mathbb{R}^2}m_{+++}(\xi,\eta,\chi)\widehat{U_+}(\xi-\eta,s)\widehat{U_+}(\eta-\chi,s)\widehat{U_+}(\chi,s)\,d\eta d\chi\\
&+e^{is\Lambda(\xi)}\int_{\mathbb{R}^2\times\mathbb{R}^2}m_{++-}(\xi,\eta,\chi)\widehat{U_+}(\xi-\eta,s)\widehat{U_+}(\eta-\chi,s)\widehat{U_-}(\chi,s)\,d\eta d\chi\\
&+e^{is\Lambda(\xi)}\int_{\mathbb{R}^2\times\mathbb{R}^2}m_{+--}(\xi,\eta,\chi)\widehat{U_+}(\xi-\eta,s)\widehat{U_-}(\eta-\chi,s)\widehat{U_-}(\chi,s)\,d\eta d\chi\\
&+e^{is\Lambda(\xi)}\int_{\mathbb{R}^2\times\mathbb{R}^2}m_{---}(\xi,\eta,\chi)\widehat{U_-}(\xi-\eta,s)\widehat{U_-}(\eta-\chi,s)\widehat{U_-}(\chi,s)\,d\eta d\chi,
\end{split}
\end{equation*}
where
\begin{equation}\label{mult1}
m_{+++}(\xi,\eta,\chi):=\frac{m_{++}(\xi,\chi)m_{++}(\xi-\chi,\eta-\chi)}{\Phi_{++}(\xi,\chi)}+\frac{m_{++}(\xi,\eta)m_{++}(\eta,\chi)}{\Phi_{++}(\xi,\eta)}+\frac{m_{+-}(\xi,\eta)m'_{++}(\eta,\chi)}{\Phi_{+-}(\xi,\eta)},
\end{equation}
\begin{equation}\label{mult2}
\begin{split}
&m_{++-}(\xi,\eta,\chi):=\frac{m_{++}(\xi,\xi-\eta)m_{+-}(\eta,\chi)}{\Phi_{++}(\xi,\xi-\eta)}+\frac{m_{++}(\xi,\eta)m_{+-}(\eta,\chi)}{\Phi_{++}(\xi,\eta)}+\frac{m_{+-}(\xi,\chi)m_{++}(\xi-\chi,\eta-\chi)}{\Phi_{+-}(\xi,\chi)}\\
&+\frac{m_{+-}(\xi,\eta)m'_{-+}(\eta,\eta-\chi)}{\Phi_{+-}(\xi,\eta)}+\frac{m_{--}(\xi,\chi)m'_{++}(\xi-\chi,\xi-\eta)}{\Phi_{--}(\xi,\chi)}+\frac{m_{--}(\xi,\xi-\chi)m'_{++}(\xi-\chi,\xi-\eta)}{\Phi_{--}(\xi,\xi-\chi)},
\end{split}
\end{equation}
\begin{equation}\label{mult3}
\begin{split}
&m_{+--}(\xi,\eta,\chi):=\frac{m_{++}(\xi,\xi-\eta)m_{--}(\eta,\chi)}{\Phi_{++}(\xi,\xi-\eta)}+\frac{m_{++}(\xi,\eta)m_{--}(\eta,\chi)}{\Phi_{++}(\xi,\eta)}+\frac{m_{+-}(\xi,\chi)m_{+-}(\xi-\chi,\eta-\chi)}{\Phi_{+-}(\xi,\chi)}\\
&+\frac{m_{+-}(\xi,\eta)m'_{--}(\eta,\eta-\chi)}{\Phi_{+-}(\xi,\eta)}+\frac{m_{--}(\xi,\chi)m'_{-+}(\xi-\chi,\xi-\eta)}{\Phi_{--}(\xi,\chi)}+\frac{m_{--}(\xi,\xi-\chi)m'_{-+}(\xi-\chi,\xi-\eta)}{\Phi_{--}(\xi,\xi-\chi)},
\end{split}
\end{equation}
and
\begin{equation}\label{mult4}
m_{---}(\xi,\eta,\chi):=\frac{m_{+-}(\xi,\chi)m_{--}(\xi-\chi,\eta-\chi)}{\Phi_{+-}(\xi,\chi)}+\frac{m_{--}(\xi,\chi)m'_{--}(\xi-\chi,\eta-\chi)}{\Phi_{--}(\xi,\chi)}+\frac{m_{--}(\xi,\eta)m'_{--}(\eta,\chi)}{\Phi_{--}(\xi,\eta)}.
\end{equation}

To summarize, we proved the following:

\begin{proposition}\label{formul}
For $t\in[0,T]$ let
\begin{equation*}
\begin{split}
&U_+(t)=U(t),\qquad U_-(t)=\overline{U(t)},\qquad V_+(t)=e^{it\Lambda}U_+(t),\qquad V_-(t)=e^{-it\Lambda}U_-(t),\\
&\widehat{W}(\xi,t)=e^{it\Lambda(\xi)}\widehat{U}(\xi,t)+ie^{it\Lambda(\xi)}\sum_{(\mu,\nu)\in\{(+,+),(+,-),(-,-)\}}\int_{\mathbb{R}^2}\frac{m_{\mu\nu}(\xi,\eta)}{\Phi_{\mu\nu}(\xi,\eta)}\widehat{U_\mu}(\xi-\eta,t)\widehat{U_\nu}(\eta,t)\,d\eta,
\end{split}
\end{equation*}
where $m_{++},m_{+-},m_{--}$ are defined in \eqref{ms1}--\eqref{ms3} and
\begin{equation*}
\Phi_{\mu\nu}(\xi,\eta)=\Lambda(\xi)-\mu\Lambda(\xi-\eta)-\nu\Lambda(\eta),\qquad (\mu,\nu)\in\{(+,+),(+,-),(-,-)\}.
\end{equation*}
Then
\begin{equation}\label{formul1}
\frac{d}{ds}[\widehat{V_+}(\xi,s)]=\sum_{(\mu,\nu)\in\{(+,+),(+,-),(-,-)\}}e^{is\Lambda(\xi)}\int_{\mathbb{R}^2}m_{\mu\nu}(\xi,\chi)\widehat{U_\mu}(\xi-\chi,s)\widehat{U_\nu}(\chi,s)\,d\chi,
\end{equation}
and
\begin{equation}\label{formul2}
\begin{split}
\widehat{W}(\xi,t)-\widehat{W}(\xi,0)&=i\int_0^t\int_{\mathbb{R}^2\times\mathbb{R}^2}e^{is\Lambda(\xi)}m_{+++}(\xi,\eta,\chi)\widehat{U_+}(\xi-\eta,s)\widehat{U_+}(\eta-\chi,s)\widehat{U_+}(\chi,s)\,d\eta d\chi ds\\
&+i\int_0^t\int_{\mathbb{R}^2\times\mathbb{R}^2}e^{is\Lambda(\xi)}m_{++-}(\xi,\eta,\chi)\widehat{U_+}(\xi-\eta,s)\widehat{U_+}(\eta-\chi,s)\widehat{U_-}(\chi,s)\,d\eta d\chi ds\\
&+i\int_0^t\int_{\mathbb{R}^2\times\mathbb{R}^2}e^{is\Lambda(\xi)}m_{+--}(\xi,\eta,\chi)\widehat{U_+}(\xi-\eta,s)\widehat{U_-}(\eta-\chi,s)\widehat{U_-}(\chi,s)\,d\eta d\chi ds\\
&+i\int_0^t\int_{\mathbb{R}^2\times\mathbb{R}^2}e^{is\Lambda(\xi)}m_{---}(\xi,\eta,\chi)\widehat{U_-}(\xi-\eta,s)\widehat{U_-}(\eta-\chi,s)\widehat{U_-}(\chi,s)\,d\eta d\chi ds,
\end{split}
\end{equation}
where the symbols $m_{+++},m_{++-},m_{+--},m_{---}$ are defined in \eqref{mult1}--\eqref{mult4}.
\end{proposition}

\subsection{Estimates}\label{estim} In this subsection we prove the bound \eqref{bo3}, as a consequence of Lemma \ref{ble1} and Lemma \ref{ble2} below. Using the dispersive estimate
\begin{equation}\label{disper}
\|P_le^{-it\Lambda}f\|_{L^\infty}\lesssim (1+t)^{-1}(1+2^{2l})\|f\|_{L^1},\qquad t\in\mathbb{R},
\end{equation}
and the definition of the $Z$-norm, it follows that for any $l\in\mathbb{Z}$ and $t\in [0,T]$
\begin{equation*}
\|P_lU_{\pm}(t)\|_{L^\infty}\lesssim (1+t)^{-1}(1+2^{2l})\|P_lV_{\pm}\|_{L^1}\lesssim (1+t)^{-1}\frac{1+2^{2l}}{2^{l/10}+2^{10l}}\varepsilon_0.
\end{equation*}
Moreover
\begin{equation*}
\|\widehat{P_lU_{\pm}}(t)\|_{L^\infty}=\|\widehat{P_lV_{\pm}}(t)\|_{L^\infty}\lesssim \|P_lV_{\pm}\|_{L^1}\lesssim (2^{l/10}+2^{10l})^{-1}\varepsilon_0.
\end{equation*}
Therefore, using the assumption $\|U\|_{X_T^N}\leq\varepsilon_0$ and the definitions \eqref{sec4}--\eqref{sec5}, for any $t\in[0,T]$
\begin{equation}\label{lo7}
\begin{split}
&\|P_lU_{\pm}(t)\|_{L^2}+\|P_lV_{\pm}(t)\|_{L^2}\lesssim\varepsilon_0 \min(2^{-N_0l},(1+t)^{\delta}2^{-Nl}),\qquad l\geq 0,\\
&\|P_lU_{\pm}(t)\|_{L^2}+\|P_lV_{\pm}(t)\|_{L^2}\lesssim\varepsilon_0 2^{9l/10},\qquad l\leq 0,\\
&(2^{l/10}+2^{8l})\|P_lU_{\pm}(t)\|_{L^\infty}\lesssim \varepsilon_0\min(2^{2l},(1+t)^{-1}),\qquad l\in\mathbb{Z},\\
&\|\widehat{P_lU_{\pm}}(t)\|_{L^\infty}+\|\widehat{P_lV_{\pm}}(t)\|_{L^\infty}\lesssim \varepsilon_0(2^{l/10}+2^{10l})^{-1},\qquad l\in\mathbb{Z}.
\end{split}
\end{equation}

\begin{lemma}\label{ble1}
We have
\begin{equation*}
\sup_{t\in[0,T]}\|e^{it\Lambda}U(t)-W(t)\|_{Y^{N_0}}\lesssim \varepsilon_0^2.
\end{equation*}
\end{lemma}

\begin{proof}[Proof of Lemma \ref{ble1}] Using the definitions, see Lemma \ref{formul}, it suffices to prove that for any $(\mu,\nu)\in\{(+,+),(+,-),(-,-)\}$
\begin{equation}\label{lo1}
\sup_{t\in[0,T]}\big\|F_{\mu\nu}(t)\big\|_{H^{N_0}}+\sup_{t\in[0,T]}\big\|F_{\mu\nu}(t)\big\|_{Z}\lesssim\varepsilon_0^2,
\end{equation}
where
\begin{equation*}
\widehat{F_{\mu\nu}}(\xi,t):=e^{it\Lambda(\xi)}\int_{\mathbb{R}^2}\frac{m_{\mu\nu}(\xi,\eta)}{\Phi_{\mu\nu}(\xi,\eta)}\widehat{U_\mu}(\xi-\eta,t)\widehat{U_\nu}(\eta,t)\,d\eta.
\end{equation*}

{\bf{Step 1.}} We prove first certain $L^2$ estimates. For any $k\in\mathbb{Z}$ and $t\in[0,T]$ we have
\begin{equation}\label{lo2}
\begin{split}
\widehat{P_kF_{\mu\nu}}(\xi,t)&=\varphi_k(\xi)e^{it\Lambda(\xi)}\int_{\mathbb{R}^2}\frac{m_{\mu\nu}(\xi,\eta)}{\Phi_{\mu\nu}(\xi,\eta)}\widehat{U_\mu}(\xi-\eta,t)\widehat{U_\nu}(\eta,t)\,d\eta\\
&=\sum_{(k_1,k_2)\in\mathcal{X}_k^1\cup\mathcal{X}_k^2}\varphi_k(\xi)e^{it\Lambda(\xi)}\int_{\mathbb{R}^2}\frac{m_{\mu\nu}(\xi,\eta)}{\Phi_{\mu\nu}(\xi,\eta)}\varphi_{k_1}(\xi-\eta)\widehat{U_\mu}(\xi-\eta,t)\varphi_{k_2}(\eta)\widehat{U_\nu}(\eta,t)\,d\eta\\
&=\sum_{(k_1,k_2)\in\mathcal{X}_k^1\cup\mathcal{X}_k^2}e^{it\Lambda(\xi)}\int_{\mathbb{R}^2}P_{k,k_1,k_2}^{\mu\nu}(\xi,\eta)\widehat{P_{k_1}U_\mu}(\xi-\eta,t)\widehat{P_{k_2}U_\nu}(\eta,t)\,d\eta
\end{split}
\end{equation}
where, as before,
\begin{equation*}
\begin{split}
&\mathcal{X}^1_k=\{(k_1,k_2)\in\mathbb{Z}\times\mathbb{Z}:|\max(k_1,k_2)-k|\leq 4\},\\
&\mathcal{X}^2_k=\{(k_1,k_2)\in\mathbb{Z}\times\mathbb{Z}:\max(k_1,k_2)\geq k+4\text{ and }|k_1-k_2|\leq 4\},
\end{split}
\end{equation*}
and
\begin{equation}\label{km2}
P_{k,k_1,k_2}^{\mu\nu}(\xi,\eta):=\varphi_k(\xi)\frac{m_{\mu\nu}(\xi,\eta)}{\Phi_{\mu\nu}(\xi,\eta)}\varphi_{[k_1-1,k_1+1]}(\xi-\eta)\varphi_{[k_2-1,k_2+1]}(\eta).
\end{equation}

We use the estimates \eqref{lo7} and \eqref{km1}: for any $k\leq 0$ and $t\in[0,T]$ it follows that
\begin{equation}\label{lo15}
\begin{split}
\|P_kF_{\mu\nu}(t)\|_{L^2}&\lesssim\sum_{(k_1,k_2)\in\mathcal{X}_k^2}(1+2^{3k_1})2^{k_1}\|P_{k_1}U_\pm(t)\|_{L^2}
\min(2^k\|P_{k_2}U_\pm(t)\|_{L^2},\|P_{k_2}U_\pm(t)\|_{L^\infty})\\
&+\sum_{(k_1,k_2)\in\mathcal{X}_k^1,\,k_1\leq k_2}2^k\|P_{k_1}U_\pm(t)\|_{L^2}\|P_{k_2}U_\pm(t)\|_{L^\infty}\\
&\lesssim \varepsilon_0^2\min((1+t)^{-1},2^k).
\end{split}
\end{equation}
Similarly, for any $k\geq 0$ and $t\in[0,T]$,
\begin{equation}\label{lo16}
\begin{split}
\|P_kF_{\mu\nu}(t)\|_{L^2}&\lesssim\sum_{(k_1,k_2)\in\mathcal{X}_k^2}2^{2k}(1+2^{3k_1})2^{k_1}\|P_{k_1}U_\pm(t)\|_{L^\infty}\|P_{k_2}U_\pm(t)\|_{L^2}\\
&+\sum_{(k_1,k_2)\in\mathcal{X}_k^1,\,-20k_2\leq k_1\leq k_2}2^{2k}(1+2^{3k_1})2^k\|P_{k_1}U_\pm(t)\|_{L^\infty}\|P_{k_2}U_\pm(t)\|_{L^2}\\
&+\sum_{(k_1,k_2)\in\mathcal{X}_k^1,\,k_1\leq -20k_2}2^{2k}(1+2^{3k_1})2^k\|P_{k_1}U_\pm(t)\|_{L^2}\|P_{k_2}U_\pm(t)\|_{L^\infty}\\
&\lesssim \varepsilon_0^2\min(2^{-(N-3)k},(1+t)^{-1}2^{-15k}).
\end{split}
\end{equation}
The bound on the first term in \eqref{lo1} follows easily from the last two inequalities.

{\bf{Step 2.}} We prove now the bound on the second term in \eqref{lo1}, i.e. for any $k\in\mathbb{Z}$ and $t\in[0,T]$
\begin{equation*}
\sum_{j\in\mathbb{Z}_+}2^j\|\varphi_j\cdot P_kF_{\mu\nu}(t)\|_{L^2}\lesssim\varepsilon_0^2(2^{k/10}+2^{10k})^{-1}.
\end{equation*}
Using \eqref{lo15} and \eqref{lo16}, it suffices to prove that
\begin{equation*}
\sum_{j\geq J_{k,t}}2^j\|\varphi_j\cdot P_kF_{\mu\nu}(t)\|_{L^2}\lesssim\varepsilon_0^22^{-10k_+},
\end{equation*}
where, for some sufficiently large constant $D$,
\begin{equation}\label{lo16.5}
2^{J_{k,t}}=D(1+t+2^{-41k/40})(1+2^{2k}).
\end{equation}
Using \eqref{lo2} and the formulas $\widehat{U_\mu}(\xi,t)=e^{-\mu it\Lambda(\xi)}\widehat{V_\mu}(\xi,t)$ we write
\begin{equation}\label{lu3.5}
\begin{split}
&P_k F_{\mu\nu}(t)=\sum_{(k_1,k_2)\in \mathcal{X}_k}F_{k,k_1,k_2}^{\mu\nu}(t),\\
&\widehat{F_{k,k_1,k_2}^{\mu\nu}}(\xi,t):=\int_{\mathbb{R}^2}e^{it[\Lambda(\xi)-\mu\Lambda(\xi-\eta)-\nu\Lambda(\eta)]}P_{k,k_1,k_2}^{\mu\nu}(\xi,\eta)\widehat{P_{k_1}V_\mu}(\xi-\eta,t)\widehat{P_{k_2}V_\nu}(\eta,t)\,d\eta,
\end{split}
\end{equation}
where, as before, $\mathcal{X}_k=\mathcal{X}_k^1\cup\mathcal{X}_k^2$. Using this decomposition, it suffices to prove that
\begin{equation}\label{lo21}
\sum_{j\geq J_{k,t}}\sum_{(k_1,k_2)\in \mathcal{X}_k}2^j\|\varphi_j\cdot F_{k,k_1,k_2}^{\mu\nu}(t)\|_{L^2}\lesssim\varepsilon_0^22^{-10k_+},
\end{equation}
 for any $k\in\mathbb{Z}$ and $t\in[0,T]$. This is similar to the proof of Proposition \ref{LocExLemma} (iii) in the previous section, and follows from the bounds \eqref{lu4}, \eqref{lu7}, and \eqref{lu15} below.
  
{\bf{Step 3.}} Since
\begin{equation*}
\sup_{\xi,\eta\in\mathbb{R}^2}|P_{k,k_1,k_2}^{\mu\nu}(\xi,\eta)|\lesssim (1+2^{\min(k_1,k_2)})2^{\max(k_1,k_2)}, 
\end{equation*}
it follows from \eqref{lo7} that
\begin{equation*}
\begin{split}
\|F_{k,k_1,k_2}^{\mu\nu}(s)\|_{L^2}&\lesssim (2^{\max(k_1,k_2)}+2^{k_1+k_2})\min\big(\|\widehat{P_{k_1}V_\mu}(s)\|_{L^1}\|\widehat{P_{k_2}V_\nu}(s)\|_{L^2},\|\widehat{P_{k_1}V_\mu}(s)\|_{L^2}\|\widehat{P_{k_2}V_\nu}(s)\|_{L^1}\big)\\
&\lesssim\varepsilon_0^2(2^{\max(k_1,k_2)}+2^{k_1+k_2})2^{(19/10)\min(k_1,k_2)}(1+2^{\max(k_1,k_2)})^{-N_0}.
\end{split}
\end{equation*}
Therefore
\begin{equation}\label{lu4}
\sum_{j\geq J_{k,t}}\sum_{(k_1,k_2)\in \mathcal{X}_k,\,\min(k_1,k_2)\leq-2j/3}2^j\|F_{k,k_1,k_2}^{\mu\nu}(s)\|_{L^2}\lesssim \varepsilon_0^22^{-10 k_+}.
\end{equation}

{\bf{Step 4.}} It remains to bound the contribution
\begin{equation*}
\sum_{j\geq J_{k,t}}\sum_{(k_1,k_2)\in \mathcal{X}_k,\,\min(k_1,k_2)\geq-2j/3}2^{j}\|\varphi_j\cdot F_{k,k_1,k_2}^{\mu\nu}(t)\|_{L^2}.
\end{equation*}
For this we write, using \eqref{lu3.5},
\begin{equation*}
\begin{split}
&F_{k,k_1,k_2}^{\mu\nu}(x,t)=\int_{\mathbb{R}^2\times\mathbb{R}^2}P_{k_1}V_\mu(y,t)P_{k_2}V_\nu(z,t)L_{k,k_1,k_2}^{\mu\nu}(x,y,z,t)\,dydz,\\
&L_{k,k_1,k_2}^{\mu\nu}(x,y,z,t):=c\int_{\mathbb{R}^2\times\mathbb{R}^2}e^{it[\Lambda(\xi)-\mu\Lambda(\xi-\eta)-\nu\Lambda(\eta)]}e^{i(x\cdot\xi-y\cdot(\xi-\eta)-z\cdot\eta)}P_{k,k_1,k_2}^{\mu\nu}(\xi,\eta)\,d\xi d\eta.
\end{split}
\end{equation*}
Since
\begin{equation*}
\sup_{\xi\in\mathbb{R}^2}|D^\alpha \Lambda(\xi)|\lesssim_{|\alpha|}1
\end{equation*}
it follows by integration by parts that
\begin{equation}\label{lu5}
\begin{split}
&|L_{k,k_1,k_2}^{\mu\nu}(x,y,z,t)|\lesssim (1+2^{4k_1}+2^{4k_2})(|x-y|+|x-z|)^{-20}\\
\text{if }&2^j\geq D(1+t),\,\,\min(k,k_1,k_2)\geq -40j/41,\,\,|x-y|+|x-z|\geq 2^{j-4},
\end{split}
\end{equation}
provided that the constant $D$ is fixed sufficiently large. Letting
\begin{equation*}
\begin{split}
&G_{j,k,k_1,k_2}^{\mu\nu}(x,t):=\int_{\mathbb{R}^2\times\mathbb{R}^2}V^\mu_{j,k_1}(y,t)V^\nu_{j,k_2}(z,t)L_{k,k_1,k_2}^{\mu\nu}(x,y,z,t)\,dydz,\\
&V^\mu_{j,l}(x,t)=\varphi_{[j-4,j+4]}(x)\cdot P_{l}V_\mu(x,t),
\end{split}
\end{equation*}
it follows from \eqref{cu5} that
\begin{equation*}
\|\varphi_j\cdot [F_{k,k_1,k_2}^{\mu\nu}(t)-G_{j,k,k_1,k_2}^{\mu\nu}(t)]\|_{L^2}\lesssim 2^{-10j}(1+2^{4k_1}+2^{4k_2})\|P_{k_1}V_+(t)\|_{L^2}\|P_{k_2}V_+(t)\|_{L^2}.
\end{equation*}
Since $\|P_{l}V_+(t)\|_{L^2}\lesssim 2^{-N_0l_+}\|V_+(t)\|_{H^{N_0}}$, $l\in\mathbb{Z}$, it follows that
\begin{equation}\label{lu7}
\sum_{j\geq J_{k,t}}\sum_{(k_1,k_2)\in \mathcal{X}_k,\,\min(k_1,k_2)\geq-2j/3}2^j\|\varphi_j\cdot [F_{k,k_1,k_2}^{\mu\nu}(t)-G_{j,k,k_1,k_2}^{\mu\nu}(t)]\|_{L^2}\leq\varepsilon_0^22^{-10 k_+}.
\end{equation}

{\bf{Step 5.}} Finally we rewrite $G_{j,k,k_1,k_2}^{\mu\nu}(t)$ in the frequency space,
\begin{equation*}
\widehat{G_{j,k,k_1,k_2}^{\mu\nu}}(\xi,t)=\int_{\mathbb{R}^2}e^{it[\Lambda(\xi)-\mu\Lambda(\xi-\eta)-\nu\Lambda(\eta)]}P^{\mu\nu}_{k,k_1,k_2}(\xi,\eta)\widehat{V^\mu_{j,k_1}}(\xi-\eta,t)\widehat{V_{j,k_2}^\nu}(\eta,t)\,d\eta,
\end{equation*}
and estimate
\begin{equation*}
\|\widehat{G_{j,k,k_1,k_2}^{\mu\nu}}(t)\|_{L^2}\lesssim \|P^{\mu\nu}_{k,k_1,k_2}\|_{L^\infty}2^{\min(k_1,k_2)}\|\widehat{V^\mu_{j,k_1}}(t)\|_{L^2}\|\widehat{V_{j,k_2}^\nu}(t)\|_{L^2}.
\end{equation*}
Therefore
\begin{equation}\label{lu12}
\begin{split}
\sum_{j\geq J_{k,t}}&\sum_{(k_1,k_2)\in \mathcal{X}_k,\,\min(k_1,k_2)\geq-2j/3}2^{j}\|\varphi_j\cdot G_{j,k,k_1,k_2}^{\mu\nu}(t)\|_{L^2}\\
&\lesssim\sum_{(k_1,k_2)\in \mathcal{X}_k,\,k_1\leq k_2}\sum_{j\geq 10}2^{j}2^{k_1+k_2}(1+2^{\min(k_1,k_2)})\|V^\pm_{j,k_1}(t)\|_{L^2}\|V_{j,k_2}^\pm(t)\|_{L^2}.
\end{split}
\end{equation}
We estimate the sum in the right-hand side of \eqref{cu12} in two different ways, depending on the relative sizes of $k,k_1,k_2$. If $k_1=\min(k_1,k_2)\notin[100-10k,k-10]$ then we estimate $\|V^\pm_{j,k_1}(t)\|_{L^2}\lesssim \varepsilon_02^{-N_0\max(k_1,0)}$, and the corresponding sum in the right-hand side of \eqref{lu12} is bounded by
\begin{equation*}
C\sum_{(k_1,k_2)\in \mathcal{X}_k,\,k_1\leq k_2,\,k_1\notin[100-10k,k-10]}\varepsilon_02^{k_1+k_2}2^{(1-N_0)\max(k_1,0)}(2^{k_2/10}+2^{10k_2})^{-1}\|V_+(t)\|_{Z}\lesssim \varepsilon_0^22^{-10 k_+}.
\end{equation*}
On the other hand, if $k_1=\min(k_1,k_2)\in[100-10k,k-10]$ (in particular $k\geq 10$, $|k_2-k|\leq 4$) then we use $\|V^\pm_{j,k_2}(t)\|_{L^2}\lesssim \varepsilon_0 2^{-N_0k}$, and the corresponding sum in the right-hand side of \eqref{cu12} is bounded by
\begin{equation*}
C\sum_{(k_1,k_2)\in \mathcal{X}_k,\,k_1\leq k_2,\,k_1\in[100-10k,k-10]}\varepsilon_02^{k_1+k_2}2^{(1-N_0)k}(2^{k_1/10}+2^{10k_1})^{-1}\|V_+(t)\|_{Z}\lesssim \varepsilon_0^22^{-10 k}.
\end{equation*}
Therefore
\begin{equation}\label{lu15}
\sum_{j\geq J_{k,t}}\sum_{(k_1,k_2)\in \mathcal{X}_k,\,\min(k_1,k_2)\geq-2j/3}2^j\|\varphi_j\cdot G_{j,k,k_1,k_2}^{\mu\nu}(t)\|_{L^2}\lesssim \varepsilon_0^22^{-10 k_+},
\end{equation}
which completes the proof.
\end{proof}

\begin{lemma}\label{ble2}
We have
\begin{equation*}
\sup_{t\in[0,T]}\|W(t)-W(0)\|_{Y^{N_0}}\lesssim \varepsilon_0^3.
\end{equation*}
\end{lemma}

\begin{proof}[Proof of Lemma \ref{ble2}] {\bf{Step 1.}} We start again by proving $L^2$ estimates. Using \eqref{formul2}, for any $t\in[0,T]$ and $k\in\mathbb{Z}$ we write
\begin{equation}\label{lo30}
\mathcal{F}(P_k(W(t)-W(0)))(\xi)=\sum_{(\mu,\nu,\sigma)\in \{(+,+,+),(+,+,-),(+,-,-),(-,-,-)\}}i\int_0^t\widehat{P_kF_{\mu\nu\sigma}}(\xi,s)\,ds,
\end{equation}
where
\begin{equation}\label{lo30.5}
\begin{split}
&\widehat{P_kF_{\mu\nu\sigma}}(\xi,s):=\varphi_k(\xi)e^{is\Lambda(\xi)}\int_{\mathbb{R}^2\times\mathbb{R}^2}m_{\mu\nu\sigma}(\xi,\eta,\chi)\widehat{U_\mu}(\xi-\eta,s)\widehat{U_\nu}(\eta-\chi,s)\widehat{U_\sigma}(\chi,s)\,d\eta d\chi\\
&=\sum_{k_1,k_2,k_3\in\mathbb{Z}}e^{is\Lambda(\xi)}\int_{\mathbb{R}^2\times\mathbb{R}^2}P_{k,k_1,k_2,k_3}^{\mu\nu\sigma}(\xi,\eta,\chi)\widehat{P_{k_1}U_\mu}(\xi-\eta,s)\widehat{P_{k_2}U_\nu}(\eta-\chi,s)\widehat{P_{k_3}U_\sigma}(\chi,s)\,d\eta d\chi,
\end{split}
\end{equation}
and
\begin{equation}\label{lo31}
P_{k,k_1,k_2,k_3}^{\mu\nu\sigma}(\xi,\eta,\chi):=\varphi_k(\xi)m_{\mu\nu\sigma}(\xi,\eta,\chi)\varphi_{[k_1-1,k_1+1]}(\xi-\eta)\varphi_{[k_2-1,k_2+1]}(\eta-\chi)\varphi_{[k_3-1,k_3+1]}(\chi).
\end{equation}
Let
\begin{equation*}
\begin{split}
&\mathcal{Y}_k^1:=\{(k_1,k_2,k_3)\in\mathbb{Z}\times\mathbb{Z}\times\mathbb{Z}:|\max(k_1,k_2,k_3)-k|\leq 4\},\\
&\mathcal{Y}_k^2:=\{(k_1,k_2,k_3)\in\mathbb{Z}\times\mathbb{Z}\times\mathbb{Z}:\max(k_1,k_2,k_3)\geq k+4,\,\max(k_1,k_2,k_3)-\mathrm{med}(k_1,k_2,k_3)\leq 4\},\\
&\mathcal{Y}_k:=\mathcal{Y}_k^1\cup\mathcal{Y}_k^2.
\end{split}
\end{equation*}

Using Lemma \ref{tech3} and \eqref{lo7} we estimate for $k\leq 0$ and $s\in[0,T]$,
\begin{equation*}
\begin{split}
\|P_kF_{\mu\nu\sigma}(s)\|_{L^2}&\lesssim \sum_{(k_1,k_2,k_3)\in \mathcal{Y}_k^2,\,k_1\leq k_2\leq k_3}(2^{2k_3}+2^{7k_3})\\
&\big[2^{2k_1+10k}(1+s)^{-10}\|P_{k_1}U_{\pm}(s)\|_{L^2}\|P_{k_2}U_{\pm}(s)\|_{L^2}\|P_{k_3}U_{\pm}(s)\|_{L^2}\\
&+(\log(2+s)+|k|)\|P_{k_1}U_{\pm}(s)\|_{L^2}\|P_{k_2}U_{\pm}(s)\|_{L^\infty}\min(2^k\|P_{k_3}U_{\pm}(s)\|_{L^2},\|P_{k_3}U_{\pm}(s)\|_{L^\infty})\big]\\
&+\sum_{(k_1,k_2,k_3)\in \mathcal{Y}_k^1,\,k_1\leq k_2\leq k_3}2^{2k}\big[2^{2k_1}(1+s)^{-10}
\|P_{k_1}U_{\pm}(s)\|_{L^2}\|P_{k_2}U_{\pm}(s)\|_{L^2}\|P_{k_3}U_{\pm}(s)\|_{L^2}\\
&+\log(2+s)\|P_{k_1}U_{\pm}(s)\|_{L^2}\|P_{k_2}U_{\pm}(s)\|_{L^\infty}\|P_{k_3}U_{\pm}(s)\|_{L^\infty}\big]\\
&\lesssim \varepsilon_0^3(1+s)^{-1}(\log(2+s)+|k|)\min((1+s)^{-1},2^k).
\end{split}
\end{equation*}
Similarly, if $k\geq 0$ and $s\in[0,T]$,
\begin{equation*}
\begin{split}
\|P_k&F_{\mu\nu\sigma}(s)\|_{L^2}\lesssim \sum_{(k_1,k_2,k_3)\in \mathcal{Y}_k^2,\,k_1\leq k_2\leq k_3}2^{7k_3}\big[2^{2k_1}(1+s)^{-10}
\|P_{k_1}U_{\pm}(s)\|_{L^2}\|P_{k_2}U_{\pm}(s)\|_{L^2}\|P_{k_3}U_{\pm}(s)\|_{L^2}\\
&+\log(2+s)\|P_{k_1}U_{\pm}(s)\|_{L^\infty}\|P_{k_2}U_{\pm}(s)\|_{L^\infty}\|P_{k_3}U_{\pm}(s)\|_{L^2}\big]\\
&+\sum_{(k_1,k_2,k_3)\in \mathcal{Y}_k^1,\,k_1\leq k_2\leq k_3}2^{4k}(1+2^{3k_2})\big[2^{2k_1}(1+s)^{-10}
\|P_{k_1}U_{\pm}(s)\|_{L^2}\|P_{k_2}U_{\pm}(s)\|_{L^2}\|P_{k_3}U_{\pm}(s)\|_{L^2}\\
&+\log(2+s)\|P_{k_1}U_{\pm}(s)\|_{L^\infty}\|P_{k_2}U_{\pm}(s)\|_{L^\infty}\|P_{k_3}U_{\pm}(s)\|_{L^2}\big]\\
&\lesssim \varepsilon_0^3(1+s)^{-9/5}2^{-(N-4)k}.
\end{split}
\end{equation*}
Therefore
\begin{equation}\label{lo32}
\begin{split}
&\|P_kF_{\mu\nu\sigma}(s)\|_{L^2}\lesssim \varepsilon_0^3(1+s)^{-1}(\log(2+s)+|k|)\min((1+s)^{-1},2^k),\qquad \text{ if } k\leq 0,\\
&\|P_kF_{\mu\nu\sigma}(s)\|_{L^2}\lesssim \varepsilon_0^3(1+s)^{-9/5}2^{-(N-4)k},\qquad \text{ if }k\geq 0.
\end{split}
\end{equation}
In particular, using \eqref{lo30}
\begin{equation*}
\|W(t)-W(0)\|_{H^{N_0}}\lesssim \varepsilon_0^3,\qquad t\in[0,T],
\end{equation*}
as desired.

{\bf{Step 2.}} We prove now the remaining bound on the $Z$ norm, i.e.
\begin{equation*}
\sum_{j\in\mathbb{Z}_+}2^j\|\varphi_j\cdot P_k(W(t)-W(0))\|_{L^2}\lesssim \varepsilon_0^3(2^{k/10}+2^{10k})^{-1},
\end{equation*}
for any $t\in[0,T]$ and $k\in\mathbb{Z}$ fixed. We fix a dyadic decomposition of the function $\mathbf{1}_{[0,t]}$, i.e. we fix functions $q_0,\ldots,q_{L+1}:\mathbb{R}\to[0,1]$, $|L-\log_2(2+t)|\leq 2$, with the properties
\begin{equation}\label{nh1}
\begin{split}
&\sum_{m=0}^{L+1}q_l(s)=\mathbf{1}_{[0,t]}(s),\qquad \mathrm{supp}\,q_0\subseteq [0,2], \qquad \mathrm{supp}\,q_{L+1}\subseteq [t-2,t],\qquad\mathrm{supp}\,q_m\subseteq [2^{m-1},2^{m+1}],\\
&q_m\in C^1(\mathbb{R})\qquad\text{ and }\qquad \int_0^t|q'_m(s)|\,ds\lesssim 1\qquad\text{ for }m=1,\ldots,L.
\end{split}
\end{equation}
Using the formula \eqref{lo30} and inserting the partition $\{q_m\}_{m=0,\ldots,L+1}$, it suffices to prove that
\begin{equation*}
\sum_{j\in\mathbb{Z}_+}2^j\Big\|\varphi_j\cdot \int_{\mathbb{R}}q_m(s)P_kF_{\mu\nu\sigma}(s)\,ds\Big\|_{L^2}\lesssim \varepsilon_0^3(2^{k/10}+2^{10k})^{-1}2^{-m/100}
\end{equation*}
for any $k\in\mathbb{Z}$, $m\in\{0,\ldots,L+1\}$, and any $(\mu,\nu,\sigma)\in \{(+,+,+),(+,+,-),(+,-,-),(-,-,-)\}$. In view of \eqref{lo32} and the support properties of the functions $q_m$, it remains to prove that
\begin{equation}\label{nh2}
\sum_{j\geq M(k,m)}2^j\Big\|\varphi_j\cdot \int_{\mathbb{R}}q_m(s)P_kF_{\mu\nu\sigma}(s)\,ds\Big\|_{L^2}\lesssim \varepsilon_0^32^{-10k_+}2^{-m/100},
\end{equation}
where, for some constant $D$ sufficiently large,
\begin{equation}\label{nh3}
M(k,m):=D+\max(7m/10,-21k/20,10k).
\end{equation}

{\bf{Step 3.}} We examine now the formula \eqref{lo30.5} and estimate first the contributions of very low and very high frequencies. More precisely, assume $j\geq M(k,m)$ is fixed. Using Lemma \ref{tech3} with $l=\min(k_1,k_2,k_3,-10j)-1$ and \eqref{lo7} we estimate, if $\min(k_1,k_2,k_3)\leq -2j/3$,
\begin{equation*}
\begin{split}
\Big\|\int_{\mathbb{R}^2\times\mathbb{R}^2}P_{k,k_1,k_2,k_3}^{\mu\nu\sigma}(\xi,\eta,\chi)&\widehat{P_{k_1}U_\mu}(\xi-\eta,s)\widehat{P_{k_2}U_\nu}(\eta-\chi,s)\widehat{P_{k_3}U_\sigma}(\chi,s)\,d\eta d\chi\Big\|_{L^2_\xi}\\
&\lesssim \varepsilon_0^3(1+2^{7c})\big(2^{-20j}2^{9b/10}2^{-N_0c_+}+j2^{9b/5}2^{-m}2^{-N_0c_+}\big)\\
&\lesssim \varepsilon_0^32^{-(N_0-7)c_+}2^{-m}2^{9b/10}2^{-j/2},
\end{split}
\end{equation*}
where $b=\min(k_1,k_2,k_3)$ and $c=\max(k_1,k_2,k_3)$.

Similarly, if $\max(k_1,k_2,k_3)\geq j/10$ and $\min(k_1,k_2,k_3)\geq -2j/3$ (in particular $(k_1,k_2,k_3)\in\mathcal{Y}_k^2$), we use Lemma \ref{tech3} with $l=-10j$ and \eqref{lo7} to estimate
\begin{equation*}
\begin{split}
\Big\|\int_{\mathbb{R}^2\times\mathbb{R}^2}P_{k,k_1,k_2,k_3}^{\mu\nu\sigma}(\xi,\eta,\chi)&\widehat{P_{k_1}U_\mu}(\xi-\eta,s)\widehat{P_{k_2}U_\nu}(\eta-\chi,s)\widehat{P_{k_3}U_\sigma}(\chi,s)\,d\eta d\chi\Big\|_{L^2_\xi}\\
&\lesssim \varepsilon_0^3(1+2^{7c})\big(2^{-20j}2^{9b/10}2^{-2N_0c_+}+j2^k2^{19b/10}2^{-2N_0c_+}\big)\\
&\lesssim \varepsilon_0^32^{-(2N_0-8)c}2^{9b/10}.
\end{split}
\end{equation*}
It follows easily that the corresponding sums are controlled as claimed in \eqref{nh2}. Therefore it remains to prove that for any $k\in\mathbb{Z}$, $m\in\{0,\ldots,L+1\}$, and $(\mu,\nu,\sigma)\in \{(+,+,+),(+,+,-),(+,-,-),(-,-,-)\}$,
\begin{equation}\label{nh6}
\sum_{j\geq M(k,m)}2^j\Big\|\varphi_j\cdot \int_{\mathbb{R}}q_m(s)P_kF^{\mu\nu\sigma}_j(s)\,ds\Big\|_{L^2}\lesssim \varepsilon_0^32^{-10k_+}2^{-m/1000},
\end{equation}
where, with $Q_j:=P_{[-2j/3,j/10]}$,
\begin{equation}\label{nh5}
\widehat{P_kF^{\mu\nu\sigma}_j}(\xi,s):=\varphi_k(\xi)e^{is\Lambda(\xi)}\int_{\mathbb{R}^2\times\mathbb{R}^2}m_{\mu\nu\sigma}(\xi,\eta,\chi)\widehat{Q_jU_\mu}(\xi-\eta,s)\widehat{Q_jU_\nu}(\eta-\chi,s)\widehat{Q_jU_\sigma}(\chi,s)\,d\eta d\chi.
\end{equation}

{\bf{Step 4.}} The claim \eqref{nh6} is proved in Lemma \ref{KerLem}, see \eqref{nh30}, if $(\mu,\nu,\sigma)=(+,+,-)$. In the remaining cases, using first the bound \eqref{nh29.5}, it remains to control the sum over $j\in(M(k,m),M'(k,m)]$. More precisely, we have to prove that
\begin{equation}\label{nh9}
\sum_{M(k,m)< j\leq M'(k,m)}2^j\Big\|\varphi_j\cdot \int_{\mathbb{R}}q_m(s)P_kF^{\mu\nu\sigma}_j(s)\,ds\Big\|_{L^2}\lesssim \varepsilon_0^32^{-10k_+}2^{-m/1000}
\end{equation}
for any $k\in\mathbb{Z}$, $m\in\{0,\ldots,L+1\}$, and $(\mu,\nu,\sigma)\in \{(+,+,+),(+,-,-),(-,-,-)\}$. Since $M(k,m)=M'(k,m)$ if $m\leq\max(-21k/20,10k)$, we may assume that
\begin{equation*}
m\geq \max(-21k/20,10k),\qquad M'(k,m)=D+m.
\end{equation*}

The bound \eqref{nh9} follows easily if $m=0$ or $m=L+1$, due to the support properties of $q_0$ and $q_{L+1}$. Indeed, if $m\in\{0,L+1\}$ then
\begin{equation*}
\sum_{M(k,m)< j\leq D+m}2^j\Big\|\varphi_j\cdot \int_{\mathbb{R}}q_m(s)P_kF^{\mu\nu\sigma}_j(s)\,ds\Big\|_{L^2}\lesssim 2^m\sup_{s\in[2^{m-1},2^{m+1}]}\sup_{M(k,m)< j\leq D+m}\|P_kF^{\mu\nu\sigma}_j(s)\|_{L^2}.
\end{equation*}
We use Lemma \ref{tech3} with $r=2$, $l=-100j$ and appropriate choices of $p_1,p_2,p_3$ to show that
\begin{equation*}
\|P_kF^{\mu\nu\sigma}_j(s)\|_{L^2}\lesssim (1+m)^62^{-2m}2^{-11k_+}.
\end{equation*}
Therefore, for \eqref{nh9} it remains to prove that if 
\begin{equation}\label{nh10}
m\geq \max(-21k/20,100k),\qquad k_1,k_2,k_3\in[-2(m+D)/3,(m+D)/10],\qquad m\in\{D,\ldots,L\},
\end{equation}
and
\begin{equation*}
\widehat{P_kF^{\mu\nu\sigma}_{k_1,k_2,k_3}}(\xi,s):=\varphi_k(\xi)e^{is\Lambda(\xi)}\int_{\mathbb{R}^2\times\mathbb{R}^2}m_{\mu\nu\sigma}(\xi,\eta,\chi)\widehat{P_{k_1}U_\mu}(\xi-\eta,s)\widehat{P_{k_2}U_\nu}(\eta-\chi,s)\widehat{P_{k_3}U_\sigma}(\chi,s)\,d\eta d\chi,
\end{equation*}
$(\mu,\nu,\sigma)\in \{(+,+,+),(+,-,-),(-,-,-)\}$, then
\begin{equation}\label{nh11}
\Big\|\int_{\mathbb{R}}q_m(s)\widehat{P_kF^{\mu\nu\sigma}_{k_1,k_2,k_3}}(s)\,ds\Big\|_{L^2}\lesssim \varepsilon_0^32^{-10k_+}2^{-101m/100}.
\end{equation}

{\bf{Step 5.}} We rewrite
\begin{equation*}
\widehat{P_kF^{\mu\nu\sigma}_{k_1,k_2,k_3}}(\xi,s)=\int_{\mathbb{R}^2\times\mathbb{R}^2}e^{is\Phi_{\mu\nu\sigma}(\xi,\eta,\chi)}P_{k,k_1,k_2,k_3}^{\mu\nu\sigma}(\xi,\eta,\chi)\widehat{P_{k_1}V_\mu}(\xi-\eta,s)\widehat{P_{k_2}V_\nu}(\eta-\chi,s)\widehat{P_{k_3}V_\sigma}(\chi,s)\,d\eta d\chi,
\end{equation*}
where $P_{k,k_1,k_2,k_3}^{\mu\nu\sigma}$ is as in \eqref{lo31}, and
\begin{equation}\label{nh17}
\Phi_{\mu\nu\sigma}(\xi,\eta,\chi):=\Lambda(\xi)-\mu\Lambda(\xi-\eta)-\nu\Lambda(\eta-\chi)-\sigma\Lambda(\chi).
\end{equation}
To prove \eqref{nh11} we integrate by parts in $s$; the key observation is that
\begin{equation}\label{nh20}
\Phi_{\mu\nu\sigma}(\xi,\eta,\chi)\gtrsim \frac{1}{1+\min(|\xi|,|\xi-\eta|,|\eta-\chi|,|\chi|)}\qquad\text{ for any }\xi,\eta,\chi\in\mathbb{R}^2,
\end{equation}
for $(\mu,\nu,\sigma)\in \{(+,+,+),(+,-,-),(-,-,-)\}$. This follows easily from \eqref{la2}. Therefore, for any $\xi\in\mathbb{R}^2$,
\begin{equation*}
\begin{split}
\int_{\mathbb{R}}q_m(s)&\widehat{P_kF^{\mu\nu\sigma}_{k_1,k_2,k_3}}(\xi,s)\,ds=c\int_{\mathbb{R}}\int_{\mathbb{R}^2\times\mathbb{R}^2}e^{is\Phi_{\mu\nu\sigma}(\xi,\eta,\chi)}\\
&\frac{P_{k,k_1,k_2,k_3}^{\mu\nu\sigma}(\xi,\eta,\chi)}{\Phi_{\mu\nu\sigma}(\xi,\eta,\chi)}\frac{d}{ds}\big[q_m(s)\widehat{P_{k_1}V_\mu}(\xi-\eta,s)\widehat{P_{k_2}V_\nu}(\eta-\chi,s)\widehat{P_{k_3}V_\sigma}(\chi,s)\big]\,d\eta d\chi ds.
\end{split}
\end{equation*}
Using the properties of the functions $q_m$ in \eqref{nh1}, it follows that the left-hand side of \eqref{nh11} is bounded by $I+II+III+IV$ where
\begin{equation*}
\begin{split}
&I:=\sup_{s\in\mathrm{supp}\,q_m}\Big\|\int_{\mathbb{R}^2\times\mathbb{R}^2}\frac{P_{k,k_1,k_2,k_3}^{\mu\nu\sigma}(\xi,\eta,\chi)}{\Phi_{\mu\nu\sigma}(\xi,\eta,\chi) }\widehat{P_{k_1}U_\mu}(\xi-\eta,s)\widehat{P_{k_2}U_\nu}(\eta-\chi,s)\widehat{P_{k_3}U_\sigma}(\chi,s)\,d\eta d\chi\Big\|_{L^2_\xi},\\
&II:=2^m\sup_{s\in\mathrm{supp}\,q_m}\Big\|\int_{\mathbb{R}^2\times\mathbb{R}^2}\frac{P_{k,k_1,k_2,k_3}^{\mu\nu\sigma}(\xi,\eta,\chi)}{\Phi_{\mu\nu\sigma}(\xi,\eta,\chi) }\widehat{P_{k_1}Z_\mu}(\xi-\eta,s)\widehat{P_{k_2}U_\nu}(\eta-\chi,s)\widehat{P_{k_3}U_\sigma}(\chi,s)\,d\eta d\chi\Big\|_{L^2_\xi},\\
&III:=2^m\sup_{s\in\mathrm{supp}\,q_m}\Big\|\int_{\mathbb{R}^2\times\mathbb{R}^2}\frac{P_{k,k_1,k_2,k_3}^{\mu\nu\sigma}(\xi,\eta,\chi)}{\Phi_{\mu\nu\sigma}(\xi,\eta,\chi) }\widehat{P_{k_1}U_\mu}(\xi-\eta,s)\widehat{P_{k_2}Z_\nu}(\eta-\chi,s)\widehat{P_{k_3}U_\sigma}(\chi,s)\,d\eta d\chi\Big\|_{L^2_\xi},\\
&IV:=2^m\sup_{s\in\mathrm{supp}\,q_m}\Big\|\int_{\mathbb{R}^2\times\mathbb{R}^2}\frac{P_{k,k_1,k_2,k_3}^{\mu\nu\sigma}(\xi,\eta,\chi)}{\Phi_{\mu\nu\sigma}(\xi,\eta,\chi) }\widehat{P_{k_1}U_\mu}(\xi-\eta,s)\widehat{P_{k_2}U_\nu}(\eta-\chi,s)\widehat{P_{k_3}Z_\sigma}(\chi,s)\,d\eta d\chi\Big\|_{L^2_\xi},
\end{split}
\end{equation*}
and
\begin{equation*}
\widehat{Z_\mu}(\zeta,s):=e^{-\mu is\Lambda(\zeta)}\cdot \frac{d}{ds}[\widehat{V_\mu}(\zeta,s)].
\end{equation*}
Using the identity \eqref{formul1}, together with the formulas \eqref{ms1}--\eqref{ms3}, and the bounds \eqref{lo7} we estimate, for any $l\in\mathbb{Z}$ and $s\in\mathrm{supp}\,q_m$,
\begin{equation}\label{nh21}
\begin{split}
\|P_l Z_+(s)\|_{L^2}&\lesssim \sum_{(l_1,l_2)\in\mathcal{X}_l}\sum_{(\mu,\nu)\in\{(+,+),(+,-),(-,-)\}}\Big\|\int_{\mathbb{R}^2}m_{\mu\nu}(\xi,\chi)\widehat{P_{l_1}U_\mu}(\xi-\chi,s)\widehat{P_{l_2}U_\nu}(\chi,s)\,d\chi\Big\|_{L^2_\xi}\\
&\lesssim \sum_{(l_1,l_2)\in\mathcal{X}_l,\,l_1\leq l_2}(2^{l_2/4}+2^{3l_2})\min(\|P_{l_1}U_+\|_{L^2}\|P_{l_2}U_+\|_{L^\infty},\|P_{l_1}U_+\|_{L^\infty}\|P_{l_2}U_+\|_{L^2})\\
&\lesssim \varepsilon_0^22^{-m}2^{-8l_+}.
\end{split}
\end{equation}

We apply Lemma \ref{LastTri}, and the bounds \eqref{lo7} and \eqref{nh21}. Recalling \eqref{nh10} and \eqref{nh20},
\begin{equation*}
I+II+III+IV\lesssim \varepsilon_0^42^{-2m}2^{20k_+}2^{m/2},
\end{equation*}
and the desired bound \eqref{nh11} follows.
\end{proof}

\section{Technical estimates}\label{technical}

In this section we collect several technical estimates used in the rest of the paper. We start with estimates on the function $\Lambda$.

\begin{lemma}\label{tech0}
With $\Lambda(x)=\sqrt{a|x|^2+b}$, $\Lambda:\mathbb{R}^2\to\mathbb{R}$, we have
\begin{equation}\label{la1}
\sup_{x\in\mathbb{R}^2}(1+|x|)^{|\alpha|-1}|D^\alpha \Lambda(x)|\lesssim_{|\alpha|}1,
\end{equation}
\begin{equation}\label{la2}
\frac{1}{|\Lambda(x)\pm\Lambda(y)\pm\Lambda(x+y)|}\lesssim 1+\min(|x|,|y|,|x+y|)\qquad\text{ for any }x,y\in\mathbb{R}^2,
\end{equation}
and
\begin{equation}\label{la3}
\frac{|x-y|}{(1+|x|^3+|y|^3)}\lesssim |\nabla\Lambda(x)-\nabla\Lambda(y)|\lesssim |x-y|\qquad\text{ for any }x,y\in\mathbb{R}^2.
\end{equation}
Moreover, if $\Phi_{\mu\nu}(\xi,\eta):=\Lambda(\xi)-\mu\Lambda(\xi-\eta)-\nu\Lambda(\eta)$ as in \eqref{keyphi}, $(\mu,\nu)\in\{(+,+),(+,-),(-,-)\}$, then 
\begin{equation}\label{la4}
|\nabla_{\xi,\eta}\Phi_{\mu\nu}(\xi,\eta)|\lesssim |\Phi_{\mu\nu}(\xi,\eta)|\qquad\text{ for any }\xi,\eta\in\mathbb{R}^2.
\end{equation}
\end{lemma}

\begin{proof}[Proof of Lemma \ref{tech0}] The bound \eqref{la1} follows from the definition. To prove \eqref{la2} we estimate
\begin{equation*}
\Lambda(x)+\Lambda(y)-\Lambda(x+y)\geq\sqrt{b}
\end{equation*}
if $|x+y|\leq\max(|x|,|y|)$, and
\begin{equation*}
\begin{split}
\Lambda(x)+\Lambda(y)-\Lambda(x+y)&\geq (\sqrt{a|x|^2+b}-\sqrt{a}|x|)+(\sqrt{a|y|^2+b}-\sqrt{a}|y|)-(\sqrt{a|x+y|^2+b}-\sqrt{a}|x+y|)\\
&\geq \max[(\sqrt{a|x|^2+b}-\sqrt{a}|x|),(\sqrt{a|y|^2+b}-\sqrt{a}|y|)]\\
&\gtrsim \max[(1+|x|)^{-1},(1+|y|)^{-1}]
\end{split}
\end{equation*}
if $|x+y|\geq\max(|x|,|y|)$. The bound \eqref{la2} follows.

The upper bound in \eqref{la3} follows from the uniform bound $\Vert D^2\Lambda\Vert_{L^\infty}\lesssim 1$, see \eqref{la1}. To prove the lower bound in \eqref{la3} let $\lambda(r)=\sqrt{ar^2+b}$, $\Lambda(\xi)=\lambda(\vert\xi\vert)$. Then
\begin{equation}\label{EstimGrad}
\vert\nabla\Lambda(x)-\nabla\Lambda(y)\vert^2=\left[\lambda^\prime(\vert x\vert)-\lambda^\prime(\vert y\vert)\right]^2+2\lambda^\prime(\vert x\vert)\lambda^\prime(\vert y\vert)\Big(1-\frac{x}{|x|}\cdot\frac{y}{|y|}\Big).
\end{equation}
Notice that, for any $u\leq v\in\mathbb{R}_+$,
\begin{equation*}
\lambda^\prime(v)-\lambda^\prime(u)=\int_{u}^{v}\frac{ds}{(b+as^2)^{3/2}}\gtrsim\frac{v-u}{1+v^3}\quad\text{ and }\quad\lambda^\prime(v)\lambda^\prime(u)\gtrsim\frac{uv}{1+v^2}.
\end{equation*}
Therefore, using \eqref{EstimGrad},
\begin{equation*}
\vert\nabla\Lambda(x)-\nabla\Lambda(y)\vert^2\gtrsim\frac{(|x|-|y|)^2}{(1+|x|^3+|y|^3)^2}+\frac{2|x||y|-2x\cdot y}{1+|x|^2+|y|^2}\gtrsim \frac{|x-y|^2}{(1+|x|^3+|y|^3)^2},
\end{equation*}
as desired.

To prove \eqref{la4} we may assume that $\Phi_{\mu\nu}(\xi,\eta)\ll 1$ (otherwise the bound follows from \eqref{la1}). In particular, we may assume that $(\mu,\nu)\in\{(+,+),(+,-)\}$. Since $\Phi_{+-}(\xi,\eta)=-\Phi_{++}(\xi-\eta,\xi)$, it suffices to consider the case $(\mu,\nu)=(+,+)$. Estimating as in the proof of \eqref{la2},
\begin{equation}\label{la6}
|\Phi_{++}(\xi,\eta)|\gtrsim (1+|\eta|)^{-1}+(1+|\xi-\eta|)^{-1}+(|\eta|+|\xi-\eta|-|\xi|).
\end{equation}
On the other hand, using \eqref{EstimGrad}, if 
\begin{equation}\label{la7}
1\leq\min(|\eta|,|\xi-\eta|)\leq \max(|\eta|,|\xi-\eta|)\leq|\xi|
\end{equation}
then
\begin{equation*}
\begin{split}
|\nabla_{\xi,\eta}\Phi_{++}(\xi,\eta)|^2&\lesssim |\nabla\Lambda(\xi)-\nabla\Lambda(\xi-\eta)|^2+|\nabla\Lambda(\xi-\eta)-\nabla\Lambda(\eta)|^2\\
&\lesssim |\lambda'(\gamma)-\lambda'(\alpha)|^2+|\lambda'(\beta)-\lambda'(\alpha)|^2+\frac{(\alpha+\beta)^2-\gamma^2}{\alpha\beta}\\
&\lesssim \alpha^{-2}+\beta^{-2}+(\alpha+\beta-\gamma)(1/\alpha+1/\beta).
\end{split}
\end{equation*}
where $\alpha:=|\xi-\eta|$, $\beta:=|\eta|$, $\gamma:=|\xi|$. The desired bound \eqref{la4} follows using also \eqref{la6}, in the range described in \eqref{la7}. In the remaining range $|\Phi_{++}(\xi,\eta)|\gtrsim 1$, see \eqref{la6}, and the desired bound \eqref{la4} follows from \eqref{la1}.
\end{proof}

We show now that Riesz transforms map our main space $Y^{N_0}$ to itself.

\begin{lemma}\label{tech1}
Assume $m:\mathbb{R}^2\to\mathbb{C}$ satisfies the symbol-type estimates
\begin{equation}\label{km9}
\sup_{|a|\in[0,10]}\sup_{\xi\in\mathbb{R}^2}|\xi|^{|a|}|\partial^am(\xi)|\leq 1. 
\end{equation}
Then, for any $\phi\in Y^{N_0}$
\begin{equation}\label{km10}
\|\mathcal{F}^{-1}(m\cdot \widehat{f})\|_{Y^{N_0}}\lesssim \|f\|_{Y^{N_0}}.
\end{equation}
\end{lemma}

\begin{proof}[Proof of Lemma \ref{tech1}] Using Plancherel's theorem
\begin{equation*}
\|\mathcal{F}^{-1}(m\cdot \widehat{f})\|_{H^{N_0}}\lesssim \|f\|_{H^{N_0}}.
\end{equation*}
To control the $Z$-norm, we fix $k\in\mathbb{Z}$, let $g=\mathcal{F}^{-1}(m\cdot \widehat{f})$, and write 
\begin{equation*}
\varphi_j(x)P_kg(x)=\int_{\mathbb{R}^2}P_kf(y)\cdot \varphi_j(x)K_k(x-y)\,dy\qquad\text{ where }\qquad K_k(z)=c\int_{\mathbb{R}^2}e^{iz\xi}m(\xi)\varphi_{[k-1,k+1]}(\xi)\,d\xi.
\end{equation*}
In view of \eqref{km9}, $|K_k(z)|\lesssim 2^{2k}(1+2^k|z|)^{-4}$, therefore
\begin{equation*}
\begin{split}
|\varphi_j(x)P_kg(x)|&\lesssim |\varphi_j(x)|\int_{|y|\leq 2^{j-2}}|P_kf(y)|\cdot |K_k(x-y)|\,dy+\sum_{j'\geq j-4}\int_{\mathbb{R}^2}|\varphi_{j'}(y)P_kf(y)|\cdot|K_k(x-y)|\,dy\\
&\lesssim |\varphi_j(x)|\cdot 2^{2k}(1+2^{k+j})^{-4}\cdot \|P_kf\|_{L^1}+\sum_{j'\geq j-4}(|\varphi_{j'}\cdot P_kf|\ast |K_k|)(x).
\end{split}
\end{equation*}
Therefore
\begin{equation*}
\begin{split}
&(2^{k/10}+2^{10k})\sum_{j\in\mathbb{Z}_+}2^j\|\varphi_j\cdot P_kg\|_{L^2}\\
&\lesssim (2^{k/10}+2^{10k})\sum_{j\geq 0}2^{2j+2k}(1+2^{k+j})^{-4}\|P_kf\|_{L^1}+(2^{k/10}+2^{10k})\sum_{j\geq 0}\sum_{j'\geq j-4}2^j\|\varphi_{j'}\cdot P_kf\|_{L^2}\\
&\lesssim \|f\|_{Z},
\end{split}
\end{equation*}
as desired.
\end{proof}

Our next lemma is used several times in the proof of Proposition \ref{LocExLemma}.

\begin{lemma}\label{tech1.5}
(i) For any multi-indices $\alpha,\beta$ with $|\alpha|,|\beta|\geq 1$ and any $f,g\in H^{|\alpha|+|\beta|}$ we have
\begin{equation*}
 \|\Lambda(D^\alpha fD^\beta g)\|_{L^2}
\lesssim_{|\al|+|\be|} \| \nabla f\|_{L^\infty}\Vert g\Vert_{H^{|\al|+|\be|}}+\|\nabla g\|_{L^\infty}\| f\|_{H^{|\alpha|+|\beta|}}.
\end{equation*}

(ii) With the norm $Z'$ defined as in \eqref{Z'norm}, we have
\begin{equation}\label{easy}
 \begin{split}
&\sum_{j=1}^2\big[\|\partial_jf\|_{L^\infty}+\|R_jf\|_{Z'}\big]\lesssim \|f\|_{Z'},\\
&\sum_{k\in\mathbb{Z}}(2^k+2^{k/10})\|P_k(f\cdot g)\|_{L^\infty}\lesssim \|f\|_{Z'}\|g\|_{H^1}+\|f\|_{H^1}\|g\|_{Z'}.
 \end{split}
\end{equation}

(iii) If $f,g\in H^2$ then
\begin{equation}\label{easy2}
\|\Lambda(f\cdot g)-\Lambda f\cdot g\|_{L^2}\lesssim \|f\|_{L^2}\|g\|_{Z'}.
\end{equation}
\end{lemma}

\begin{proof}[Proof of Lemma \ref{tech1.5}] For part (i), we estimate, for any $k\in\mathbb{Z}_+$,
\begin{equation*}
\begin{split}
\|P_k[\Lambda(D^\alpha fD^\beta g)]\|_{L^2}&\lesssim 2^{k}\|P_k(D^\alpha fD^\beta g)\|_{L^2}\\
&\lesssim 2^{k}\|P_{\leq k-4}D^\alpha f\|_{L^\infty}\|P_{[k-4,k+4]}D^\beta g\|_{L^2}
+2^{k}\|P_{[k-4,k+4]}D^\alpha f\|_{L^2}\|P_{\leq k-4}D^\beta g\|_{L^\infty}\\
&+2^{k}\sum_{k_1,k_2\geq k-3,\,|k_1-k_2|\leq 8}\|P_{k_1}D^\alpha f\|_{L^2}\|P_{k_2}D^\beta g\|_{L^\infty}\\
&\leq \sum_{k'\geq k-4}2^{k-k'}\|\nabla f\|_{L^\infty}\|P_{k'}g\|_{H^{|\al|+|\be|}}
+\sum_{k'\geq k-4}2^{k-k'}\|\nabla g\|_{L^\infty}\|P_{k'}f\|_{H^{|\al|+|\be|}},
\end{split}
\end{equation*}
where $P_{\leq l}=P_{(-\infty,l]}$. A similar calculation shows that 
\begin{equation*}
\|P_{\leq 0}[\Lambda(D^\alpha fD^\beta g)]\|_{L^2}\lesssim_{|\al|+|\be|} \| \nabla f\|_{L^\infty}\Vert g\Vert_{H^{|\al|+|\be|}}+\|\nabla g\|_{L^\infty}\| f\|_{H^{|\alpha|+|\beta|}},
\end{equation*}
and the desired estimate follows.
 
The first inequality in \eqref{easy} follows from the definitions, since
\begin{equation*}
 \|P_kR_jf\|_{L^\infty}+2^{-k}\|P_k\partial_jf\|_{L^\infty}\lesssim \|P_kf\|_{L^\infty},\qquad k\in\mathbb{Z},j\in\{1,2\}.
\end{equation*}
To prove the second inequality we estimate, for any $k\in\mathbb{Z}$
\begin{equation*}
\begin{split}
&\|P_k(f\cdot g)\|_{L^\infty}\\
&\lesssim \sum_{|k_2-k|\leq 3}\|P_{\leq k-4}f\cdot P_{k_2}g\|_{L^\infty}+\sum_{|k_1-k|\leq 3}\|P_{k_1}f\cdot P_{\leq k-4}g\|_{L^\infty}+\sum_{k_1,k_2\geq k-3,\,|k_1-k_2|\leq 8}\|P_{k_1}f\cdot P_{k_2}g\|_{L^\infty}\\
&\lesssim \sum_{k'\geq k-3}\min(2^{k'},2^{k'/2})\|f\|_{H^1}(2^{k'/2}+2^{2k'})^{-1}\|g\|_{Z'}+\min(2^k,2^{k/2})\|g\|_{H^1}(2^{k/2}+2^{2k})^{-1}\|f\|_{Z'}\\
&\lesssim \min(1,2^{-3k/2})\big[\|f\|_{Z'}\|g\|_{H^1}+\|f\|_{H^1}\|g\|_{Z'}\big].
\end{split}
\end{equation*}
The bound in the second line of \eqref{easy} follows.

To prove part (iii), we decompose
\begin{equation*}
\begin{split}
&\Lambda(f\cdot g)-\Lambda f\cdot g=F_1+F_2+F_3,\\
&F_1=\sum_{k_1,k_2\in\mathbb{Z},\,|k_1-k_2|\leq 10}(\Lambda-\sqrt{b})(P_{k_1}f\cdot P_{k_2}g)-(\Lambda-\sqrt{b}) P_{k_1}f\cdot P_{k_2}g,\\
&F_2=\sum_{k_1,k_2\in\mathbb{Z},\,k_1\leq k_2-11}(\Lambda-\sqrt{b})(P_{k_1}f\cdot P_{k_2}g)-(\Lambda-\sqrt{b}) P_{k_1}f\cdot P_{k_2}g,\\
&F_3=\sum_{k_1,k_2\in\mathbb{Z},\,k_2\leq k_1-11}\Lambda(P_{k_1}f\cdot P_{k_2}g)-\Lambda P_{k_1}f\cdot P_{k_2}g.
\end{split}
\end{equation*}
Recalling that $\Lambda=\sqrt{a|\nabla|^2+b}$, we estimate easily
\begin{equation*}
\|F_1\|_{L^2}\lesssim \sum_{k_1,k_2\in\mathbb{Z},\,|k_1-k_2|\leq 10}2^{k_1}\|P_{k_1}f\|_{L^2}\cdot \|P_{k_2}g\|_{L^\infty}\lesssim \|f\|_{L^2}\|g\|_{Z'}
\end{equation*}
and
\begin{equation*}
\|F_2\|_{L^2}\lesssim \sum_{k_2\in\mathbb{Z}}2^{k_2}\|P_{\leq k_2-11}f\|_{L^2}\cdot \|P_{k_2}g\|_{L^\infty}\lesssim \|f\|_{L^2}\|g\|_{Z'}.
\end{equation*}

To estimate $\|F_3\|_{L^2}$ we write $F_3=c\sum_{k_1,k_2\in\mathbb{Z},\,k_2\leq k_1-11}F^3_{k_1,k_2}$, where
\begin{equation*}
\begin{split}
F^3_{k_1,k_2}(x)&:=\int_{\mathbb{R}^2\times\mathbb{R}^2}e^{ix\cdot\xi}(\Lambda(\xi)-\Lambda(\xi-\eta))\widehat{P_{k_1}f}(\xi-\eta)\widehat{P_{k_2}g}(\eta)\,d\xi d\eta\\
&=\int_{\mathbb{R}^2\times\mathbb{R}^2}P_{k_1}f(y)P_{k_2}g(z)G_{k_1,k_2}(x,y,z)\,dy dz,
\end{split}
\end{equation*}
and
\begin{equation*}
\begin{split}
G_{k_1,k_2}(x,y,z)&:=\int_{\mathbb{R}^2\times\mathbb{R}^2}e^{ix\cdot\xi}e^{-iy\cdot(\xi-\eta)}e^{-iz\cdot\eta}(\Lambda(\xi)-\Lambda(\xi-\eta))\varphi_{[k_1-1,k_1+1]}(\xi-\eta)\varphi_{[k_2-1,k_2+1]}(\eta)\,d\xi d\eta\\
&=\int_{\mathbb{R}^2\times\mathbb{R}^2}e^{i(x-y)\cdot\chi}e^{i(x-z)\cdot\eta}(\Lambda(\chi+\eta)-\Lambda(\chi))\varphi_{[k_1-1,k_1+1]}(\chi)\varphi_{[k_2-1,k_2+1]}(\eta)\,d\chi d\eta.
\end{split}
\end{equation*}
Since $|D^\alpha\Lambda(\xi)|\lesssim_{|\alpha|}(1+|\xi|)^{1-|\alpha|}$ for any $\xi\in\mathbb{R}^2$, by integration by parts in both $\chi$ and $\eta$ we estimate
\begin{equation*}
|G_{k_1,k_2}(x,y,z)|\lesssim 2^{2k_1}(1+2^{k_1}|x-y|)^{-3}\cdot 2^{3k_2}(1+2^{k_2}|x-z|)^{-3},
\end{equation*}
provided that $k_2\leq k_1-11$. Therefore
\begin{equation*}
\|F^3_{k_1,k_2}\|_{L^2}\lesssim 2^{k_2}\|P_{k_1}f\|_{L^2}\|P_{k_2}g\|_{L^\infty},
\end{equation*}
and the desired bound \eqref{easy2} follows.
\end{proof}

We prove now an important bilinear estimate, which is used repeatedly in section \ref{Bootstrap}. Recall the definition $k_+=\max(k,0)$ for any $k\in\mathbb{Z}$.

\begin{lemma}\label{tech2}
For any $(\mu,\nu)\in\{(+,+),(+,-),(-,-)\}$, $k,k_1,k_2\in\mathbb{Z}$, and $p,q,r\in[1,\infty]$, $r\geq 2$, satisfying $1/p+1/q+1/r=1$ we have
\begin{equation}\label{km1}
\begin{split}
\Big\|&\int_{\mathbb{R}^2}P_{k,k_1,k_2}^{\mu\nu}(\xi,\eta)\widehat{f}(\xi-\eta)\widehat{g}(\eta)\,d\eta\Big\|_{L^2_\xi}\\
&\lesssim (1+2^{2k})(1+2^{3\min(k_1,k_2)})(2^{k}+2^{k_1}+2^{k_2})\|f\|_{L^p}\|g\|_{L^q}2^{k(1-2/r)},
\end{split}
\end{equation}
where $P_{k,k_1,k_2}^{\mu\nu}$ is as in \eqref{km2}.
\end{lemma}

\begin{proof}[Proof of Lemma \ref{tech2}] We estimate the left-hand side of \eqref{km1} by
\begin{equation}\label{lo3}
\begin{split}
&\sup_{\|h\|_{L^2}=1}\Big|\int_{\mathbb{R}^2\times\mathbb{R}^2\times\mathbb{R}^2}h(x)e^{ix\cdot\xi}P_{k,k_1,k_2}^{\mu\nu}(\xi,\eta)\widehat{f}(\xi-\eta)\widehat{g}(\eta)\,d\eta d\xi dx\Big|\\
&\lesssim\sup_{\|h\|_{L^2}=1}\Big|\int_{\mathbb{R}^2\times\mathbb{R}^2\times\mathbb{R}^2}(P_{[k-1,k+1]}h)(x)f(y)g(z)Q^{\mu\nu}_{k,k_1,k_2}(x,y,z)\,dxdydz\Big|,
\end{split}
\end{equation}
where
\begin{equation}\label{lo4}
\begin{split}
&P_{k,k_1,k_2}^{\mu\nu}(\xi,\eta)=\varphi_k(\xi)\frac{m_{\mu\nu}(\xi,\eta)}{\Phi_{\mu\nu}(\xi,\eta)}\varphi_{[k_1-1,k_1+1]}(\xi-\eta)\varphi_{[k_2-1,k_2+1]}(\eta),\\
&Q^{\mu\nu}_{k,k_1,k_2}(x,y,z)=\int_{\mathbb{R}^2\times\mathbb{R}^2}e^{i(x-y)\cdot\xi}e^{i(y-z)\cdot\eta}P_{k,k_1,k_2}^{\mu\nu}(\xi,\eta)\,d\xi d\eta.
\end{split}
\end{equation}
We assume that $k_2\leq k_1$ (the case $k_1\leq k_2$ is similar), thus $\max(k,k_2)\leq k_1+4$. Recall that $\Phi_{\mu\nu}(\xi,\eta)=\Lambda(\xi)-\mu\Lambda(\xi-\eta)-\nu\Lambda(\eta)$, and the inequalities, see Lemma \ref{tech0},
\begin{equation*}
|\Phi_{\mu\nu}(\xi,\eta)|^{-1}\lesssim 1+\min(|\xi|,|\eta|,|\xi-\eta|),\quad \big|D_\xi^mD_\eta^n\Phi_{\mu\nu}(\xi,\eta)\big|\lesssim \min(1,|\Phi_{\mu\nu}(\xi,\eta)|),\quad |m|+|n|\in[1,10].
\end{equation*}
The functions $(\xi,\eta)\to m_{\mu\nu}(\xi,\eta)\varphi_k(\xi)\varphi_{[k_1-1,k_1+1]}(\xi-\eta)\varphi_{[k_2-1,k_2+1]}(\eta)$ are sums of symbols of the form
\begin{equation*}
\begin{split}
&(2^{k}+2^{k_1}+2^{k_2})a_1(\xi)a_2(\xi-\eta)a_3(\eta),\\
&\sup_{|m|\in[0,10]}\sup_{v\in\mathbb{R}^2}\big[2^{k|m|}|\partial^m a_1(v)|+2^{k_1|m|}|\partial^m a_2(v)|+2^{k_2|m|}|\partial^m a_3(v)|\big]\lesssim 1,
\end{split}
\end{equation*}
see \eqref{ms1}--\eqref{ms3}. Integrating by parts $\xi$ and $\eta$ in \eqref{lo4} it follows that
\begin{equation*}
\big|Q^{\mu\nu}_{k,k_1,k_2}(x,y,z)\big|\lesssim (1+2^{k_2})(2^k+2^{k_1}+2^{k_2})\cdot [2^{2k_2}(1+2^{\min(0,k_2)}|y-z|)^{-4}]\cdot [2^{2k}(1+2^{\min(0,k)}|x-y|)^{-4}].
\end{equation*}
The desired bound \eqref{km1} follows from \eqref{lo3} and the estimate
\begin{equation*}
\|P_{[k-1,k+1]}h\|_{L^r}\lesssim 2^{k(1-2/r)}\|h\|_{L^2},\qquad r\in[2,\infty].
\end{equation*}
\end{proof}

The last two lemmas, which concern trilinear estimates, are needed only in the proof of Lemma \ref{ble2}.

\begin{lemma}\label{tech3}
Assume $k,k_1,k_2,k_3\in\mathbb{Z}$ and $l\leq\min(k,k_1,k_2,k_3)-1$. Then
\begin{equation}\label{TriBoundsErr}
\begin{split}
\Big\|\int_{\mathbb{R}^2\times\mathbb{R}^2}&P_{k,k_1,k_2,k_3}^{\mu\nu\sigma}(\xi,\eta,\chi)\widehat{f_1}(\xi-\eta)\widehat{f_2}(\eta-\chi)\widehat{f_3}(\chi)\,d\eta d\chi\Big\|_{L^2_\xi}\lesssim (1+2^{3\widetilde{a}})(2^{2\widetilde{c}}+2^{4\widetilde{c}})\\
&\left[2^{2l}\Vert f_1\Vert_{L^2}\Vert f_2\Vert_{L^2}\Vert f_3\Vert_{L^2}+\left(\min(\widetilde{a},k)-l\right)2^{k(1-2/r)}\Vert f_1\Vert_{L^{p_1}}\Vert f_2\Vert_{L^{p_2}}\Vert f_3\Vert_{L^{p_3}}\right],
\end{split}
\end{equation}
where $P_{k,k_1,k_2,k_3}^{\mu\nu\sigma}$ are as in \eqref{lo31}, $\widetilde{a}=\mathrm{med}(k_1,k_2,k_3)$, $\widetilde{c}=\max(k_1,k_2,k_3)$,  and $p_1,p_2,p_3,r\in[2,\infty]$ satisfy
\begin{equation*}
\frac{1}{p_1}+\frac{1}{p_2}+\frac{1}{p_3}+\frac{1}{r}=1.
\end{equation*}
\end{lemma}

\begin{proof}[Proof of Lemma \ref{tech3}] We make linear changes of variables and relabel the functions $f_1,f_2,f_3$. In view of the formulas \eqref{mult1}--\eqref{mult4} we consider operators of the form
\begin{equation*}
T[f_1,f_2,f_3](\xi)=\varphi_k(\xi)\int_{\mathbb{R}^2\times\mathbb{R}^2}m^0(\xi,\xi_2+\xi_3)m^1(\xi_2+\xi_3,\xi_3)\widehat{P_{k_1}f_1}(\xi-\xi_2-\xi_3)\widehat{P_{k_2}f_2}(\xi_2)\widehat{P_{k_3}f_3}(\xi_3)\,d\xi_2 d\xi_3
\end{equation*}
where $m^0(x,y)=\frac{m_{\mu\nu}(x,y)}{\Phi_{\mu\nu}(x,y)}$ or $m^0(x,y)=\frac{m_{\mu\nu}(x,x-y)}{\Phi_{\mu\nu}(x,x-y)}$ and 
$m^1=m_{\mu\nu}$ or $m^1=m'_{\mu\nu}$, for suitable $\mu,\nu\in\{+,-\}$. It suffices to prove that
\begin{equation}\label{TriBounds2}
\begin{split}
\|T[f_1,f_2,f_3]\|_{L^2}& \lesssim (1+2^{3\widetilde{a}})(2^{2\widetilde{c}}+2^{4\widetilde{c}})\\
&\left[2^{2l}\Vert f_1\Vert_{L^2}\Vert f_2\Vert_{L^2}\Vert f_3\Vert_{L^2}+\left(\min(\widetilde{a},k)-l\right)2^{k(1-2/r)}\Vert f_1\Vert_{L^{p_1}}\Vert f_2\Vert_{L^{p_2}}\Vert f_3\Vert_{L^{p_3}}\right].
\end{split}
\end{equation}

We can further decompose
\begin{equation}\label{SumOp}
\begin{split}
&T=\sum_{n\in\mathbb{Z}}T_n,\\
&T_n[f_1,f_2,f_3](\xi)
=\varphi_k(\xi)\int_{\mathbb{R}^2}m^0(\xi,v)\varphi_n(v)\widehat{P_{k_1}f_1}(\xi-v)\widehat{G}(v)\,dv,\\
&\widehat{G}(v)=\int_{\mathbb{R}^2}m^1(v,\xi_3)\widehat{P_{k_2}f_2}(v-\xi_3)\widehat{P_{k_3}f_3}(\xi_3)\,d\xi_3.
\end{split}
\end{equation}
Since $\|P_nG\|_{L^1}\lesssim (2^{k_2}+2^{k_3})\|f_2\|_{L^2}\|f_3\|_{L^2}$, we estimate
\begin{equation*}
\begin{split}
\sum_{n\le l}\Vert T_{n}[f_1,f_2,f_3]\Vert_{L^2}&\lesssim \sum_{n\le l}(2^k+2^{k_1})(1+2^n)\Vert f_1\Vert_{L^2}\|\widehat{P_nG}\|_{L^1}\\
&\lesssim (2^{2l}+2^{3l})(2^k+2^{k_1})(2^{k_2}+2^{k_3})\|f_1\|_{L^2}\|f_2\|_{L^2}\|f_3\|_{L^2}.
\end{split}
\end{equation*}
Since $\|P_nG\|_{L^{p_2p_3/(p_2+p_3)}}\lesssim (2^{k_2}+2^{k_3})\|f_2\|_{L^{p_2}}\|f_3\|_{L^{p_3}}$, we estimate using Lemma \ref{tech2},
\begin{equation}\label{nh85}
\begin{split}
\Vert T_{n}&[f_1,f_2,f_3]\Vert_{L^2}\lesssim 2^{k(1-2/r)}(1+2^{2k})(1+2^{3\min(k_1,n)})(2^{k}+2^{k_1})\Vert f_1\Vert_{L^{p_1}}\Vert G\Vert_{L^{p_2p_3/(p_2+p_3)}_x}\\
&\lesssim 2^{k(1-2/r)}(1+2^{2k})(1+2^{3\min(k_1,n)})(2^{k}+2^{k_1})(2^{k_2}+2^{k_3})\Vert f_1\Vert_{L^{p_1}}\Vert f_2\Vert_{L^{p_2}}\Vert f_3\Vert_{L^{p_3}}.
\end{split}
\end{equation}
It is easy to see that there are at most $C+|\min(\widetilde{a},k)-l|$ values of $n$ in $[l,\infty)\cap\mathbb{Z}$ for which $T_{n}[f_1,f_2,f_3]$ is nontrivial, and for all such values 
$2^{\min(k_1,n)}\lesssim 2^{\mathrm{med}(k_1,k_2,k_3)}$. The desired bound \eqref{TriBounds2} follows.
\end{proof}

\begin{lemma}\label{KerLem}
Assume $m\in\mathbb{Z}_+$, $k\in\mathbb{Z}$, $(\mu,\nu,\sigma)\in \{(+,+,+),(+,+,-),(+,-,-),(-,-,-)\}$, $s\in [2^{m-1},2^{m+1}]$. Then, for any $f_1,f_2,f_3\in Y^{N_0}$,
\begin{equation}\label{nh29.5}
\begin{split}
&\sum_{j\ge M'(k,m)}2^j\Big\| \varphi_j\cdot P_k\mathcal{F}^{-1}\int_{\mathbb{R}^2\times\mathbb{R}^2}e^{is\Phi_{\mu\nu\sigma}(\xi,\eta,\chi)}m_{\mu\nu\sigma}(\xi,\eta,\chi)\widehat{Q_jf_1}(\xi-\eta)\widehat{ Q_jf_2}(\eta-\chi)\widehat{Q_jf_3}(\chi)d\eta d\chi\Big\|_{L^2_x}\\
&\lesssim 2^{-51m/50}2^{-10k_+}\Vert f_1\Vert_{Y^{N_0}}\Vert f_2\Vert_{Y^{N_0}}\Vert f_3\Vert_{Y^{N_0}},
\end{split}
\end{equation}
where, as in the proof of Lemma \ref{ble2}, $Q_j=P_{[-2j/3,j/10]}$, $\Phi_{\mu\nu\sigma}(\xi,\eta,\chi)=\Lambda(\xi)-\mu\Lambda(\xi-\eta)-\nu\Lambda(\eta-\chi)-\sigma\Lambda(\chi)$,
and for some sufficiently large constant $D$,
\begin{equation*}
M'(k,m):=D+\max(m,-21k/20,10k).
\end{equation*}
In addition, with $M(k,m)$ as in \eqref{nh3}, for any $f_1,f_2,f_3\in Y^{N_0}$
\begin{equation}\label{nh30}
\begin{split}
&\sum_{j\ge M(k,m)}2^j\Big\| \varphi_j\cdot P_k\mathcal{F}^{-1}\int_{\mathbb{R}^2\times\mathbb{R}^2}e^{is\Phi_{++-}(\xi,\eta,\chi)}m_{++-}(\xi,\eta,\chi)\widehat{Q_jf_1}(\xi-\eta)\widehat{ Q_jf_2}(\eta-\chi)\widehat{Q_jf_3}(\chi)d\eta d\chi\Big\|_{L^2_x}\\
&\lesssim 2^{-1001m/1000}2^{-10k_+}\Vert f_1\Vert_{Y^{N_0}}\Vert f_2\Vert_{Y^{N_0}}\Vert f_3\Vert_{Y^{N_0}}.
\end{split}
\end{equation}
\end{lemma}

\begin{proof}[Proof of Lemma \ref{KerLem}] We may assume that $\Vert f_1\Vert_{Y^{N_0}}=\Vert f_2\Vert_{Y^{N_0}}=\Vert f_3\Vert_{Y^{N_0}}=1$ and decompose the multipliers  $m_{\mu\nu\sigma}(\xi,\eta,\chi)$ into multipliers of the form $m^0(\xi,\xi_2+\xi_3)m^1(\xi_2+\xi_3,\xi_3)$ where, for some $(\mu,\nu),(\mu',\nu')\in \{(+,+),(+,-),(-,-)\}$,
\begin{equation}\label{CanForm1}
\begin{split}
&m^1(x,y)=m_{\mu\nu}(x,y)\text{ or }m^1(x,y)=m_{\mu\nu}(x,x-y)\text{ or }m^1(x,y)=m^\prime_{\mu\nu}(x,y)\text{ or }m^1(x,y)=m'_{\mu\nu}(x,x-y),\\
&m^0(x,y)=\frac{m_{\mu'\nu'}(x,y)}{\Phi_{\mu'\nu'}(x,y)}\quad\text{ or }\quad m^0(x,y)=\frac{m_{\mu'\nu'}(x,x-y)}{\Phi_{\mu'\nu'}(x,x-y)}.
\end{split}
\end{equation}
After relabelling the variables, as in the proof of Lemma \ref{tech3}, we consider the functions $G^{\mu\nu\sigma}_{j,k_1,k_2,k_3}$, $\mu,\nu,\sigma\in\{+,-\}$, defined by
\begin{equation}\label{CanForm0}
\begin{split}
\widehat{G^{\mu\nu\sigma}_{j,k_1,k_2,k_3}}(\xi):=\varphi_k(\xi)\int_{\mathbb{R}^2\times\mathbb{R}^2}e^{is\tilde{\Phi}_{\mu\nu\sigma}(\xi,\xi_2,\xi_3)}m^0(\xi,\xi_2+\xi_3)m^1(\xi_2+\xi_3,\xi_3)\\
\widehat{P_{k_1}f_1}(\xi-\xi_2-\xi_3)\widehat{P_{k_2} f_2}(\xi_2)\widehat{P_{k_3} f_3}(\xi_3)\,d\xi_2 d\xi_3,
\end{split}
\end{equation}
where $j\geq M(k,m)$, $k_1,k_2,k_3\in[-2j/3,j/10]\cap\mathbb{Z}$, $m^0,m^1$ are as in \eqref{CanForm1}, and
\begin{equation}\label{CanForm}
\tilde{\Phi}_{\mu\nu\sigma}(\xi,\xi_2,\xi_3):=\Lambda(\xi)-\mu\Lambda(\xi-\xi_2-\xi_3)-\nu\Lambda(\xi_2)-\sigma\Lambda(\xi_3).
\end{equation}

{\bf{Step 1.}} We estimate now $2^j\|\varphi_j(x)\cdot G^{\mu\nu\sigma}_{j,k_1,k_2,k_3}(x)\|_{L^2}$ for fixed $k_1,k_2,k_3,j$. For this we write, after changes of variables
\begin{equation*}
\widehat{G^{\mu\nu\sigma}_{j,k_1,k_2,k_3}}(\xi)=\varphi_k(\xi)\int_{\mathbb{R}^2\times\mathbb{R}^2}e^{is\tilde{\Phi}_{\mu\nu\sigma}(\xi,v-\xi_3,\xi_3)}m^0(\xi,v)m^1(v,\xi_3)\widehat{P_{k_1}f_1}(\xi-v)\widehat{P_{k_2} f_2}(v-\xi_3)\widehat{P_{k_3} f_3}(\xi_3)\,dv d\xi_3.
\end{equation*}
We decompose
\begin{equation*}
G^{\mu\nu\sigma}_{j,k_1,k_2,k_3}=\sum_{n\in\mathbb{Z}}G^{\mu\nu\sigma}_{n,j,k_1,k_2,k_3}
\end{equation*}
where
\begin{equation*}
\begin{split}
\widehat{G^{\mu\nu\sigma}_{n,j,k_1,k_2,k_3}}(\xi)=&\varphi_k(\xi)\int_{\mathbb{R}^2\times\mathbb{R}^2}e^{is\tilde{\Phi}_{\mu\nu\sigma}(\xi,v-\xi_3,\xi_3)}\\
&\times m^0(\xi,v)m^1(v,\xi_3)\varphi_n(v)\cdot\widehat{P_{k_1}f_1}(\xi-v)\widehat{P_{k_2} f_2}(v-\xi_3)\widehat{P_{k_3} f_3}(\xi_3)\,dv d\xi_3.
\end{split}
\end{equation*}

Let
\begin{equation}\label{nh47.5}
Q(\xi,v,\xi_3):=\varphi_k(\xi)\varphi_{[k_1-1,k_1+1]}(\xi-v)\varphi_{[k_2-1,k_2+1]}(v-\xi_3)\varphi_{[k_3-1,k_3+1]}(\xi_3)\cdot m^0(\xi,v)m^1(v,\xi_3).
\end{equation}
For simplicity of notation we drop the subscripts $j,k_1,k_2,k_3$ and rewrite
\begin{equation}\label{nh48}
G^{\mu\nu\sigma}_{j,k_1,k_2,k_3}=\sum_{n\in\mathbb{Z}}T^{\mu\nu\sigma}_{n}[P_{k_1}f_1,P_{k_2}f_2,P_{k_3}f_3],
\end{equation}
where, for any $g_1,g_2,g_3\in L^2(\mathbb{R}^2)$,
\begin{equation}\label{nh49}
\mathcal{F}\big[T^{\mu\nu\sigma}_{n}[g_1,g_2,g_3]\big](\xi):=\int_{\mathbb{R}^2\times\mathbb{R}^2}e^{is\tilde{\Phi}_{\mu\nu\sigma}(\xi,v-\xi_3,\xi_3)}\varphi_n(v)Q(\xi,v,\xi_3)\widehat{g_1}(\xi-v)\widehat{g_2}(v-\xi_3)\widehat{g_3}(\xi_3)\,dv d\xi_3.
\end{equation}
We will also need the representation of the operators $T^{\mu\nu\sigma}_{n}[g_1,g_2,g_3]$ in the physical space,
\begin{equation}\label{nh50}
T^{\mu\nu\sigma}_{n}[g_1,g_2,g_3](x)=\int_{\mathbb{R}^2\times\mathbb{R}^2\times\mathbb{R}^2}g_1(y_1)g_2(y_2)g_3(y_3)\cdot R^{\mu\nu\sigma}_{n}(x;y_1,y_2,y_3)\,dy_1dy_2dy_3,
\end{equation}
where
\begin{equation}\label{nh40}
\begin{split}
R^{\mu\nu\sigma}_{n}(x;y_1,y_2,y_3):=c\int_{\mathbb{R}^2\times\mathbb{R}^2\times\mathbb{R}^2}&e^{i(x-y_1)\cdot\xi}e^{i(y_1-y_2)\cdot v}e^{i(y_2-y_3)\cdot \xi_3}\\
&\times e^{is[\Lambda(\xi)-\mu\Lambda(\xi-v)-\nu\Lambda(v-\xi_3)-\sigma\Lambda(\xi_3)]}\varphi_n(v)Q(\xi,v,\xi_3)\,d\xi dv d\xi_3.
\end{split}
\end{equation}

Recall that $j\geq m+D$ for $D$ sufficiently large, and 
\begin{equation*}
\sup_{\xi\in\mathbb{R}^2}|D^\alpha\Lambda(\xi)|\lesssim _{|\alpha|}1.
\end{equation*}
We prove bounds on the kernel $R^{\mu\nu\sigma}_{n}$ using the formula \eqref{nh40} and integration by parts. More precisely, we use the general bound
\begin{equation}\label{nh41.5}
\begin{split}
&\Big|\int_{\mathbb{R}^d}e^{iKf}g\,dx\Big|\lesssim_N(K\rho)^{-N}|{\mathrm{supp}}\,g|,\qquad 0<\rho\leq 1/\rho\leq K,\qquad\text{ provided that}\qquad\\
&\|D^\alpha g\|_{L^\infty}\lesssim_{|\alpha|}\rho^{-|\alpha|},\,|\alpha|\geq 0,\qquad \|D^\alpha f \cdot\mathbf{1}_{{\mathrm{supp}}\,g}\|_{L^\infty}\lesssim_{|\alpha|}\rho^{1-|\alpha|},\,|\alpha|\geq 1,\qquad|\nabla f|\geq \mathbf{1}_{{\mathrm{supp}}\,g},
\end{split}
\end{equation}
which follows easily by smooth localization to balls of radius $\approx\rho$ and integration by parts. 

Recall that $|k|,|k_1|,|k_2|,|k_3|$ are all bounded by $99j/100$. Integrating by parts only in $\xi$, it follows that
\begin{equation}\label{nh40.5}
\text{ if }|x-y_1|\geq 2^{j-10}\text{ then }\big|R^{\mu\nu\sigma}_{n}(x;y_1,y_2,y_3)\big|\lesssim 2^{2n}(2^j+|x-y_1|)^{-100}.
\end{equation}
Therefore, letting
\begin{equation*}
f^1_{j,k_1}(y):=\varphi_{[j-4,j+4]}(y)\cdot P_{k_1}f_1(y),
\end{equation*}
and noticing that the operators $T^{\mu\nu\sigma}_{n}$ are nontrivial only if $n\leq \max(k_2,k_3)+4$, 
\begin{equation}\label{nh41}
\sum_{n\in\mathbb{Z}}2^j\big\|\varphi_j\cdot T^{\mu\nu\sigma}_{n}[P_{k_1}f_1-f^1_{j,k_1},P_{k_2}f_2,P_{k_3}f_3]\big\|_{L^2}\lesssim 2^{-50j}.
\end{equation}

To estimate the contributions of the functions $f^1_{j,k_1}$ we analyze two cases: $n\leq -4j/5$ and $n\geq -4j/5$. If $n\leq -4j/5$ then we rewrite in the frequency space, using the representation \eqref{nh49},
\begin{equation*}
\begin{split}
\mathcal{F}\big[T^{\mu\nu\sigma}_{n}[f^1_{j,k_1},P_{k_2}f_2,P_{k_3}f_3]\big](\xi)=\int_{\mathbb{R}^2\times\mathbb{R}^2}&e^{is\tilde{\Phi}_{\mu\nu\sigma}(\xi,v-\xi_3,\xi_3)}\\
&\times\varphi_n(v)Q(\xi,v,\xi_3)\widehat{f^1_{j,k_1}}(\xi-v)\widehat{P_{k_2} f_2}(v-\xi_3)\widehat{P_{k_3} f_3}(\xi_3)\,dv d\xi_3.
\end{split}
\end{equation*}
Since $|Q(\xi,v,\xi_3)|\lesssim 1+2^{3\max(k_1,k_2,k_3)}$, we estimate
\begin{equation*}
\big\|\mathcal{F}\big[T^{\mu\nu\sigma}_{n}[f^1_{j,k_1},P_{k_2}f_2,P_{k_3}f_3]\big]\big\|_{L^2}\lesssim (1+2^{3\max(k_1,k_2,k_3)})2^{2n}\|\widehat{f^1_{j,k_1}}\|_{L^2}\|\widehat{P_{k_2} f_2}\|_{L^2}\|\widehat{P_{k_3} f_3}\|_{L^2}.
\end{equation*}
Recalling that $j\geq m$ and $\max(k,k_1,k_2,k_3)\leq j/10+4$, it follows that
\begin{equation}\label{nh42}
\sum_{n\leq-4j/5}2^j\big\|T^{\mu\nu\sigma}_{n}[f^1_{j,k_1},P_{k_2}f_2,P_{k_3}f_3]\big\|_{L^2}\lesssim 2^{-11m/10}2^{-7k_+}2^{-2j/5}.
\end{equation}

To estimate the contributions corresponding to $n\geq -4j/5$ we reexamine the formula \eqref{nh40} and notice that we can integrate by parts in either $\xi$ or $v$ or $\xi_3$. As a result, using \eqref{nh41.5}, the bound \eqref{nh40.5} can be replaced with the stronger bound
\begin{equation*}
\text{ if }|x-y_1|+|x-y_2|+|x-y_3|\geq 2^{j-10}\text{ then }\big|R^{\mu\nu\sigma}_{n}(x;y_1,y_2,y_3)\big|\lesssim (|x-y_3|+|x-y_1|+|x-y_2|)^{-100}.
\end{equation*}
As a consequence, if
\begin{equation*}
f^2_{j,k_2}(y):=\varphi_{[j-4,j+4]}(y)\cdot P_{k_2}f_2(y),\qquad f^3_{j,k_3}(y):=\varphi_{[j-4,j+4]}(y)\cdot P_{k_3}f_3(y),
\end{equation*}
it follows that
\begin{equation}\label{nh44}
\sum_{n\geq-4j/5}2^j\Big[\big\|\varphi_j\cdot T^{\mu\nu\sigma}_{n}[f^1_{j,k_1},P_{k_2}f_2-f^2_{j,k_2},P_{k_3}f_3]\big\|_{L^2}+\big\|\varphi_j\cdot T^{\mu\nu\sigma}_{n}[f^1_{j,k_1},f^2_{j,k_2},P_{k_3}f_3-f^3_{j,k_3}]\big\|_{L^2}\Big]\lesssim 2^{-50j}.
\end{equation}

Finally, we estimate the contributions of $T^{\mu\nu\sigma}_{n}[f^1_{j,k_1},f^2_{j,k_2},f^3_{j,k_3}]$ in the Fourier space, using the formula \eqref{nh49}. Since $|Q(\xi,v,\xi_3)|\lesssim 1+2^{3\max(k_1,k_2,k_3)}$, we estimate as before
\begin{equation*}
\big\|\mathcal{F}\big[T^{\mu\nu\sigma}_{n}[f^1_{j,k_1},f^2_{j,k_2},f^3_{j,k_3}]\big]\big\|_{L^2}\lesssim (1+2^{3\max(k_1,k_2,k_3)})2^{2n}\|\widehat{f^1_{j,k_1}}\|_{L^2}\|\widehat{f^2_{j,k_2}}\|_{L^2}\|\widehat{f^3_{j,k_3}}\|_{L^2}.
\end{equation*}
Recalling that $j\geq m$ and that the operators $T^{\mu\nu\sigma}_{n}$ are nontrivial only if $n\leq \max(k_2,k_3)+4$,
\begin{equation}\label{nh45}
\sum_{-4j/5\leq n}2^j\big\|T^{\mu\nu\sigma}_{n}[f^1_{j,k_1},f^2_{j,k_2},f^3_{j,k_3}]\big\|_{L^2}\lesssim 2^{-11m/10}2^{-5k_+}2^{-3j/5}.
\end{equation}

It follows from the formula \eqref{nh48} and the bounds \eqref{nh41}, \eqref{nh42}, \eqref{nh44}, and \eqref{nh45} that
\begin{equation*}
2^j\big\|\varphi_j\cdot G^{\mu\nu\sigma}_{j,k_1,k_2,k_3}\big\|_{L^2}\lesssim 2^{-11m/10}2^{-7k_+}2^{-2j/5}
\end{equation*}
and the desired bound \eqref{nh29.5} follows by summing over $j,k_1,k_2,k_3$ in the appropriate ranges.

{\bf{Step 2.}} It remains to prove \eqref{nh30}. We may assume that $M(k,m)<M'(k,m)$, i.e.
\begin{equation}\label{nh60}
m\geq\max(-21k/20,10k),\qquad M'(k,m)=D+m.
\end{equation}
We define the functions $G^{\mu\nu\sigma}_{j,k_1,k_2,k_3}$ as in \eqref{CanForm0} and we have to consider the cases $(\mu,\nu,\sigma)=(+,+,-)$ and $(\mu,\nu,\sigma)=(-,+,+)$. We use first $L^2$ estimates to further restrict the ranges of the parameters: if $j\in[1,m+D]$ and $k_1,k_2,k_3\in[-2j/3,j/10]$ then we use Lemma \ref{tech3} with $r=2$, $l=-100j$ and appropriate choices of $p_1,p_2,p_3$ to derive the bounds
\begin{equation*}
\begin{split}
&\|G^{\mu\nu\sigma}_{j,k_1,k_2,k_3}\|_{L^2}\lesssim j2^{-2m}2^{\widetilde{b}/2}2^{-\widetilde{c}_+}\qquad\text{ if }k\leq 0,\\
&\|G^{\mu\nu\sigma}_{j,k_1,k_2,k_3}\|_{L^2}\lesssim j2^{-2m}\min(2^{-\widetilde{b}/5}2^{-15\widetilde{c}_+},2^{3\widetilde{b}/5})\qquad\text{ if }k\geq 0,
\end{split}
\end{equation*}
where $\widetilde{b}=\min(k_1,k_2,k_3)$ and $\widetilde{c}=\max(k_1,k_2,k_3)$. Therefore 
\begin{equation*}
\|G^{\mu\nu\sigma}_{j,k_1,k_2,k_3}\|_{L^2}\lesssim j2^{-2m}2^{-10k_+}\min(2^{\widetilde{b}/15},2^{-\widetilde{c}_+/2}),
\end{equation*}
so the contribution of the $L^2$ norms of the functions $G^{\mu\nu\sigma}_{j,k_1,k_2,k_3}$ when $j\leq 49m/50+2D$ or when $\widetilde{b}\leq -m/50$ or when $\widetilde{c}\geq m/50$ are bounded by $C2^{-10k_+}2^{-1001m/1000}$ as desired. It remains to prove that
\begin{equation}\label{nh65}
2^j\|\varphi_j(x)\cdot G^{\mu\nu\sigma}_{j,k_1,k_2,k_3}(x)\|_{L^2_x}\lesssim 2^{-101m/100}2^{-10k_+},
\end{equation}
where $(\mu,\nu,\sigma)\in\{(+,+,-),(-,+,+)\}$, for parameters $k,m,j,k_1,k_2,k_3$ fixed and satisfying
\begin{equation}\label{nh66}
k_1,k_2,k_3\in[-m/50,m/400],\qquad j\in[49m/50+2D,m+D],\qquad m\geq\max(-21k/20,10k).
\end{equation}

We use again the identities \eqref{nh47.5}--\eqref{nh40}. Using \eqref{nh47.5}, \eqref{nh49}, and the bound $|Q(\xi,v,\xi_3)|\lesssim 1+2^{3\max(k_1,k_2,k_3)}$, we estimate as before
\begin{equation}\label{nh67}
\big\|T^{\mu\nu\sigma}_{n}[g_1,g_2,g_3]\big\|_{L^2}\lesssim (1+2^{3\max(k_1,k_2,k_3)})2^{2n}\|g_1\|_{L^2}\|g_2\|_{L^2}\|g_3\|_{L^2},
\end{equation}
for any $g_1,g_2,g_3\in L^2(\mathbb{R}^2)$ and any $n\in\mathbb{Z}$. For \eqref{nh65} it suffices to prove that
\begin{equation}\label{nh68}
\sum_{n\in\mathbb{Z}}2^j\|\varphi_j\cdot T^{\mu\nu\sigma}_{n}[P_{k_1}f_1,P_{k_2}f_2,P_{k_3}f_3]\big\|_{L^2}\lesssim 2^{-11m/10},
\end{equation}
where $k,m,j,k_1,k_2,k_3$ satisfy \eqref{nh66}, $(\mu,\nu,\sigma)\in\{(+,+,-),(-,+,+)\}$, and $\|f_\alpha\|_{Y^{N_0}}\leq 1$, $\al=1,2,3$. As before, let
\begin{equation*}
\widetilde{c}=\max(k_1,k_2,k_3),\qquad \widetilde{a}=\mathrm{med}(k_1,k_2,k_3),\qquad \widetilde{b}=\min(k_1,k_2,k_3).
\end{equation*}

{\bf{Step 3. }} We estimate first the contribution to the left-hand side of \eqref{nh68} coming from $n\leq -4j/5$. Notice that this contribution is trivial unless $|k-k_1|\leq 4$ and $|k_2-k_3|\leq 4$. For any $\widetilde{j}\geq 1$ let
\begin{equation*}
f^\al_{\leq\widetilde{j},k_\al}(y):=\varphi_{(-\infty,\widetilde{j}]}(y)\cdot P_{k_\al}f_{\al}(y),\qquad \al=1,2,3.
\end{equation*}
Since $\|f_\al\|_{Y^{N_0}}\leq 1$, $\alpha\in\{1,2,3\}$, for any $\widetilde{j}\geq 1$ we have
\begin{equation}\label{nh69}
\begin{split}
&\|P_{k_\al}f_{\al}-f^\al_{\leq\widetilde{j},k_\al}\|_{L^2}\lesssim 2^{-\widetilde{j}}(2^{k_\al/10}+2^{10k_\alpha})^{-1},\\
&\|P_{k_\al}f_{\al}-f^\al_{\leq\widetilde{j},k_\al}\|_{L^2}+\|f^\al_{\leq\widetilde{j},k_\al}\|_{L^2}\lesssim \min(2^{9k_\alpha/10},2^{-N_0k_\alpha}),\\
&\|P_{k_\al}f_{\al}-f^\al_{\leq\widetilde{j},k_\al}\|_{L^1}+\|f^\al_{\leq\widetilde{j},k_\al}\|_{L^1}\lesssim \min(2^{-k_\alpha/10},2^{-10k_\alpha}).
\end{split}
\end{equation}
Using \eqref{nh66}, \eqref{nh67}, and \eqref{nh69} it follows that
\begin{equation*}
\big\|T^{\mu\nu\sigma}_{n}[P_{k_1}f_1,P_{k_2}f_2,P_{k_3}f_3]-T^{\mu\nu\sigma}_{n}[f^1_{\leq4j/5,k_1},f^2_{\leq4j/5,k_2},f^3_{\leq4j/5,k_3}]\big\|_{L^2}\lesssim 2^{m/50}2^{2n}2^{-4j/5}.
\end{equation*}
therefore, using \eqref{nh66},
\begin{equation}\label{nh70}
\sum_{n\leq-4j/5}2^j\big\|T^{\mu\nu\sigma}_{n}[P_{k_1}f_1,P_{k_2}f_2,P_{k_3}f_3]-T^{\mu\nu\sigma}_{n}[f^1_{\leq4j/5,k_1},f^2_{\leq4j/5,k_2},f^3_{\leq4j/5,k_3}]\big\|_{L^2}\lesssim 2^{-11m/10}.
\end{equation}

To estimate the contributions of $T^{\mu\nu\sigma}_{n}[f^1_{\leq 4j/5,k_1},f^2_{\leq 4j/5,k_2},f^3_{\leq 4j/5,k_3}]$ we use the physical space formula \eqref{nh50} and Lemma \ref{tech0}. Let $\Psi^{\mu\nu\sigma}(\xi,v,\xi_3,x,y_1,y_2,y_3)$ denote the phase in the integral defining the kernels $R^{\mu\nu\sigma}$, i.e.
\begin{equation}\label{nh71}
\begin{split}
\Psi^{\mu\nu\sigma}(\xi,v,\xi_3,x,y_1,y_2,y_3):&=(x-y_1)\cdot\xi+(y_1-y_2)\cdot v+(y_2-y_3)\cdot \xi_3\\
&+s[\Lambda(\xi)-\mu\Lambda(\xi-v)-\nu\Lambda(v-\xi_3)-\sigma\Lambda(\xi_3)].
\end{split}
\end{equation}
Recalling \eqref{nh66}, it follows from \eqref{la3} that if $|v|\in[2^{n-1},2^{n+1}]$ then
\begin{equation*}
\begin{split}
&\Big|\nabla_\xi\big[s(\Lambda(\xi)-\Lambda(\xi-v)-\Lambda(v-\xi_3)+\Lambda(\xi_3))\big]\Big|=s\big|\nabla\Lambda(\xi)-\nabla\Lambda(\xi-v)\big|\lesssim 2^{m+n},\\
&\Big|\nabla_{\xi_3}\big[s(\Lambda(\xi)+\Lambda(\xi-v)-\Lambda(v-\xi_3)-\Lambda(\xi_3))\big]\Big|=s\big|\nabla\Lambda(v-\xi_3)-\nabla\Lambda(\xi_3)\big|\gtrsim 2^{m+k_3-5c_+}\gtrsim 2^{9m/10}.
\end{split}
\end{equation*}
Using \eqref{nh41.5} and integrating by parts only in $\xi$ in the formula \eqref{nh40},
\begin{equation*}
\text{ if }|x-y_1|\geq 2^{j-10}\text{ then }|R_n^{++-}(x;y_1,y_2,y_3)|\lesssim 2^{-100j}.
\end{equation*}
Similarly, using \eqref{nh41.5} and integrating by parts only in $\xi_3$ in the formula \eqref{nh40},
\begin{equation*}
\text{ if }|y_2-y_3|\leq 2^{4j/5+10}\text{ then }|R_n^{-++}(x;y_1,y_2,y_3)|\lesssim 2^{-100j},
\end{equation*}
provided that the constant $D$ is sufficiently large. Thus
\begin{equation*}
\sum_{n\leq -4j/5}2^j\big\|\varphi_j\cdot T^{\mu\nu\sigma}_{n}[f^1_{\leq4j/5,k_1},f^2_{\leq4j/5,k_2},f^3_{\leq4j/5,k_3}]\big\|_{L^2}\lesssim 2^{-50j},
\end{equation*}
for $(\mu,\nu,\sigma)\in\{(+,+,-),(-,+,+)\}$. Using also \eqref{nh70} it follows that
\begin{equation}\label{nh75}
\sum_{n\leq -4j/5}2^j\big\|\varphi_j\cdot T^{\mu\nu\sigma}_{n}[P_{k_1}f_1,P_{k_2}f_2,P_{k_3}f_3]\big\|_{L^2}\lesssim 2^{-11m/10},
\end{equation}
as desired.

{\bf{Step 4.}} We consider now the contribution to the left-hand side of \eqref{nh68} coming from $n\geq -4j/5$. Using Lemma \eqref{tech3}, more precisely the bound \eqref{nh85}, together with the bounds \eqref{nh69} and the dispersive estimate \eqref{disper}, it follows that
\begin{equation}\label{nh76}
\begin{split}
&\big\|T^{\mu\nu\sigma}_{n}[P_{k_1}f_1-f^1_{\leq j/2,k_1},P_{k_2}f_2,P_{k_3}f_3]\big\|_{L^2}\lesssim 2^{-j/2}2^{-2m}2^{m/4},\\
&\big\|T^{\mu\nu\sigma}_{n}[f^1_{\leq j/2,k_1},P_{k_2}f_2-f_{\leq j/2,k_2}^2,P_{k_3}f_3]\big\|_{L^2}\lesssim 2^{-j/2}2^{-2m}2^{m/4},\\
&\big\|T^{\mu\nu\sigma}_{n}[f^1_{\leq j/2,k_1},f_{\leq j/2,k_2}^2,P_{k_3}f_3-f_{\leq j,k_3}^3]\big\|_{L^2}\lesssim 2^{-j/2}2^{-2m}2^{m/4}.
\end{split}
\end{equation}

To estimate the contribution of the functions $f_{\leq j/2,k_1}^1,f_{\leq j/2,k_2}^2,f_{\leq j/2,k_3}^3$ we recall the definitions \eqref{nh71} and notice that, in the support of the function $Q(\xi,v,\xi_3)\varphi_n(v)$,
\begin{equation}\label{nh77}
\begin{split}
&\big|\nabla_{\xi,\xi_3,v}\Psi^{++-}(\xi,v,\xi_3,x,y_1,y_2,y_3)\big|\gtrsim 2^{9j/10},\\
&\big|\nabla_{\xi,\xi_3,v}\Psi^{-++}(\xi,v,\xi_3,x,y_1,y_2,y_3)\big|\gtrsim 2^{9j/10},
\end{split}
\end{equation}
provided that \eqref{nh66} holds, $|x|\approx 2^j$, and $|y_1|+|y_2|+|y_3|\lesssim 2^{j/2}$. Indeed, using \eqref{la3}, we notice that
\begin{equation*}
\begin{split}
&\big|\nabla_{\xi}\Psi^{++-}(\xi,v,\xi_3,x,y_1,y_2,y_3)\big|\gtrsim 2^{j}\qquad\text{ unless }2^j\lesssim s|v|,\\
&\big|\nabla_{\xi_3}\Psi^{++-}(\xi,v,\xi_3,x,y_1,y_2,y_3)\big|\gtrsim 2^{m+n}2^{-10c_+}\qquad\text{ if }2^j\lesssim s|v|.
\end{split}
\end{equation*}
The bound in the first line of \eqref{nh77} follows. In addition, we have
\begin{equation*}
\begin{split}
&\text{ if }\big|\nabla_{\xi}\Psi^{-++}(\xi,v,\xi_3,x,y_1,y_2,y_3)\big|\leq 2^{9j/10}\qquad\text{ then }2^{j}\lesssim s|v-2\xi|,\\
&\text{ if }\big|\nabla_{\xi_3}\Psi^{-++}(\xi,v,\xi_3,x,y_1,y_2,y_3)\big|\leq 2^{9j/10}\qquad\text{ then }s|v-2\xi_3|\lesssim 2^{9j/10}2^{10c_+},\\
&\text{ if }\big|\nabla_{v}\Psi^{-++}(\xi,v,\xi_3,x,y_1,y_2,y_3)\big|\leq 2^{9j/10}\qquad\text{ then }s|\xi-\xi_3|\lesssim 2^{9j/10}2^{10c_+}.
\end{split}
\end{equation*}
The inequality in the second line of \eqref{nh77} follows.

Given \eqref{nh77}, we use \eqref{nh41.5} as before to conclude that if $n\geq -4j/5$ then
\begin{equation}\label{nh78}
2^j\big\|\varphi_j\cdot T^{\mu\nu\sigma}_{n}[f^1_{\leq j/2,k_1},f^2_{\leq j/2,k_2},f^3_{\leq j/2,k_3}]\big\|_{L^2}\lesssim 2^{-50j}.
\end{equation}
The desired bound \eqref{nh68} follows from \eqref{nh75}, \eqref{nh76}, and \eqref{nh78}. This completes the proof of the lemma.
\end{proof}

\begin{lemma}\label{LastTri}
Assume that $k\in\mathbb{Z}$, $k_1,k_2,k_3\in[-K,K/10]\cap\mathbb{Z}$ for some $K\geq 1$, $\widetilde{a}=\mathrm{med}(k_1,k_2,k_3)$, $\widetilde{c}=\max(k_1,k_2,k_3)$, $\rho\in[2^{-c_+},1]$, and $p_1,p_2,p_3\in\{2,\infty\}$, $1/p_1+1/p_2+1/p_3=1/2$. Then
\begin{equation}\label{nh90}
\begin{split}
\Big\|\int_{\mathbb{R}^2\times\mathbb{R}^2}&\frac{P_{k,k_1,k_2,k_3}^{\mu\nu\sigma}(\xi,\eta,\chi)}{\Psi(\xi,\eta,\chi)}\widehat{f_1}(\xi-\eta)\widehat{f_2}(\eta-\chi)\widehat{f_3}(\chi)\,d\eta d\chi\Big\|_{L^2_\xi}\\
&\lesssim \rho^{-6}2^{6(\widetilde{a}_++\widetilde{c}_+)}\big[2^{-10K}\Vert f_1\Vert_{L^2}\Vert f_2\Vert_{L^2}\Vert f_3\Vert_{L^2}+K\|f_1\|_{L^{p_1}}\|f_2\|_{L^{p_2}}\|f_3\|_{L^{p_3}}\big],
\end{split}
\end{equation}
where $P_{k,k_1,k_2,k_3}^{\mu\nu\sigma}$ is as in \eqref{lo31}, and the smooth function $\Psi:\mathbb{R}^2\times\mathbb{R}^2\times\mathbb{R}^2$ satisfies the inequalities
\begin{equation}\label{nh91}
\inf_{\xi,\eta,\chi\in\mathbb{R}^2}|\Psi(\xi,\eta,\chi)|\geq \rho,\qquad\sup_{|\alpha|\in[1,20]}\sup_{\xi,\eta,\chi\in\mathbb{R}^2}|D^\alpha_{\xi,\eta,\chi}\Psi(\xi,\eta,\chi)|\lesssim 1.
\end{equation}
\end{lemma}

\begin{proof}[Proof of Lemma \ref{LastTri}] We decompose the integral in the left-hand side of \eqref{nh90} as in the proof of Lemma \ref{KerLem}, see \eqref{CanForm}. More precisely, with $Q$ defined as in \eqref{nh47.5}, it suffices to prove that
\begin{equation}\label{nh92}
\begin{split}
\Big\|\int_{\mathbb{R}^2\times\mathbb{R}^2}\frac{Q(\xi,v,\xi_3)}{\tilde{\Psi}(\xi,v,\xi_3)}&\widehat{f_1}(\xi-v)\widehat{f_2}(v-\xi_3)\widehat{f_3}(\xi_3)\,dv d\xi_3\Big\|_{L^2_\xi}\\
&\lesssim \rho^{-6}2^{6(\widetilde{a}_++\widetilde{c}_+)}\big[2^{-10K}\Vert f_1\Vert_{L^2}\Vert f_2\Vert_{L^2}\Vert f_3\Vert_{L^2}+K\|f_1\|_{L^{p_1}}\|f_2\|_{L^{p_2}}\|f_3\|_{L^{p_3}}\big],
\end{split}
\end{equation}
where $\widetilde{\Psi}$ satisfies the bounds
\begin{equation}\label{nh93}
\inf_{\xi,v,\xi_3\in\mathbb{R}^2}|\tilde{\Psi}(\xi,v,\xi_3)|\gtrsim \rho,\qquad\sup_{|\alpha|\in[1,20]}\sup_{\xi,v,\xi_3\in\mathbb{R}^2}|D^\alpha_{\xi,v,\xi_3}\tilde\Psi(\xi,v,\xi_3)|\lesssim 1.
\end{equation}
As in the proof of Lemma \ref{KerLem} we insert cutoff functions in $v$. Estimating the $L^2$ norms in the Fourier space and recalling the bound $|Q(\xi,v,\xi_3)|\lesssim 2^{3c_+}$, it follows that
\begin{equation}\label{nh94}
\begin{split}
\sum_{n\leq -10K}\Big\|\int_{\mathbb{R}^2\times\mathbb{R}^2}\frac{Q(\xi,v,\xi_3)}{\tilde{\Psi}(\xi,v,\xi_3)}\varphi_n(v)\widehat{f_1}(\xi-v)\widehat{f_2}(v-\xi_3)\widehat{f_3}(\xi_3)\,dv d\xi_3\Big\|_{L^2_\xi}\lesssim \sum_{n\leq-10K}\rho^{-1}2^{3c_+}2^{2n}\\
\lesssim 2^{3c_+}2^{-10K}.
\end{split}
\end{equation}

To estimate the remaining sum we pass to the physical space and estimate
\begin{equation}\label{nh95}
\begin{split}
&\Big\|\int_{\mathbb{R}^2\times\mathbb{R}^2}\frac{Q(\xi,v,\xi_3)}{\tilde{\Psi}(\xi,v,\xi_3)}\varphi_n(v)\widehat{f_1}(\xi-v)\widehat{f_2}(v-\xi_3)\widehat{f_3}(\xi_3)\,dv d\xi_3\Big\|_{L^2_\xi}\\
&\lesssim\sup_{\|g\|_{L^2}\leq 1}\Big|\int_{\mathbb{R}^2\times\mathbb{R}^2\times\mathbb{R}^2\times\mathbb{R}^2}g(x)f_1(y_1)f_2(y_2)f_3(y_3)R_{n}(x;y_1,y_2,y_3)\,dxdy_1dy_2dy_3\Big|,
\end{split}
\end{equation}
where
\begin{equation*}
R_{n}(x;y_1,y_2,y_3):=\int_{\mathbb{R}^2\times\mathbb{R}^2\times\mathbb{R}^2}e^{i(x-y_1)\cdot\xi}e^{i(y_1-y_2)\cdot v}e^{i(y_2-y_3)\cdot \xi_3}\varphi_n(v)\frac{Q(\xi,v,\xi_3)}{\tilde{\Psi}(\xi,v,\xi_3)}\,d\xi dv d\xi_3.
\end{equation*}
By integration by parts it is easy to see that
\begin{equation*}
|R_{n}(x;y_1,y_2,y_3)|\lesssim |K_1(x-y_1)|\cdot |K_2(y_1-y_2)|\cdot |K_3(y_2-y_3)|
\end{equation*}
for some functions $K_1,K_2,K_3:\mathbb{R}^2\to[0,\infty)$ satisfying
\begin{equation*}
\|K_1\|_{L^1}\cdot\|K_2\|_{L^1}\cdot \|K_3\|_{L^1}\lesssim \rho^{-6}2^{6(\widetilde{a}_++\widetilde{c}_+)}.
\end{equation*}
Using \eqref{nh95}, it follows that
\begin{equation}\label{nh96}
\Big\|\int_{\mathbb{R}^2\times\mathbb{R}^2}\frac{Q(\xi,v,\xi_3)}{\tilde{\Psi}(\xi,v,\xi_3)}\varphi_n(v)\widehat{f_1}(\xi-v)\widehat{f_2}(v-\xi_3)\widehat{f_3}(\xi_3)\,dv d\xi_3\Big\|_{L^2_\xi}\lesssim \rho^{-6}2^{6(\widetilde{a}_++\widetilde{c}_+)}\|f_1\|_{L^{p_1}}\|f_2\|_{L^{p_2}}\|f_3\|_{L^{p_3}}.
\end{equation}
The desired bound \eqref{nh92} follows from \eqref{nh94} and \eqref{nh96}.
\end{proof}


\begin{thebibliography}{}

\bibitem{Bit} J. A. Bittencourt Fundamentals of plasma physics, 3rd edition, 2004, Springer ISBN-13: 978-1441919304.

\bibitem{BoSm} J. L. Bona and R. Smith, The initial-value problem for the Korteweg-de Vries equation, {\it Philos. Trans. R. Soc. Lond.}, Ser. A {\bf{278}} (1975), 555--601.

\bibitem{GerMas} P. Germain and N. Masmoudi, Global existence for the Euler-Maxwell system, preprint, arXiv:1107.1595.

\bibitem{GerMasSha} P. Germain, N. Masmoudi and J. Shatah, Global
solutions for 3D quadratic Schr\"odinger equations, \textit{Int. Math. Res.
Not.}, 2009, no. 3, 414--432.

\bibitem{GerMasSha2} P. Germain, N. Masmoudi and J. Shatah, Global solutions for the gravity water waves
equation in dimension 3, \textit{Ann. Math.} to appear.

\bibitem{Guo} Y. Guo, Smooth irrotational Flows in the large to the
Euler-Poisson system in $R^{3+1}$ \textit{Comm. Math. Phys.} 195, (1998),
249--265.

\bibitem{GuoPau} Y. Guo and B. Pausader, Global Smooth Ion Dynamics in the Euler-Poisson System,
\textit{Comm. Math. Phys.} 303 (2011), 89-125.

\bibitem{GuoTad} Y. Guo, and A. S. Tahvildar-Zadeh, Formation of singularities
in relativistic fluid dynamics and in spherically symmetric plasma dynamics.
Nonlinear partial differential equations (Evanston, IL, 1998), 151--161, 
\textit{Contemp. Math.}, 238, Amer. Math. Soc., Providence, RI, 1999.

\bibitem{GNT2} S. Gustafson, K. Nakanishi and T.P. Tsai, Global dispersive
solutions for the Gross-Pitaevskii equation in two and three dimensions. 
\textit{Ann. IHP} 8 (2007), no. 7, 1303--1331.

\bibitem{GNT} S. Gustafson, K. Nakanishi and T.P. Tsai, Scattering theory for the Gross-Pitaevskii equation
in three dimensions. \textit{Comm. Contemp. Math.} 11 (2009), no. 4,
657--707.

\bibitem{Juhi} J. Jang, The 2D Euler-Poisson System with Spherical Symmetry. arXiv:1109.2643.

\bibitem{JangLiZhang} J. Jang, D. Li, X. Zhang. Smooth global solutions for the two dimensional Euler-Poisson
system, arXiv:1109.3882.

\bibitem{Kla} S. Klainerman, Global existence of small amplitude solutions to nonlinear Klein-Gordon
equations in four space-time dimensions. {\it Comm. Pure. Appl. Math.} 38, 631–641 (1985).

\bibitem{Sha} J. Shatah, Normal forms and quadratic nonlinear Klein-Gordon
equations. \textit{Comm. Pure Appl. Math.} \textbf{38} (1985), No 5,
685--696.

\bibitem{Sid} Sideris, T. Formation of singularities in three-dimensional
compressible fluids. \textit{Comm. Math. Phys.} 101, (1985), 475--485.

\bibitem{WTB} D. Wei, E. Tadmore and H. Bae Critical Thresholds in multi-dimensional Euler-Poisson equations with radial symmetry, {\it Comm. Math. Sci. 
c 2011 International Press}
Vol. 9, No. 1.


\end{thebibliography}
\end{document}